\documentclass[smallextended,envcountsect]{svjour3}
\usepackage{amsmath}
\usepackage{amssymb}
\usepackage{mathptmx} %Times New Roman
\usepackage{cite}
\usepackage[left=4.5cm, right=4.5cm, top=3.0cm, bottom=3.0cm]{geometry}
\usepackage{hyperref}
\hypersetup{
colorlinks=true,
linkcolor=blue,
citecolor=red,
urlcolor=blue}
\usepackage[weather,alpine,misc,geometry]{ifsym}
\usepackage[latin1]{inputenc}
\smartqed

\newcommand{\al}{\alpha}
\newcommand{\be}{\beta}
\newcommand{\ga}{\gamma}
\newcommand{\la}{\lambda}
\newcommand{\de}{\delta}
\newcommand{\eps}{\varepsilon}
\newcommand{\bx}{\bar x}
\newcommand{\by}{\bar y}

\newcommand{\iv}{^{-1} }

\newcommand {\R} {\mathbb R}
\newcommand {\N} {\mathbb N}

\newcommand {\B} {\mathbb B}

\newcommand {\gph} {{\rm gph}\,}%Graph
\newcommand {\dom} {{\rm dom}\,}
\newcommand {\epi} {{\rm epi}\,}

\newcommand {\bd} {{\rm bd}\,}

\newcommand {\Int} {{\rm int}\,}

\renewcommand{\iff}{$ \Leftrightarrow\ $}%iff
\newcommand{\folgt}{$ \Rightarrow\ $}

\def\nbh{neighbourhood}
\def\es{\emptyset}
\def\lsc{lower semicontinuous}

\def\RHS{right-hand side}
\def\SVM{set-valued mapping}
\def\EVP{Ekeland variational principle}

\newcommand{\norm}[1]{\left\Vert#1\right\Vert}

\newcommand{\red}[1]{\textcolor{red}{#1}}

\newcommand{\qdtx}[1]{\quad\mbox{#1}\quad}
\newcommand{\AND}{\quad\mbox{and}\quad}
\newcounter{mycount}

\newcommand{\AK}[1]{\todo[inline]{AK {#1}}}

\begin{document}
\title{Transversality Properties: Primal Sufficient Conditions
\thanks{The research was supported by the Australian Research Council, project DP160100854.
The second author benefited from the support of the FMJH Program PGMO and from the support of EDF.}
}
\author{Nguyen Duy Cuong \and
Alexander Y. Kruger}
\institute{
Nguyen Duy Cuong 	
\\
Centre for Informatics and Applied Optimization, School of Science, Engineering and Information Technology, Federation University Australia, POB 663, Ballarat, Vic, 3350, Australia\\
Department of Mathematics, College of Natural Sciences, Can Tho University, Can Tho City, Vietnam\\
Email: 
duynguyen@students.federation.edu.au,
{ndcuong@ctu.edu.vn}\smallskip\\
Alexander Y. Kruger (\Letter)
\\
Centre for Informatics and Applied Optimization, School of Science, Engineering and Information Technology, Federation University Australia, POB 663, Ballarat, Vic, 3350, Australia\\
Email: a.kruger@federation.edu.au, 	
}
\date{Received: date / Accepted: date}
\maketitle

\begin{abstract}
The paper studies `good arrangements' (transversality properties)
of collections of sets
in a normed vector space
near a given point in their intersection.
We
%establish
target
primal (metric and slope) characterizations of transversality properties
in the nonlinear setting.
The H\"older case is given a special attention.
Our main objective is not formally extending our earlier results from the H\"older to a more general nonlinear setting, but rather to develop a general framework for quantitative analysis of transversality properties.
The nonlinearity is just a simple setting, which allows us to unify the existing results on the topic.
Unlike the well-studied subtransversality property, not many characterizations of the other two important properties: se\-mitransversality and transversality have been known even in the linear case.
Quantitative relations between nonlinear transversality properties and the corresponding  regularity properties of \SVM s as well as
nonlinear extensions of the new transversality properties of a \SVM\ to a set in the range space due to Ioffe are also discussed.
\end{abstract}



\setcounter{tocdepth}{2}
%\tableofcontents
\section{Introduction}\label{S1}

\if{
{\bf Dedication.}
The paper is dedicated to Professor Marco
A. L{\'o}pez, on the occasion of his 70th
birthday.
\medskip
}\fi

This paper continues a series of publications by the authors \cite{BuiCuoKru,CuoKru4,CuoKru5,CuoKru6,ThaBuiCuoVer20,
Kru09,Kru06,Kru18, KruTha13,Kru05,KruTha14,KruTha15,KruTha16, KruLukTha17,KruLukTha18,KruLop12.1,KruLop12.2} dedicated to studying `good arrangements' of collections of sets in normed spaces near a point in their intersection.
Following Ioffe \cite{Iof17}, such arrangements are now commonly referred to as \emph{transversality} properties.
Here we refer to transversality broadly as a group of `good arrangement' properties, which includes \emph{semitransversality}, \emph{subtransversality}, \emph{transversality} (a specific property) and some others.
The term \emph{regularity} was extensively used for the same purpose in the earlier publications by the second author, and is still preferred by many authors.

Transversality (regularity) properties of collections of sets play
an important role in optimization and variational analysis, e.g., as constraint qualifications, qualification conditions in subdifferential, normal cone and coderivative calculus, and convergence analysis of computational algorithms.
Significant efforts have been invested into studying this class of properties and establishing their primal and dual necessary and/or sufficient characterizations in various settings (convex and nonconvex, finite and infinite dimensional, finite and infinite collections of sets).
In addition to the references provided above, we refer the readers to \cite{BauBor96,Iof00,LewLukMal09,NgaThe01, NgZan07,ZheNg08,BorLiTam17,BorLiYao14,DruLiWol17,ZheWeiYao10, HesLuk13,DruIofLew15,NolRon16, BakDeuLi05,BauBorLi99,Pen13} {for results and historical comments}.

Our aim is to develop a general framework for quantitative analysis of transversality properties of collections of sets.
In this paper, we focus on primal space conditions, and
establish metric characterizations and slope-type sufficient conditions for three closely related general nonlinear transversality properties: \emph{$\varphi-$se\-mitransversali\-ty}, \emph{$\varphi-$subtransversality} and \emph{$\varphi-$transversality}.
\sloppy

The slope sufficient conditions stem from applying the \EVP\ to the definitions of the respective properties; the proofs are rather straightforward.
This type of conditions are often considered as just a first step on the way to producing more involved dual ({subdifferential and} normal cone) conditions, and the primal sufficient conditions remain hidden in the proofs.
%We are also using the slope sufficient conditions formulated in the current paper in our subsequent \red{papers \cite{CuoKru4,CuoKru5}} dedicated to dual characterizations.
%At the same time,
We believe that primal conditions
%themselves
(being in a sense analogues of very popular slope conditions for error bounds) can be of importance for
%some
applications.
Moreover, subdividing the conventional regularity/transversality theory into the primal and dual parts clarifies the roles of the main tools employed within the theory: the \EVP\ in the primal part and subdifferential sum rules in the dual part.
As a result, the proofs in both the primal and the `more involved' dual parts become straightforward.
This observation goes beyond the transversality of collections of sets and applies also to the regularity of \SVM s and the {error bound} theory.

Unlike the earlier publications, here, besides estimates for the transversality moduli, we provide also quantitative estimates for the parameters $\de$'s involved in the definitions; cf. Definitions~\ref{D0} and \ref{D1}.
This can be of importance from the computational point of view.
We also examine quantitative relations between the nonlinear transversality properties of collections of sets and the corresponding nonlinear regularity properties of \SVM s as well as nonlinear extensions of the new \emph{transversality properties of a \SVM\ to a set in the range space} due to Ioffe.

We would like to emphasize that our main objective is not formally extending our earlier results from the H\"older to a more general nonlinear setting, but rather to develop a comprehensive theory of transversality.
The nonlinearity is just a simple setting, which allows us to unify the existing (and hopefully also future) results on the topic.
In fact, unlike the subtransversality property which has been well studied in the linear and H\"older settings (see, for instance, \cite{BauBorLi99,Iof00,BakDeuLi05, NgZan07,ZheNg08,HesLuk13, KruTha14,DruIofLew15,KruLukTha17}), for the other two properties: se\-mitransversality and transversality not many characterizations have been known even in the linear case; we fill this gap in the current paper.

Besides the conventional H\"older case, which is given a special attention in the paper, our general model covers also so called \emph{H\"older-type} settings \cite{BolNguPeySut17,Li13} that have recently come into play in the closely related error bound theory due to their importance for applications.
Such nonlinear settings of transversality properties have not been studied before.
Some characterizations are new even in the linear setting.

Apart from being of interest on their own, the slope sufficient conditions for nonlinear transversality properties established in this paper lay the foundation for the dual sufficient conditions for the respective properties in Banach and Asplund spaces in %the subsequent paper
\cite{CuoKru5}.
Primal and dual necessary conditions for the nonlinear transversality properties are
%going to appear
studied
in \cite{CuoKru4,CuoKru6}.

%It is well known that
There exist strong connections between transversality properties of collections of sets and the corresponding regularity properties of \SVM s.
In this paper, we establish quantitative relations between the two models in the general nonlinear setting.
Nonlinear regularity properties of \SVM s
and closely related error bound properties of (extended-)real-valued functions
have been intensively studied since 1980s; cf. \cite{FraQui12,GayGeoJea11,Kum09,LiMor12, BorZhu88,Fra87,Iof13,Kru16,Kru16.2,KruTha14,OuyZhaZhu19, ZheZhu16, AzeCor17,AzeCor14,CorMot08,YaoZhe16}.
%Similarly, strong connections exist between subtransversality and error bounds.
%Nonlinear error bounds have been studied in the last decade by Az\'e \& Corvellec \cite{AzeCor17,AzeCor14,CorMot08} and other authors; cf. \cite{YaoZhe16}.
The slope sufficient conditions for $\varphi-$sub\-trans\-versality in Section~\ref{S4} can be interpreted in terms of the corresponding conditions for nonlinear error bounds.
The semitransversality and transversality properties do not have exact counterparts within the conventional error bound theory.
%even in the linear case.

As in most of our previous publications on the topic, our working model in this paper is a collection of $n\ge2$ arbitrary subsets $\Omega_1,\ldots,\Omega_n$ of a normed vector space $X$, having a common point $\bx\in\cap_{i=1}^n\Omega_i$.
The next definition introduces three most common Hölder transversality properties.
It is a modification of \cite[Definition~1]{KruTha14}.

\begin{definition}\label{D0}
Let
%$\Omega_1,\ldots,\Omega_n$ be subsets of a normed space $X$, $\bx\in\cap_{i=1}^n\Omega_i$,
$\alpha>0$ and $q>0$.
{The collection} $\{\Omega_1,\ldots,\Omega_n\}$ is
\begin{enumerate}
\item
$\al-$semitransversal of order $q$ at $\bx$ if there exists a $\delta>0$ such that
\begin{align}\label{D0-1}
\bigcap_{i=1}^{n}(\Omega_i-x_i)\cap B_\rho(\bar{x})\ne\emptyset
\end{align}
for all $\rho\in]0,\delta[$ and
$x_i\in{X}$ {$(i=1,\ldots,n)$ with $\max_{1\le i\le n}\|x_i\|^q<\al\rho$};

\item
$\al-$subtransversal of order $q$ at $\bx$ if there exist $\delta_1>0$ and $\delta_2>0$ such that
\begin{align}\label{D0-2}
\bigcap_{i=1}^n\Omega_i\cap B_\rho(x)\ne\es
\end{align}
for all $\rho\in]0,\de_1[$ and $x\in B_{\de_2}(\bx)$ with
{$\max_{1\le i\le n}d^q(x,\Omega_i)<\al\rho$};

\item
$\al-$transversal of order $q$ at $\bx$ if there exist $\delta_1>0$ and $\delta_2>0$ such that
\begin{align}\label{D0-3}
\bigcap_{i=1}^n(\Omega_i-\omega_i-x_i)\cap(\rho\B)\neq\emptyset
\end{align}
for all $\rho\in]0,\delta_1[$, $\omega_i\in\Omega_i\cap B_{\delta_2}(\bar{x})$ and $x_i\in{X}$ $(i=1,\ldots,n)$ with $\max_{1\le i\le n}\|x_i\|^q<\al\rho$.
\end{enumerate}
\end{definition}

The three properties in the above definition were referred to in \cite{KruTha14} as $[q]-$semi\-regularity, $[q]-$subregularity and $[q]-$regularity, respectively.
{Property} (ii) was defined in \cite{KruTha14} in a slightly different but equivalent way, under an additional
%; cf. Proposition~\ref{pro1}\red{(i).}
assumption
%in\cite[Definition~1]{KruTha14}
that $q\le1$.
When {$\cap_{i=1}^n\Omega_i$ is closed and} $\bx\in\bd\cap_{i=1}^n\Omega_i$, the condition $q\le1$ is indeed necessary for the $\al-$subtransversality and $\al-$transversality properties; see Remark~\ref{R2}.
At the same time, as observed in \cite{KruTha14}, the property of $\al-$semitransversality can be meaningful with any positive $q$ (and any positive $\al$); see Example~\ref{E1}.

With $q=1$ (linear case), properties (i) and (iii) in Definition~\ref{D0} were discussed in \cite{Kru06} (see also \cite[Properties~(R)$_S$ and (UR)$_S$]{Kru09}), {while property (ii) first appeared}
%(in the form similar to condition \eqref{D1-2})
in \cite{KruTha15}.
If {$\cap_{i=1}^n\Omega_i$ is closed and} $\bx\in\bd\cap_{i=1}^n\Omega_i$, then one can observe that properties (ii) and (iii) can only hold with $\al\le1$; see Remark~\ref{R2}.
If $q=1$, when referring to the three properties in the above definition, we
%drop the mentioning of the order and
talk simply about \emph{$\al-$(semi-/sub-) trans\-versality}.

If a collection $\{\Omega_1,\ldots,\Omega_n\}$ is $\al-$semitransversal (respectively, $\al-$subtransversal or $\al-$trans\-ver\-sal) of order $q$ at $\bx$ with some $\al>0$ and $\delta>0$ (or $\delta_1>0$ and $\delta_2>0$), we often simply say that $\{\Omega_1,\ldots,\Omega_n\}$ is \emph{semitransversal} (respectively, \emph{subtransversal} or \emph{transversal}) of order $q$ at $\bx$.
The number $\al$ characterizes the corresponding property quantitatively.
The exact upper bound of all $\al>0$ such that the property holds with some $\delta>0$ (or $\delta_1>0$ and $\delta_2>0$) is called the \emph{modulus} of this property.
We use the notations s$_{\rm e}$tr$_q[\Omega_1,\ldots,\Omega_n](\bx)$, str$_q[\Omega_1,\ldots,\Omega_n](\bx)$ and tr$_q[\Omega_1,\ldots,\Omega_n](\bx)$ for the moduli of the respective properties.
If the property does not hold, then by convention the respective modulus equals 0.

If $q<1$, the H\"older transversality properties in Definition~\ref{D0} are obviously weaker than the corresponding conventional linear properties and can be satisfied for collections of sets when the conventional ones fail.
This can happen in many natural situations (see examples in \cite[Section~2.3]{KruTha14}), which explains the growing interest of researchers to studying the more subtle nonlinear transversality properties.

Our basic notation is standard, see, e.g., \cite{RocWet98,Mor06.1,DonRoc14}.
Throughout the paper, $X$ and $Y$ are either metric or, more often, normed vector spaces.
The open unit ball in any space is denoted by $\B$, and $B_\delta(x)$ stands for the open ball with center $x$ and radius $\delta>0$.
If not explicitly stated otherwise, products of normed vector spaces are assumed to be equipped with the maximum norm $\|(x,y)\|:=\max\{\|x\|,\|y\|\}$, $(x,y)\in X\times Y$.
%For brevity, we often write $\|x,y\|$ instead of $\|(x,y)\|$ and use similar conventions regarding other expressions involving vectors in product spaces.
The symbols $\R$ and $\R_+$ denote the real line (with the usual norm) and the set of all nonnegative real numbers, respectively.

%AK30/12/18
Given a set $\Omega$,
its interior and boundary
are denoted by $\Int\Omega$ and $\bd\Omega$, respectively.
The distance from a point $x$ to $\Omega$ is defined by $d(x,\Omega):=\inf_{u \in \Omega}\|u-x\|$, and we use the convention $d(x,\emptyset) = +\infty$.
The indicator function of $\Omega$ is defined as follows: $i_\Omega(x)=0$ if $x\in \Omega$ and $i_\Omega(x)=+\infty$ if $x\notin \Omega$.

For an extended-real-valued function $f:X\to\R\cup\{+\infty\}$, %on a normed vector space $X$,
%AK15/09/18.
%we define the domain and the epigraph of $f$,
its domain and epigraph are defined,
respectively, by $\dom f:=\{x \in X\mid f(x) < +\infty\}$ and
$\epi f:=\{(x,\alpha) \in X \times \mathbb{R}\mid f(x) \le \alpha\}$.
The inverse of $f$ (if it exists) is denoted by $f\iv$.
%Throughout the paper, we employ the conventional definitions of lower and upper limits:
%$$
%\liminf_{u\to x}f(u) :=\sup_{\eps>0}\inf_{0<\|u-x\|<\eps}f(u)
%\AND
%\limsup_{u\to x}f(u) :=\inf_{\eps>0}\sup_{0<\|u-x\|<\eps}f(u).
%$$
%Note that both definitions exclude the reference point $x$ when computing the respective $\inf$ and $\sup$.
A set-valued mapping $F:X\rightrightarrows Y$ between two sets $X$ and $Y$ is a mapping, which assigns to every $x\in X$ a subset (possibly empty) $F(x)$ of $Y$.
We use the notations $\gph F:=\{(x,y)\in X\times Y\mid
y\in F(x)\}$ and $\dom\:F:=\{x\in X\mid F(x)\ne\emptyset\}$
for the graph and the domain of $F$, respectively, and $F^{-1}: Y\rightrightarrows X$ for the inverse of $F$.
This inverse (which always exists with possibly empty values at some $y$) is defined by $F^{-1}(y):=\{x\in X \mid y\in F(x)\}$, $y\in Y$.
Obviously ${\dom F^{-1}=F(X)}$.

The closed and open intervals between points $x_1$ and $x_2$ in a normed space are defined, respectively, by
\begin{align*}
[x_1,x_2]:=\{tx_1+(1-t)x_2\mid t\in[0,1]\},
\quad
]x_1,x_2[:=\{tx_1+(1-t)x_2\mid t\in]0,1[\}.
\end{align*}
{The semi-open intervals} $]x_1,x_2]$ and $[x_1,x_2[$ are defined in a similar way.

The key tool in the proofs of the main results is the celebrated Ekeland variational principle; cf. \cite{Mor06.1,DonRoc14,Pen13,Iof17}.

\begin{lemma}\label{EVP}
Suppose $X$ is a complete metric space, $f:X\to\mathbb{R}\cup\{+\infty\}$ is lower semicontinuous,
$x\in X$, $\varepsilon>0$ and $\lambda>0$.
If
\begin{align*}
f(x)<\inf_{X} f+\varepsilon,
\end{align*}
then there exists an $\hat x\in X$ such that
\begin{enumerate}
\item
$d(\hat{x},x)<\lambda$;
\item
$f(\hat{x})\le f(x)$;
\item
$f(u)+(\varepsilon/\lambda)d(u,\hat{x})\ge f(\hat{x})$ for all $u\in X.$
\end{enumerate}
\end{lemma}

The \emph{slope} \cite{DegMarTos80} and \emph{nonlocal slope} \cite{NgaThe08,Kru15} of a function $f:X\rightarrow \R\cup\{+\infty\}$ on a metric space at $x\in\dom f$ are defined, respectively, by
\begin{align*}
|\nabla f|(x):=\limsup_{u\rightarrow x,\,u\ne x}\dfrac{ [f(x)-f(u)]_+}{d(x,u)}
,\quad
|\nabla f|^\diamond(x):=\sup\limits_{u\ne x}
\dfrac{[f(x)-f_+(u)]_+}{d(x,u)},
\end{align*}
where $\al_+:=\max\{0,\al\}$ for any $\al\in\R$.
The limit $|\nabla f|(x)$ provides the rate of steepest descent
of $f$ at $x$.
If $X$ is a normed space, and $f$ is Fr\'echet differentiable at $x$, then $|\nabla f|(x)=\|f'(x)\|$.
When $x\notin\dom f$, we set $|\nabla f|(x)=|\nabla f|^\diamond(x):=+\infty$.
The next proposition is straightforward.

\begin{proposition}\label{P1.1}
Suppose $X$ is a metric space, $f:X\to\mathbb{R}\cup\{+\infty\}$, and $x\in X$.
\begin{enumerate}
\item
If $f$ is not lower semicontinuous at $x$, then $|\nabla f|(x)=+\infty$.
\item
If $f(x)>0$, then
$|\nabla f|(x)\le |\nabla f|^\diamond(x)$.
\end{enumerate}
\end{proposition}

%The next lemma provides a chain rule for slopes, which is used in Section~\ref{S3}.

When proving primal and dual characterizations of transversality properties in the nonlinear setting we use chain rules for slopes and subdifferentials, respectively.
The next lemma provides a chain rule for slopes, which is used in Section~\ref{S3}.
For its subdifferential counterparts we refer the reader to \cite[Proposition~2.1]{CuoKru5}.

\begin{lemma}\label{L2}
Let $X$ be a metric space, $f:X\rightarrow\R\cup\{+\infty\}$,
$\varphi:\R\rightarrow\R\cup\{+\infty\}$,
${x\in\dom f}$ and $f(x)\in\dom\varphi$.
Suppose $\varphi$ is {nondecreasing} on $\R$ and differentiable at $f(x)$ {with $\varphi'(f(x))>0$}.
Then
$|\nabla(\varphi\circ f)|(x)=\varphi'(f(x))|\nabla f|(x)$.
\end{lemma}

\begin{proof}
If $x$ is a local minimum of $f$, then, thanks to the monotonicity of $\varphi$, it is also a local minimum of $\varphi \circ f$, and consequently, $|\nabla(\varphi\circ f)|(x)=|\nabla f|(x)=0.$
Suppose $x$ is not a local minimum of $f$.
If $f$ is not \lsc\ at $x$, i.e. $\al:=\lim_{k\to+\infty}f(x_k)<f(x)$ for some sequence $x_k\to x$,
then, in view of the assumptions,
%$\varphi'(f(x))>0$, 
$\varphi$ is strictly increasing near $f(x)$,
and consequently,
%\red{we have}
$\liminf_{k\to+\infty}\varphi(f(x_k))\le\varphi(\al)<\varphi(f(x))$ (with the convention that $\varphi(-\infty)=-\infty$), i.e.
$\varphi\circ f$ is not \lsc\ at $x$; hence, in view of Proposition~\ref{P1.1}(i),
$|\nabla(\varphi\circ f)|(x)=|\nabla f|(x)=+\infty$.
Suppose $f$ is \lsc\ at $x$, i.e.
$\liminf_{u\to x,\,u\ne x}f(u)=f(x)$.
Then, taking into account that $x$ is not a local minimum of $f$,
\begin{align*}
|\nabla(\varphi\circ f)|(x)
&=\limsup_{\substack{u\rightarrow x,\,u \ne x}} \dfrac{\varphi(f(x))-\varphi(f(u))}{d(u,x)}\\
&=\limsup_{\substack{u\rightarrow x,\,u \ne x\\ f(u)<f(x)}} \dfrac{\varphi(f(x))-\varphi(f(u))}{d(u,x)}
=\limsup_{\substack{u\rightarrow x,\,u \ne x\\ f(u)\uparrow f(x)}} \dfrac{\varphi(f(x))-\varphi(f(u))}{d(u,x)}\\
&=\limsup_{\substack{u\rightarrow x,\,u\ne x\\
f(u)\uparrow f(x)}} \left(\dfrac{\varphi(f(x))-\varphi(f(u))}{f(x)-f(u)} \cdot \dfrac{f(x)-f(u)}{d(u,x)}\right)\\
&=\varphi'(f(x))\limsup_{\substack{u\rightarrow x,\,u\ne x}}\dfrac{f(x)-f(u)}{d(u,x)}
=\varphi'(f(x))|\nabla f|(x).
\end{align*}
The proof is complete.
\qed\end{proof}

\begin{remark}\label{R01}
\begin{enumerate}
\item
The slope chain rule in Lemma \ref{L2} is a local result.
Instead of assuming that $\varphi$ is defined on the whole real line, one can assume that $\varphi$ is defined and finite on a closed interval $[\al,\be]$ around the point $f(x)$: ${\al<f(x)<\be}$.
It is sufficient to define the composition $\varphi\circ f$ for $x$ with $f(x)\notin[\al,\be]$ as follows:
$(\varphi{\circ}f)(x):=\varphi(\al)$ if $f(x)<\al$, and
$(\varphi{\circ}f)(x):=\varphi(\be)$ if $f(x)>\be$.
This does not affect the conclusion of the lemma.

\item
Lemma \ref{L2} slightly improves \cite[Lemma~4.1]{AzeCor17}, where $f$ and $\varphi$ are assumed lower semicontinuous and continuously differentiable, respectively.
\end{enumerate}
\end{remark}	

The rest of the paper is organized as follows.
In Section~\ref{S2}, we discuss transversality properties of finite collections of sets in the nonlinear setting.
In Section \ref{S3}, we establish metric characterizations of these properties.
Section \ref{S4} is devoted to slope sufficient conditions for the nonlinear transversality properties.
In Section~\ref{S5}, we discuss quantitative relations between nonlinear transversality of collections of sets and the corresponding nonlinear regularity properties of set-valued mappings, and show that the two popular models are in a sense equivalent in the general nonlinear setting.
As a consequence, we improve some results established in \cite{KruTha14} in the H\"older setting.
We also briefly discuss nonlinear extensions of the new {transversality properties of a \SVM\ to a set in the range space} due to Ioffe \cite{Iof17}.

\section{Definitions and Basic Relations}\label{S2}

The nonlinearity in the {definitions} of the transversality properties is determined by a continuous strictly increasing function $\varphi:\mathbb{R}_+ \rightarrow\mathbb{R}_+$ satisfying $\varphi(0) = 0$ and $\lim_{t\to+\infty}\varphi(t)=+\infty$.
The family of all such functions is denoted by $\mathcal{C}$.
We denote by $\mathcal{C}^1$ the subfamily of functions from $\mathcal{C}$ which are differentiable on $]0,+\infty[$ with $\varphi'(t)>0$ for all $t>0$.
Obviously, if $\varphi\in\mathcal{C}$ ($\varphi\in\mathcal{C}^1$), then $\varphi\iv\in\mathcal{C}$ ($\varphi\iv\in\mathcal{C}^1$).
Observe that,
for any $\alpha > 0$ and $q > 0$,
the function $t\mapsto\alpha t^q$ on $\R_+$ belongs to $\mathcal{C}^1$.
\sloppy

\begin{remark}
For the purposes of the paper, it is sufficient to assume that functions
$\varphi\in\mathcal{C}$ are defined and invertible near 0.	
\end{remark}	

In addition to our standing assumption that
$\Omega_1,\ldots,\Omega_n$ are subsets of a normed space $X$ and $\bx\in\cap_{i=1}^n\Omega_i$,
if not explicitly stated otherwise,
we assume from now on
that ${\varphi\in\mathcal{C}}$.	

\begin{definition}\label{D1}
The collection $\{\Omega_1,\ldots,\Omega_n\}$ is
\begin{enumerate}
\item
$\varphi-$semitransversal at $\bx$ if there exists a $\delta>0$ such that condition \eqref{D0-1} is satisfied
for all $\rho \in ]0,\delta[$ and $x_i\in{X}$ $(i=1,\ldots,n)$ with $\varphi(\max_{1\le i\le n}\|x_i\|)<\rho$;

\item
$\varphi-$subtransversal at $\bx$ if there exist $\delta_1>0$ and $\delta_2>0$ such that condition \eqref{D0-2} is satisfied
for all $\rho\in]0,\de_1[$ and $x\in B_{\de_2}(\bx)$ with {$\varphi(\max_{1\le i\le n}d(x,\Omega_i))<\rho$};

\item
$\varphi-$transversal at $\bx$ if there exist $\delta_1>0$ and $\delta_2>0$ such that condition \eqref{D0-3} is satisfied for all $\rho\in]0,\delta_1[$, $\omega_i\in\Omega_i\cap B_{\delta_2}(\bar{x})$ and $x_i\in{X}$ {$(i=1,\ldots,n)$ with $\varphi(\max_{1\le i\le n}\|x_i\|)<\rho$}.
\end{enumerate}
\end{definition}

Observe that conditions \eqref{D0-1} and \eqref{D0-3} are trivially satisfied when $x_i=0$ $(i=1,\ldots,n)$.
Hence, in parts (i) and (iii) of Definition~\ref{D1} (as well as Definition~\ref{D0}) one can additionally assume that $\max_{1\le i\le n}\|x_i\|>0$.
Similarly, in part (ii) of Definition~\ref{D1} (as well as Definition~\ref{D0}) one can assume that $x\notin\cap_{i=1}^n\Omega_i$.

Each of the properties in Definition~\ref{D1} is determined by a function $\varphi\in\mathcal{C}$, and a number $\de>0$ in item (i) or numbers $\de_1>0$ and $\de_2>0$ in items (ii) and (iii).
The function plays the role of a kind of rate or modulus of the respective property, while the role of the $\de$'s is more technical: they control the size of the interval for the values of $\rho$ and, in the case of $\varphi-$subtransversality and $\varphi-$transversality in parts (ii) and (iii), the size of the \nbh s of $\bx$ involved in the respective definitions.
Of course, if a property is satisfied with some $\de_1>0$ and $\de_2>0$, it is satisfied also with the single $\de:=\min\{\de_1,\de_2\}$ in place of both $\de_1$ and $\de_2$.
Unlike our previous publications on (linear and H\"older) transversality properties, we use in the current paper
two different parameters to emphasise their different roles in the definitions and the corresponding characterizations.
Moreover, we are going to provide quantitative estimates for the values of these parameters.

\if{
\NDC{25/3/19.
Should we mention that this is the first time we use two $\de'$s in the definition of the properties?}
\AK{26/3/19.
In the Introduction?}
}\fi

Given a $\de>0$ in item (i) ($\de_1>0$ and $\de_2>0$ in items (ii) and (iii)), if a property is satisfied for some function $\varphi\in\mathcal{C}$, it is obviously satisfied for any function $\hat\varphi\in\mathcal{C}$ such that $\hat\varphi\iv(t)\le\varphi\iv(t)$ for all $t\in]0,\de[$ ($t\in]0,\de_1[$), or equivalently, $\hat\varphi(t)\ge\varphi(t)$ for all $t\in]0,\varphi\iv(\de)[$ ($t\in]0,\varphi\iv(\de_1)[$).
Thus, it makes sense looking for the smallest function in ${\mathcal{C}}$ (if it exists) ensuring the corresponding property for the given sets.
Observe also that taking a smaller $\de>0$ (smaller $\de_1>0$ and $\de_2>0$) may allow each of the properties to be satisfied with a smaller $\varphi$.
When the exact value of $\de$ ($\de_1$ and $\de_2$) in the definition of the respective property is not important,
%we will drop ``with some $\delta>0$'' and talk about simply $\varphi-$(semi-/sub-)transversality at $\bx$.
%In this case,
it makes sense to look for the smallest function ensuring the corresponding property for some $\delta>0$ ($\de_1$ and~$\de_2$).

The most important realization of the three properties in Definition~\ref{D1} corresponds to the H\"older setting, i.e. $\varphi$ being a power function, given for all $t\ge0$ by $\varphi(t):=\alpha^{-1}t^q$ with some $\alpha>0$ and $q>0$.
In this case, Definition~\ref{D1} reduces to Definition~\ref{D0}.
%(with $\de:=\min\{\de_1,\de_2\}$ in parts (ii) and (iii)).

Another important for applications class of functions is given by the so called \emph{H\"older-type} \cite{BolNguPeySut17,Li13} ones, i.e. functions of the form $t\mapsto\al\iv(t^q+t)$, frequently used in the error bound theory, %\cite{BolNguPeySut17,Li13,Li95,WanPan94,ManShi86},
or more generally, functions ${t\mapsto\al\iv(t^q+\be t)}$ with some ${\al>0}$, $\be>0$ and $q>0$.
Depending on the value of $q$, transversality properties determined by such functions can be approximated by H\"older (if $q<1$) or even linear (if $q\ge1$) ones.

\if{
\NDC{2.11.19
I prefer to put the phrase `The collection' directly after a dot.
In other cases, I think we can drop this phrase for simplicity.
What do you think?}
}\fi

\begin{proposition}\label{P2.3}
Let $\varphi(t):=\al\iv(t^q+\be t)$ with some $\al>0$, $\be>0$ and $q>0$.
If the collection $\{\Omega_1,\ldots,\Omega_n\}$ is $\varphi-$(semi-/sub-)transversal at $\bx$,
then it is $\al'-$(semi-/sub-) transversal of order $q'$ at $\bx$, where:
\begin{enumerate}
\item
if $q<1$, then $q'=q$ and $\al'$ is any number in $]0,\al[$;
\item
if $q=1$, then $q'=1$ and $\al':=\al(1+\be)\iv$;
\item
if $q>1$, then $q'=1$ and $\al'$ is any number in $]0,\al\be\iv[$.
\end{enumerate}
\end{proposition}

\begin{proof}
The assertions follow from Definition~\ref{D0} in view of the following observations:
\begin{enumerate}
\item
if $q<1$ and $\al'\in]0,\al[$, then, for all sufficiently small $t>0$, it holds ${\al'{(1+\be t^{1-q})}<\al}$, and consequently,
$\varphi(t)=\al\iv(1+\be t^{1-q})t^q<(\al')\iv t^q$;
\sloppy

\item
if $q=1$ and $\al'=\al(1+\be)\iv$, then $\varphi(t)=\al\iv(1+\be)t=(\al')\iv t$;

\item
{if $q>1$} and $\al'\in]0,\al\be\iv[$, then, for all sufficiently small  $t>0$, it holds $\al'{(\be\iv t^{q-1}+1)}<\al\be\iv$, and consequently, $\varphi(t)=\al\iv\be(\be\iv t^{q-1}+1)t<(\al')\iv t$.
\qed\end{enumerate}
\end{proof}	

The next two propositions collect some simple facts about the properties in Definition~\ref{D1} and clarify relationships between them.

\begin{proposition}\label{pro1}
\begin{enumerate}
\item
If $\Omega_1=\ldots=\Omega_n$, and there exists a $\de_1> 0$ such that $\varphi(t) \ge t$ for all $t \in ]0,\de_1[$, then $\{\Omega_1,\ldots,\Omega_n\}$ is $\varphi-$subtransversal at
$\bx$ with $\delta_1$ and any $\de_2>0$.
\item
If {$\{\Omega_1,\ldots,\Omega_n\}$} is $\varphi-$transversal at $\bx$ with some $\de_1>0$ and $\de_2>0$, then it is $\varphi-$se\-mi\-transversal at $\bx$
with $\delta_1$ and $\varphi-$subtransversal at $\bx$
with any $\de'_1\in]0,\de_1]$ and $\de'_2>0$ such that $\varphi\iv(\de'_1)+\de'_2\le\de_2$.

\item
If $\bx\in\Int\cap_{i=1}^{n}\Omega_i$, then {$\{\Omega_1,\ldots,\Omega_n\}$} is $\varphi-$transversal at $\bx$ with some $\de_1>0$ and $\de_2>0$.
\end{enumerate}	
\end{proposition}

\if{
\NDC{25/3/19.
I think we should mention that $\{\Omega_1,\ldots,\Omega_n\}$ is $\varphi-$subtransversal with $\delta'_1:=\de_1$ and $\de'_2>0$ such that $\varphi\iv(\de_1)+\de'_2=\de_2$ since I think it make sense to mention just the largest neighborhoods of variables in the definition of the subtransversality and also transversality.}
\AK{26/3/19.
Of course we are looking for the largest neighborhoods.
Observe that taking the largest $\delta'_1$ as you suggest, you automatically get the smallest $\delta'_2$.
Is it fair?
Besides, what if $\varphi\iv(\de_1)>\de_2$?
I am not sure if I did the right thing replacing the equality by the inequality.
We have done it back and force several times already.
Do you prefer the equality?
I am happy to change if you wish, especially because we now do not have to worry about `the existence of the unique solution' as we had to before.
}}\fi

\begin{proof}
\begin{enumerate}
\item
Let $\Omega:=\Omega_1=\ldots=\Omega_n$.
Then condition \eqref{D0-2} becomes $\Omega\cap B_\rho(x)\ne\es$.
This inclusion is trivially satisfied if $\varphi(d(x,\Omega))<\rho$ and $\varphi(\rho)\ge\rho$.
\item
Let $\{\Omega_1,\ldots,\Omega_n\}$ be $\varphi-$transversal at $\bx$
with some $\de_1>0$ and $\de_2>0$.
Since condition \eqref{D0-1} is a particular case of condition \eqref{D0-3} with $\omega_i=\bx$ ($i=1,\ldots,n$), we can conclude that $\{\Omega_1,\ldots,\Omega_n\}$ is $\varphi-$semitransversal at $\bx$
with $\delta_1$.
\smallskip
Let $\de'_1\in]0,\de_1]$ and $\de'_2>0$ be such that $\varphi\iv(\de'_1)+\de'_2\le\de_2$, and let
$\rho\in]0,\de'_1[$ and ${x\in B_{\de'_2}(\bx)}$ with {${\varphi(\max_{1\le i\le n}d(x,\Omega_i))<\rho}$}.
Choose $\omega_i\in\Omega_i$ $(i=1,\ldots,n)$ such that {${\varphi(\max_{1\le i\le n}\|x-\omega_i\|)<\rho}$}.
Then, for any $i=1,\ldots,n$,
\sloppy
\begin{align*}
\|\omega_i-\bx\|\le\|x-\omega_i\|+\|x\|<
\varphi\iv(\rho)+\de'_2<\de_2.
\end{align*}
Set $x_i:=x-\omega_i$ $(i=1,\ldots,n)$.
We have $\rho\in]0,\de_1[$, $\omega_i\in\Omega_i\cap B_{\de_2}(\bx)$ {$(i=1,\ldots,n)$ and ${\varphi(\max_{1\le i\le n}\|x_i\|)<\rho}$}.
By Definition~\ref{D1}(iii), condition \eqref{D0-3} is satisfied. %with $x_i$ and $\omega_i$ $(i=1,\ldots,n)$.
This is equivalent to condition \eqref{D0-2}.
In view of Definition~\ref{D1}(ii), $\{\Omega_1,\ldots,\Omega_n\}$ is $\varphi-$subtransversal at $\bx$ with $\de'_1$ and $\de'_2$.

\item
{Let $\bx \in \Int\cap_{i=1}^{n}\Omega_i$.
Choose} numbers $\de_1>0$ and $\de_2>0$ such that, with ${\de:=\varphi\iv(\de_1)+\de_2}$, it holds
$B_{\de}(\bx)\subset\cap_{i=1}^{n}\Omega_i$.
Then, for all $\omega_i \in \Omega_i \cap B_{\de_2}(\bar{x})$ and $x_i\in X$ {$(i=1,\ldots,n)$ with $\varphi(\max_{1\le i\le n}\|x_i\|)<\de_1$}, it holds $0\in\cap_{i=1}^{n}(\Omega_i-\omega_i-x_i)$, and consequently, condition \eqref{D0-3} is satisfied with any $\rho>0$.
Hence, $\{\Omega_1,\ldots,\Omega_n\}$ is $\varphi-$transversal at $\bx$ with $\de_1$ and $\de_2$.
%\sloppy
\qed\end{enumerate}
\end{proof}

\begin{remark}%\label{R1}
\begin{enumerate}
\item
The inequality $\varphi\iv(\de'_1)+\de'_2\le\de_2$ in Proposition~\ref{pro1}(ii) {and some statements below} can obviously be replaced by the equality $\varphi\iv(\de'_1)+\de'_2=\de_2$ providing in a sense the best estimate for the values of the parameters $\de'_1$ and $\de'_2$.

\item
In the H\"older setting, parts (i) and (iii) of Proposition~\ref{pro1} recapture \cite[Remarks~4 and 3]{KruTha14}, respectively, while part (ii) improves \cite[Remark~1]{KruTha14}.

\item
The nonlinear semitransversality and subtransversality properties are in general independent; see examples in \cite[Section~2.3]{KruTha14} and \cite[Section~3.2]{KruTha15}.
\end{enumerate}
\end{remark}

\begin{proposition}\label{P4}
Let ${\cap_{i=1}^n\Omega_i}$ be closed and $\bx\in\bd{\cap_{i=1}^n\Omega_i}$.
If {$\{\Omega_1,\ldots,\Omega_n\}$} is $\varphi-$sub\-trans\-versal (in particular, if it is $\varphi-$tran\-sversal) at $\bx$ with some $\de_1>0$ and $\de_2>0$, then there exists a $\bar t\in]0,\min\{\de_2,\varphi^{-1}(\de_1)\}[$ such that $\varphi(t)\ge t$ for all $t\in]0,\bar t]$.
\sloppy
\end{proposition}

\begin{proof}
Let $\{\Omega_1,\ldots,\Omega_n\}$ be $\varphi-$sub\-transversal at $\bx$ with some $\de_1>0$ and ${\de_2>0}$.
Choose a point $\hat x\notin\cap_{i=1}^{n}\Omega_i$ such that $\|\hat x-\bx\| <\min\{\varphi^{-1}(\de_1),\de_2\}$ and set ${\bar t:=d(\hat x,\cap_{i=1}^{n}\Omega_i)}$.
Then $\bar t<\min\{\varphi^{-1}(\de_1),\de_2\}$.
Besides, $\bar t>0$ since ${\cap_{i=1}^n\Omega_i}$ is closed.
Thanks to the continuity of the function $d(\cdot,\cap_{i=1}^{n}\Omega_i)$,
for any $t\in]0,\bar t]$ there is an ${x\in]\bx,\hat x]}$
such that $d(x,\cap_{i=1}^{n}\Omega_i)=t$.
We have $\|x-\bx\|\le\|\hat x-\bx\|<\de_2$ and ${\varphi(t)\le\varphi(\bar t)<\de_1}$.
Take a $\rho\in]\varphi(t),\de_1[$.
Then {$\varphi(\max_{1\le i\le n}d(x,\Omega_i))\le\varphi(t)<\rho$}.
By Definition~\ref{D1}(ii), $t=d(x,\cap_{i=1}^{n}\Omega_i)<\rho$, and letting $\rho\downarrow\varphi(t)$, we arrive at $t\le\varphi(t)$.
If $\{\Omega_1,\ldots,\Omega_n\}$ is $\varphi-$transversal at $\bx$, the conclusion follows
in view of Proposition~\ref{pro1}{(ii)}.
%\sloppy
\qed\end{proof}

\begin{remark}\label{R2}
The conditions on $\varphi$ in {Proposition~\ref{P4} in} the H\"older setting can only be satisfied if either $q<1$, or $q=1$ and $\al\le1$.
This reflects the well known fact that the H\"older subtransversality and transversality properties are only meaningful when $q\le1$ and, moreover, the linear case ($q=1$) is only meaningful when $\al\le1$; cf. \cite[p.~705]{KruLukTha17}, \cite[p.~118]{Kru18}.
The extreme case $q=\al=1$ is in a sense singular for subtransversality as in this case Definition~\ref{D0}(ii) yields $d(x,\cap_{i=1}^{n}\Omega_i)=\max_{1\le i\le n} d(x,\Omega_i)$ for all $x$ near~$\bx$.
%This condition is hard to satisfy unless $\Omega_1=\ldots=\Omega_n$; cf. Proposition~\ref{pro1}(i).

In accordance with Proposition~\ref{P4}, the $\varphi-$subtransversality and $\varphi-$transversality properties impose serious restrictions on the function $\varphi$.
This is not the case with the $\varphi-$semi\-transversality property: $\varphi$ can be, e.g., any power function.
\end{remark}

\begin{example}\label{E1}
Let $\R^2$ be {equipped} with the maximum norm, and let $q>0$, $\ga>0$, $\Omega_1:=\left\{(\xi_1,\xi_2)\in\R^2\mid \ga^{\frac{1}{q}}\xi_2+|\xi_1|^{\frac{1}{q}}\ge0\right\}$,
$\Omega_2:=\left\{(\xi_1,\xi_2)\in\R^2\mid \ga^{\frac{1}{q}}\xi_2-|\xi_1|^{\frac{1}{q}}\le0\right\}$
and $\bx:=(0,0)$.
Note that, when $q>1$, the sets $\Omega_1$ and $\Omega_2$ are nonconvex.
We claim that the pair $\{\Omega_1,\Omega_2\}$ is $\varphi-$se\-mi\-trans\-versal at $\bx$ with $\varphi(t):=\ga t^q$ ($t\ge0$).
\end{example}

\begin{proof}
Given an $r>0$, set $x_1:=(0,-r)$ and $x_2:=(0,r)$.
Then $\|x_1\|=\|x_2\|=r$ and $(\pm\ga r^q,0) \in(\Omega_1-x_1)\cap(\Omega_2-x_2)$.
Moreover, it is easy to notice that either $(\ga r^q,0)$ or $(-\ga r^q,0)$ belongs to $(\Omega_1-x_1)\cap(\Omega_2-x_2)$ for any choice of vectors $x_1,x_2\in\R^2$ with $\max\{\|x_1\|,\|x_2\|\}\le r$.
Hence, $(\Omega_1-x_1)\cap(\Omega_2-x_2)\cap B_\rho(\bar{x})\ne\emptyset$ for all such vectors $x_1,x_2\in\R^2$ as long as $\rho>\ga r^q$, and consequently, $\{\Omega_1,\Omega_2\}$ is $\varphi-$semitransversal at~$\bx$.
\qed
\end{proof}

\section{Metric Characterizations}\label{S3}

The three transversality properties are defined in Definition~\ref{D1} geometrically.
We now show that they can be characterized in metric terms.
These metric characterizations can be used as equivalent definitions of the respective properties.
\if{
\NDC{2.1.19
I have merged the three theorems into one to reduce the number of theorems.}
}\fi

\begin{theorem}\label{T4.1}
The collection $\{\Omega_1,\ldots,\Omega_n\}$ is
\begin{enumerate}
\item
$\varphi-$semitransversal at $\bx$ with some $\delta>0$ if and only if
\begin{align}\label{T1-1}
d\left(\bx,\bigcap_{i=1}^n(\Omega_i-x_i)\right) \le\varphi\left(\max_{1\le i\le n}\|x_i\|\right)
\end{align}
for all $x_i\in X$ $(i=1,\ldots,n)$ with $\varphi(\max_{1\le i\le n}\|x_i\|)<\delta$;
\item
$\varphi-$subtransversal at $\bx$ with some $\de_1>0$ and $\de_2>0$ if and only if the following equivalent {conditions hold}:
\begin{enumerate}
\item
for all $x\in B_{\de_2}(\bx)$ with $\varphi\left(\max_{1\le i\le n}d(x,\Omega_i)\right)<\de_1$, it holds
\begin{align}\label{T1-2}
d\left(x,\bigcap_{i=1}^n\Omega_i\right)\le\varphi\left(\max_{1\le i\le n}d(x,\Omega_i)\right);
\end{align}

\item
for all {$x_i\in X$ and $\omega_i\in\Omega_i$
$(i=1,\ldots,n)$ with $\varphi(\max_{1\le i\le n}\|x_i\|)<\de_1$} and ${\omega_1+x_1=\ldots=\omega_n+x_n\in B_{\de_2}(\bx)}$, it holds
\begin{align}\label{T1-5}
d\left(0,\bigcap_{i=1}^n(\Omega_i-\omega_i-x_i)\right)
\le\varphi\left(\max_{1\le i\le n}\|x_i\|\right);
\end{align}
\end{enumerate}

\item
$\varphi-$transversal at $\bx$ with some $\de_1>0$ and $\de_2>0$ if and only if inequality \eqref{T1-5} holds
for all $\omega_i\in \Omega_i\cap B_{\de_2}(\bx)$ and $x_i\in X$ {$(i=1,\ldots,n)$ with
$\varphi(\max_{1\le i\le n}\|x_i\|)<\de_1$.}
\end{enumerate}
\end{theorem}

\begin{proof}
\begin{enumerate}
\item
Let $\{\Omega_1,\ldots,\Omega_n\}$ be $\varphi-$semitransversal at $\bx$ with some $\delta>0$, and let ${x_i\in X}$ {$(i=1,\ldots,n)$ with $\rho_0:=\varphi(\max_{1\le i\le n}\|x_i\|)<\delta$.}
%Thus, $\rho_0<\delta$.
Choose a $\rho\in]\rho_0,\de[$.
By \eqref{D0-1}, $d\left(\bar{x},\cap_{i=1}^{n}(\Omega_i-x_i)\right)<\rho$.
Letting $\rho\downarrow\rho_0$, we arrive at inequality \eqref{T1-1}.
 \smallskip

Conversely, let $\delta>0$
and inequality \eqref{T1-1} hold for all $x_i\in X$ {$(i=1,\ldots,n)$ with $\varphi(\max_{1\le i\le n}\|x_i\|)<\delta$}.
For all $\rho\in]0,\delta[$ and $x_i\in X$ {$(i=1,\ldots,n)$ with $\varphi(\max_{1\le i\le n}\|x_i\|)$ $<\rho$}, we have
%$\varphi(\max_{1\le i\le n}\|x_i\|)<\rho<\delta$.
%It follows from \eqref{T1-1} that
$d\left(\bar{x},\cap_{i=1}^{n}(\Omega_i-x_i)\right)<\rho$, which implies condition \eqref{D0-1}.	
{By Definition~\ref{D1}(i),} $\{\Omega_1,\ldots,\Omega_n\}$ is $\varphi-$semi\-transversal at $\bx$ with $\delta$. \sloppy
\item
We first prove the equivalence between (a) and (b).
Suppose condition (a) is satisfied.
Let $\omega_i\in\Omega_i$ and $x_i\in{X}$
$(i=1,\ldots,n)$ with $\varphi(\max_{1\le i\le n}\|x_i\|)<\de_1$ and $x:=\omega_1+x_1=\ldots=\omega_n+x_n\in B_{\de_2}(\bx)$.
Then
\begin{align*}
\varphi\left(d(x,\Omega_i)\right)=\varphi\left(d(\omega_i+x_i,\Omega_i)\right)
\le\varphi(\|x_i\|)<\de_1
\;\;
(i=1,\ldots,n),
\end{align*}
and consequently, {inequality} \eqref{T1-2} is satisfied.
Hence,
\begin{align*} d\Big(0,\bigcap_{i=1}^n(\Omega_i-\omega_i-x_i)\Big)
&=d\Big(x,\bigcap_{i=1}^n\Omega_i\Big)\\
&\le\varphi\left(\max_{1\le i\le n}d(x,\Omega_i)\right)\le\varphi\Big(\max_{1\le i\le n}\|x_i\|\Big).
\end{align*}
Suppose condition (b) is satisfied.
Let $x\in B_{\de_2}(\bx)$ with ${\varphi\left(\max_{1\le i\le n}d(x,\Omega_i)\right)<\de_1}$.
Choose $\omega_i\in\Omega_i$ {$(i=1,\ldots,n)$ such that $\varphi(\max_{1\le i\le n}\|x-\omega_i\|)<\de_1$} and
set $x'_i:=x-\omega_i$ $(i=1,\ldots,n)$.
Then $x=x'_i+\omega_i\in B_{\de_2}(\bx)$ {$(i=1,\ldots,n)$ and $\varphi(\max_{1\le i\le n}\|x'_i\|)<\de_1$}.
In view of inequality \eqref{T1-5} with $x'_i$ in place of $x_i$ ${(i=1,\ldots,n)}$, we obtain
\begin{align*}%\label{T1-6}
d\Big(x,\bigcap_{i=1}^n\Omega_i\Big)
\le\varphi\Big(\max_{1\le i\le n}\|x-\omega_i\|\Big).
\end{align*}
Taking {infimum} in the \RHS\ over $\omega_i\in\Omega_i$ $(i=1,\ldots,n)$, we arrive at inequality \eqref{T1-2}.

Next we show that $\varphi-$subtransversality is equivalent to condition (a).
Let $\{\Omega_1,\ldots,\Omega_n\}$ be $\varphi-$subtransversal at $\bx$ with some $\de_1>0$ and $\de_2>0$, and let $x\in B_{\de_2}(\bx)$ with {$\rho_0:=\varphi\left(\max_{1\le i\le n}d(x,\Omega_i)\right)<\de_1$.}
%Thus, $\rho_0<\delta_1$.
Choose a $\rho\in]\rho_0,\de_1[$.
By Definition~\ref{D1}(ii),
$\cap_{i=1}^n\Omega_i\cap B_\rho(x)\ne\es$, and consequently,
$d\left(x,\cap_{i=1}^{n}\Omega_i\right)<\rho$.
Letting $\rho\downarrow\rho_0$, we arrive at inequality \eqref{T1-2}.
Conversely, let $\de_1>0$ and $\de_2>0$, and inequality
\eqref{T1-2} hold for all ${x\in B_{\delta_2}(\bar{x})}$ with {$\varphi( \max_{1\le i\le n}d(x,\Omega_i))<\de_1$}.
For any $\rho\in]0,\delta_1[$ and $x\in B_{\delta_2}(\bar{x})$ with ${\varphi(\max_{1\le i\le n}d(x,\Omega_i))<\rho}$, we have
%$\varphi\left(\max_{1\le i\le n} d(x,\Omega_i)\right)<\rho<\delta_1$.
%It follows from \eqref{T1-2} that
$d\left(x,\cap_{i=1}^{n}\Omega_i\right)<\rho,$
which implies condition \eqref{D0-2}.
By Definition~\ref{D1}(ii),
$\{\Omega_1,\ldots,\Omega_n\}$ is $\varphi-$sub\-transversal at $\bx$ with $\delta_1$ and $\delta_2$.	
\sloppy
\item
Let $\{\Omega_1,\ldots,\Omega_n\}$ be $\varphi-$transversal at $\bx$ with some $\delta_1>0$ and ${\de_2>0}$, and let $\omega_i\in\Omega_i\cap B_{\de_2}(\bx)$ and $x_i\in X$ {$(i=1,\ldots,n)$ with $\rho_0:=\varphi(\max_{1\le i\le n}\|x_i\|)<\de_1$}.
%Thus, $\rho_0<\delta_1$.
Choose a $\rho\in]\rho_0,\de_1[$.
{By \eqref{D0-3},}
$d\left(0,\cap_{i=1}^{n}(\Omega_i-\omega_i-x_i)\right)<\rho.$
Letting $\rho\downarrow\rho_0$, we arrive at inequality~\eqref{T1-5}.
\smallskip

Conversely, let $\delta_1>0$ and $\de_2>0$, and inequality \eqref{T1-5} hold for all ${\omega_i\in \Omega_i\cap B_{\de_2}(\bx)}$ and $x_i\in X$ $(i=1,\ldots,n)$ with $\varphi(\max_{1\le i\le n}\|x_i\|)<\de_1$.
For any {$\rho\in]0,\delta_1[$, ${\omega_i\in\Omega_i\cap B_{\de_2}(\bx)}$ and $x_i\in X$ $(i=1,\ldots,n)$ with $\varphi(\max_{1\le i\le n}\|x_i\|)<\rho$}, we have
%$\varphi(\|x_i\|)<\rho<\delta_1$ $(i=1,\ldots,n)$.
%It follows from \eqref{T1-5} that
$d\left(0,\cap_{i=1}^{n}(\Omega_i-\omega_i-x_i)\right)<\rho$,
which is equivalent to condition \eqref{D0-3}.
{By Definition~\ref{D1}(iii),} $\{\Omega_1,\ldots,\Omega_n\}$ is $\varphi-$transversal at $\bx$ with $\delta_1$ and $\de_2$.
\qed\end{enumerate}
\end{proof}	

\begin{example}\label{E3.1}
Let $\R^2$ be equipped with the maximum norm, and let $\Omega_1:=\{(\xi_1,\xi_2)\in\R^2\mid\xi_2\ge 0\}$,
$\Omega_2:=\{(\xi_1,\xi_2)\in\R^2\mid\xi_2\le \xi_1^2\}$ and $\bx:=(0,0)$.
Thus, $\Omega_1\cap\Omega_2=\{(\xi_1,\xi_2)\in\R^2\mid0\le\xi_2\le \xi_1^2\}$, and no shift of the sets can make their intersection empty.
We claim that the pair $\{\Omega_1,\Omega_2\}$ is $\varphi-$semitransversal at $\bx$ with
$\varphi(t):=\sqrt{2t}$ $(t\ge 0)$ and
$\de:=2$.
\end{example}	

\begin{proof}
Observe that, given any $\eps\ge0$, the vertical shifts of the sets determined by $x_{1\eps}:=(0,-\eps)$ and $x_{2\eps}=(0,\eps)$ produce the largest `gap' between them compared to all possible shifts $x_1$ and $x_2$ with $\max\{\|x_1\|,\|x_2\|\}\le\eps$.
Indeed,
\begin{align*}
(\Omega_1-x_{1\eps})\cap (\Omega_2-x_{2\eps}) &=\{(\xi_1,\xi_2)\in\R^2\mid\eps\le\xi_2\le\xi_1^2-\eps\}
\\
&\subset (\Omega_1-x_{1})\cap (\Omega_2-x_{2}),
\end{align*}
as long as $\max\{\|x_1\|,\|x_2\|\}\le\eps$.
Observe also that $(\sqrt{2\eps},\eps)\in(\Omega_1-x_{1\eps})\cap (\Omega_2-x_{2\eps})$.
Hence, for any $x_1,x_2\in\R^2$ with $\eps:=\max\{\|x_1\|,\|x_2\|\}<\varphi\iv(\de)=2$, we have
\begin{align*}
d(\bx,(\Omega_1-x_1)\cap(\Omega_2-x_2)) \le\|(\sqrt{2\eps},\eps)\|=\sqrt{2\eps} =\varphi\left(\max\{\|x_1\|,\|x_2\|\}\right).
\end{align*}
In view of Theorem~\ref{T4.1}(i), $\{\Omega_1,\Omega_2\}$ $\varphi-$semitransversal at $\bx$ with $\delta$.
\qed\end{proof}

\begin{example}\label{E3.2}
Let $\R^2$ be equipped with the maximum norm, and let $\Omega_1:=\{(\xi_1,\xi_2)\in\R^2\mid\xi_2=\xi_1^2\}$,
$\Omega_2:=\{(\xi_1,\xi_2)\in\R^2\mid\xi_2=-\xi_1^2\}$
and $\bx:=(0,0)$.
Thus, $\Omega_1\cap\Omega_2=\{\bx\}$.
We claim that, for any $\ga>1$, the pair $\{\Omega_1,\Omega_2\}$ is $\varphi-$subtransversal at $\bx$ with
$\varphi(t):=\ga\sqrt{t}$ $(t\ge 0)$ and any $\delta_1>0$ and $\de_2>0$ satisfying $\de_2+\frac{1}{2}+\sqrt{\de_2+\frac{1}{4}}<\ga^2$.
\end{example}	

\begin{proof}
Observe that, $d(x,\Omega_1\cap\Omega_2)=\|x\|$ for all $x\in\R^2$ and,
given any $\eps\ge0$ and the corresponding point $x_{\eps}=(0,\eps)$, one has
\begin{align*}%\label{E3.2-2}
\min_{\|x\|=\eps}\max\{d(x,\Omega_1),d(x,\Omega_2)\} =d(x_\eps,\Omega_1)=d(x_\eps,\Omega_2) =\min_{t\ge0}\max\{\eps-t,t^2\}.
\end{align*}
It is easy to see that the minimum in the rightmost minimization problem is attained at $t:=\sqrt{\eps+\frac{1}{4}}-\frac{1}{2}$ satisfying $\eps-t=t^2$.
Thus,
\begin{align*}%\label{E3.2-2}
\min_{\|x\|=\eps}\max\{d(x,\Omega_1),d(x,\Omega_2)\} =\eps+\frac{1}{2}-\sqrt{\eps+\frac{1}{4}} =\frac{\eps^2}{\eps+\frac{1}{2}+\sqrt{\eps+\frac{1}{4}}}.
\end{align*}
Hence, for any $x\in\R^2$ with $\|x\|<\de_2$, we have
\begin{align*}%\label{E3.2-2}
d(x,\Omega_1\cap\Omega_2)=\|x\| \le\frac{\ga\|x\|} {\sqrt{\|x\|+\frac{1}{2}+\sqrt{\|x\|+\frac{1}{4}}}} \le\varphi(\max\{d(x,\Omega_1),d(x,\Omega_2)\}).
\end{align*}
In view of Theorem~\ref{T4.1}(ii), $\{\Omega_1,\Omega_2\}$ is $\varphi-$subtransversal at $\bx$ with $\de_1$ and $\de_2$.
\end{proof}	
\if{
\red{
\begin{remark}
In view of Proposition~\ref{P2.3}(i),
Examples \ref{E3.1} and \ref{E3.2} improve  \cite[Examples~1 and 2]{KruTha14}, respectively.
\end{remark}
}
}\fi

The next statement provides alternative metric characterizations of
$\varphi-$transversa\-lity.
These characterizations differ from the one in {Theorem~\ref{T4.1}(iii)} by values of the parameters $\de_1$ and $\de_2$ and have certain advantages, e.g., when establishing connections with metric regularity of \SVM s.
The relations between the values of the parameters in the two
groups of metric characterizations can be estimated.

\begin{theorem}\label{P3.5}
Let $\de_1>0$ and $\de_2>0$.
The following conditions are equivalent:
\begin{enumerate}
\item
inequality \eqref{T1-5} is satisfied
{for all $x_i\in{X}$ and $\omega_i\in\Omega_i$ with $\omega_{i}+x_i\in B_{\de_2}(\bx)$ ${(i=1,\ldots,n)}$ and $\varphi(\max_{1\le i\le n}\|x_i\|)<\de_1$};
\sloppy

\item
{for all $x_i\in\de_2\B$ $(i=1,\ldots,n)$ with
$\varphi(\max_{1\le i\le n}d(\bx,\Omega_i-x_i))<\de_1$}, it holds
\begin{align}\label{P2.2-1}
d\Big(\bx,\bigcap_{i=1}^n(\Omega_i-x_i)\Big) \le\varphi\Big(\max_{1\le i\le n}d(\bx,\Omega_i-x_i)\Big);
\end{align}

\item
{for all $x,x_i\in{X}$ with $x+x_i\in B_{\de_2}(\bx)$ $(i=1,\ldots,n)$ and $\varphi(\max_{1\le i\le n}d(x,\Omega_i-x_i))<\de_1$}, it holds
\begin{align}\label{P2.2-2}
d\Big(x,\bigcap_{i=1}^n(\Omega_i-x_i)\Big) \le\varphi\Big(\max_{1\le i\le n}d(x,\Omega_i-x_i)\Big).
\end{align}
\end{enumerate}

Moreover, if {$\{\Omega_1,\ldots,\Omega_n\}$} is $\varphi-$transversal at $\bx$ with some $\de_1>0$ and $\de_2>0$, then conditions {\rm (i)--(iii)} hold with any $\de_1'\in ]0,\de_1]$ and $\de_2'>0$ satisfying ${\varphi\iv(\de_1')+\de_2'}\le\de_2$ in place of $\de_1$ and $\de_2$.

Conversely, if conditions {\rm (i)--(iii)} hold with some $\de_1>0$ and $\de_2>0$, then {$\{\Omega_1,\ldots,\Omega_n\}$} is $\varphi-$transversal at $\bx$ with any $\de_1'\in ]0,\de_1]$ and $\de_2'>0$ satisfying ${\varphi\iv(\de_1')+\de_2'}\le\de_2$.
\sloppy
\end{theorem}

\begin{proof}
We first prove the equivalence of conditions (i)--(iii).

{(i) \folgt (ii)}.
Let $x_i\in\de_2\B$ $(i=1,\ldots,n)$ with $\varphi(\max_{1\le i\le n}d(\bx,\Omega_i-x_i))<\de_1$.
Choose $\omega_{i}\in\Omega_i$ $(i=1,\ldots,n)$ such that $\varphi(\max_{1\le i\le n}\|\bx+x_i-\omega_i\|)<\de_1$.
Set $x'_i:=\bx+x_i-\omega_i$ $(i=1,\ldots,n)$.
Then $\omega_i+x'_i\in B_{\de_2}(\bx)$ $(i=1,\ldots,n)$ and $\varphi(\max_{1\le i\le n}\|x'_i\|)<\de_1$.
By (i), inequality \eqref{T1-5} is satisfied with $x'_i$ in place of $x_i$ $(i=1,\ldots,n)$, i.e.
\begin{align*}
d\Big(\bx,\bigcap_{i=1}^n(\Omega_i-x_i)\Big) \le\varphi\Big(\max_{1\le i\le n}\|\bx+x_i-\omega_i\|\Big).
\end{align*}
Taking the infimum in the righ-hand side over $\omega_i\in\Omega_i$ $(i=1,\ldots,n)$, we
arrive at inequality \eqref{P2.2-1}.

(ii) \folgt (iii).
Let $x,x_i\in X$ with $x+x_i\in B_{\de_2}(\bx)$ $(i=1,\ldots,n)$ and $\varphi(\max_{1\le i\le n}d(x,\Omega_i-x_i))<\de_1$.
Set $x'_i:=x+x_i-\bx$ $(i=1,\ldots,n)$.
Then $x'_i\in\de_2\B$ $(i=1,\ldots,n)$ and
{$\varphi(\max_{1\le i\le n}d(\bx,\Omega_i-x'_i))<\de_1$}.
By (ii), inequality \eqref{P2.2-1} is satisfied with $x'_i$ in place of $x_i$ $(i=1,\ldots,n)$.
This is equivalent to inequality \eqref{P2.2-2}.
%\sloppy

{(iii) \folgt (i)}.
Let $x_i\in{X}$ and $\omega_i\in\Omega_i$ with $\omega_i+x_i\in B_{\de_2}(\bx)$ $(i=1,\ldots,n)$ and $\varphi(\max_{1\le i\le n}\|x_i\|)<\de_1$.
Set $x'_i:=\omega_i+x_i-\bx$
	$(i=1,\ldots,n)$.
Then, $\bx+x'_i\in B_{\de_2}(\bx)$ and
$\varphi(\max_{1\le i\le n}d(\bx,\Omega_i-x'_i))\le\varphi(\|x_i\|)<\de_1$.
By (iii), inequality \eqref{P2.2-2} is satisfied with $\bx$ and $x'_i$ in place of $x$ and $x_i$ $(i=1,\ldots,n)$, respectively,
i.e.
\begin{align*}
d\Big(0,\bigcap_{i=1}^n(\Omega_i-\omega_i-x_i)\Big) \le\varphi\Big(\max_{1\le i\le n}d(0,\Omega_i-\omega_i-x_i)\Big).
\end{align*}
Since $\omega_i\in\Omega_i$ $(i=1,\ldots,n)$,
inequality \eqref{T1-5} is satisfied.
		
Suppose $\{\Omega_1,\ldots,\Omega_{n} \}$ is $\varphi-$transversal at $\bx$ with some $\de_1>0$ and $\de_2>0$, and let $\de_1'\in ]0,\de_1]$ and $\de_2'>0$ be such that $\varphi\iv(\de_1')+\de_2'\le \de_2$.
Then, for all $x_i\in X$ and $\omega_i \in\Omega_i$ with $\omega_{i}+x_i\in B_{\de'_2}(\bx)$ $(i=1,\ldots,n)$ and $\varphi(\max_{1\le i\le n}\|x_i\|)<\de_1'$, we have $\|{\omega_{i}-\bx}\|\le\|x_i\|+ \|\omega_{i}+x_i-\bx\| <\varphi\iv(\de_1')+\de_2'\le\de_2$ ($i=1,\ldots,n$).
By {Theorem~\ref{T4.1}(iii)}, inequality \eqref{T1-5} is satisfied, and consequently, condition (i) (as well as conditions (ii) and (iii)) holds with $\de_1'$ and $\de_2'$.

Conversely, suppose conditions {\rm (i)--(iii)} hold with some $\de_1>0$ and $\de_2>0$, and let $\de_1'\in ]0,\de_1]$ and $\de_2'>0$ be such that $\varphi\iv(\de_1')+\de_2'\le\de_2$.
Then, for all ${\omega_i\in\Omega_i \cap B_{\delta'_2}(\bar{x})}$ and $x_i\in X$ $(i=1,\ldots,n)$ with $\varphi(\max_{1\le i\le n}\|x_i\|)<\de_1'$, we have $\|\omega_{i}+x_i-\bx\|\le\|x_i\|+\|\omega_i-\bx\| <\varphi\iv(\de_1')+\de_2'\le\de_2$ $(i=1,\ldots,n)$.
By (i), inequality \eqref{T1-5} is satisfied, and consequently, $\{\Omega_1,\ldots,\Omega_n\}$ is $\varphi-$transversal at $\bx$ with $\de_1'$ and $\de_2'$ according to {Theorem~\ref{T4.1}(iii)}.
\qed\end{proof}	

\begin{remark}
\begin{enumerate}
\item
In the H\"older case, i.e. when $\varphi(t):=\al\iv t^q$ ($t\ge 0$) for some $\al>0$ and $q\in]0,1]$,
condition \eqref{P2.2-2} served as the main metric characterization of transversality; cf. \cite{KruTha15,KruTha14}.
In the linear case, condition \eqref{P2.2-1} has been picked up recently in \cite{BuiKru19,BuiCuoKru}.
This condition seems an important advancement as it replaces an arbitrary point $x$ in \eqref{P2.2-2} with the given reference point $\bx$.
Condition \eqref{T1-5} in part (i) seems new.
In view of {Theorem~\ref{T4.1}(iii)}, it is the most straightforward metric counterpart of the original geometric property \eqref{D0-3}.

\item
The metric characterizations of the three $\varphi-$transversality properties in the above theorems look similar: each of them provides an upper error bound type estimate for the distance from a point to the intersection of sets, which can be useful from the computational point of view.
For the account of nonlinear error bounds theory, we refer the reader to \cite{AzeCor17,AzeCor14,CorMot08,YaoZhe16}.
\end{enumerate}
\end{remark}

{The next corollary provides} qualitative metric characterizations of the three nonlinear trans\-versality properties.
They are direct consequences of {Theorems~\ref{T4.1} and~\ref{P3.5}.}

\begin{corollary}\label{C4.1}
The collection $\{\Omega_1,\ldots,\Omega_n\}$ is
\begin{enumerate}
\item
$\varphi-$semitransversal at $\bx$ if and only if
there exists a $\delta>0$ such that inequality \eqref{T1-1} holds for all $x_i \in\delta\B$ $(i=1,\ldots,n)$;

\item
$\varphi-$subtransversal at $\bx$
if and only if the following equivalent conditions {hold}:
\begin{enumerate}

\item
there exists a $\delta>0$ such that inequality \eqref{T1-2} holds for all $x\in B_{\delta}(\bar{x})$;

\item
there exists a $\delta>0$ such that inequality \eqref{T1-5} holds for all $\omega_i\in\Omega_i\cap B_{\de}(\bx)$ and $x_i\in\de\B$ $(i=1,\ldots,n)$ with $\omega_1+x_1=\ldots=\omega_n+x_n$;
\end{enumerate}

\item
$\varphi-$transversal at $\bx$ if and only if the following equivalent conditions {hold}:
\begin{enumerate}

\item	
there exists a $\delta>0$ such that inequality \eqref{T1-5} holds for all $\omega_i\in\Omega_i\cap B_\de(\bx)$ and $x_i\in\delta\B$ $(i=1,\ldots,n)$;

\item
there exists a $\delta>0$ such that inequality \eqref{P2.2-1} holds for all $x_i\in\delta\B$ $(i=1,\ldots,n)$;
\item
there exists a $\delta>0$ such that inequality \eqref{P2.2-2} holds for all $x\in B_\de(\bx)$ and $x_i\in\delta\B$ $(i=1,\ldots,n)$.
\end{enumerate}
\end{enumerate}
\end{corollary}

\begin{remark}%\label{R3}m
In the H\"older setting, i.e. when
$\varphi(t):=\alpha^{-1} t^q$ $(t\ge0)$ with some ${\alpha>0}$ and $q>0$, the above corollary improves \cite[Theorem~1]{KruTha14}.
In the linear case, the equivalence of the three characterizations of transversality in {Corollary~\ref{C4.1}(iii)} has been established in \cite{BuiCuoKru}.
%The equivalent metric characterizations in Corollaries \ref{C4.1} and \ref{C4.2} seem new even in the linear setting.
We refer the readers to \cite{KruLukTha17,KruLukTha18,Kru18} for more discussions and historical comments.
\end{remark}	

The next two propositions identify important situations when `restricted' versions of the metric characterizations of nonlinear transversality properties in
{Theorem~\ref{T4.1}} can be used: with all but one sets being translated in the cases of $\varphi-$semi\-transversality and $\varphi-$transversality, and with the point $x$ restricted to one of the sets in the case of $\varphi-$subtransversality.
The latter restricted version is of importance, for instance, when dealing with alternating (or cyclic) projections.
The first proposition formulates simplified necessary conditions for the transversality properties which are direct consequences of the respective statements, while the second one gives conditions under which these conditions become sufficient in the case of two sets.

\begin{proposition}
\begin{enumerate}
\item
If $\{\Omega_1,\ldots,\Omega_n\}$ is $\varphi-$semitransversal at $\bx$ with some $\de>0$, then
\begin{gather*}%\label{prop2-3}
d\Big(\bx,\bigcap_{i=1}^{n-1}(\Omega_i-x_i)\cap\Omega_n\Big) \le \varphi\Big(\max_{1\le i\le n-1}\|x_i\|\Big)
\end{gather*}
{for all $x_i\in X$ $(i=1,\ldots,n-1)$ with $\varphi(\max_{1\le i\le n-1}\|x_i\|)<\delta$.}

\item
If $\{\Omega_1,\ldots,\Omega_n\}$ is $\varphi-$subtransversal at $\bx$ with some $\de_1>0$ and $\de_2>0$, then
\begin{gather*}%\label{prop2-1}
d\Big(x,\bigcap_{i=1}^{n} \Omega_i\Big)\le{\varphi\Big(\max_{1\le i\le n-1}d(x,\Omega_i)\Big)}
\end{gather*}
for all $x\in\Omega_n\cap B_{\delta_2}(\bar{x})$ with ${\varphi(\max_{1\le i\le n-1}d(x,\Omega_i))}<\de_1$.

\item
If $\{\Omega_1,\ldots,\Omega_n\}$ is $\varphi-$transversal at $\bx$ with some $\de_1>0$ and $\de_2>0$, then
\begin{align*}
d\Big(0,\bigcap_{i=1}^{n-1}(\Omega_i-\omega_i-x_i) \cap(\Omega_n-\omega_n)\Big)
\le\varphi\Big(\max_{1\le i\le n-1}\|x_i\|\Big)
\end{align*}
for all {$\omega_i\in\Omega_i\cap B_{\de_2}(\bx)$ $(i=1,\ldots,n)$ and
$x_i\in X$ $(i=1,\ldots,n-1)$ with $\varphi(\max_{1\le i\le n-1}\|x_i\|)<\de_1$.}
\end{enumerate}
\end{proposition}

\begin{proposition}\label{P6}
Let $\Omega_1,\Omega_2$ be subsets of a normed space {$X$,} and $\bx\in\Omega_1\cap\Omega_2$.
Let $\al>0$, $\bar t>0$,
$\varphi(t)\le\al t$ for all $t\in]0,\bar t\,]$,
and ${\al':=(1+2\al)\iv}$.
\begin{enumerate}
\item
If
for all $x\in\bar t\,\B$,
\begin{gather}\label{P4.2-1}
d\left(\bx,(\Omega_1-x)\cap\Omega_2\right) \le\varphi(\|x\|),
\end{gather}
then $\{\Omega_1,\Omega_2\}$ is $\al'-$semitransversal at $\bx$ with $\de:=\left(\al+\frac{1}{2}\right)\bar t$.

\item
If there exists a $\de_2>0$ such that,
for all $x\in\Omega_2\cap B_{2\delta_2}(\bar{x})$ with $d(x,\Omega_1)<\bar t$,
\begin{gather}\label{P4.2-2}
d\left(x,\Omega_1\cap\Omega_2\right)\le\varphi(d(x,\Omega_1)),
\end{gather}
then $\{\Omega_1,\Omega_2\}$ is $\al'-$subtransversal at $\bx$ with $\de_1:=\left(\al+\frac{1}{2}\right)\bar t$ and {$\de_2$}.
\item
If there exists a $\de_2>0$ such that,
for all $\omega_i\in\Omega_i\cap B_{\de_2}(\bx)$ $(i=1,2)$ and $x\in\bar t\,\B$,
\begin{gather*}%\label{P4.2-3}
d\left(0,(\Omega_1-\omega_1-x)\cap(\Omega_2-\omega_2)\right)\le\varphi(\|x\|),
\end{gather*}
then $\{\Omega_1,\Omega_2\}$ is $\al'-$transversal at $\bx$ with ${\de_1}:=\left(\al+\frac{1}{2}\right)\bar t$ and $\de_2$.
\end{enumerate}
\end{proposition}

\begin{proof}
\begin{enumerate}
\item
Let $\de:=\left(\al+\frac{1}{2}\right)\bar t$, and {inequality} \eqref{P4.2-1} be satisfied for all $x\in\bar t\,\B$.
Let $\rho\in]0,\delta[$ {and $x_1,x_2\in X$ with} ${\max\{\|x_1\|,\|x_2\|\}<{\al'}\rho}$.
Set $x':=x_1-x_2$.
Thus,
$\|x'\|\le 2\max\{\|x_1\|,\|x_2\|\}<2{\al'}\de=\bar t$.
Hence, by \eqref{P4.2-1} with $x'$ in place of $x$,
\sloppy
\begin{align*}
d(\bx,(\Omega_1-x_1)\cap(\Omega_2-x_2))
&\le \|x_2\|+d(\bx-x_2,(\Omega_1-x_1)\cap(\Omega_2-x_2))\\
&=\|x_2\|+d(\bx,(\Omega_1-x')\cap\Omega_2)\\
&\le\|x_2\|+\varphi(\|x'\|)\le \|x_2\|+\al\|x'\|\\
&\le (1+2\al)\max\{\|x_1\|,\|x_2\|\}<\rho.
\end{align*}
Hence, $(\Omega_1-x_1)\cap(\Omega_2-x_2)\cap B_\rho(\bx){\ne\es}$ and, by Definition~\ref{D0}(i),
$\{\Omega_1,\Omega_2\}$ is
$\al'-$semitransversal at $\bx$ with $\de$.

\item
Let $\de_1:=\left(\al+\frac{1}{2}\right)\bar t$, $\de_2>0$, and
{inequality} \eqref{P4.2-2} be satisfied
for all $x\in\Omega_2\cap B_{2\delta_2}(\bar{x})$ with $d(x,\Omega_1)<\bar t$.
Let $\rho\in]0,\de_1[$ and ${x\in B_{\de_2}(\bx)}$ with $\max\{d(x,\Omega_1), d(x,\Omega_2)\}<\al'\rho$.
Choose a number $\ga>1$ such that
$$\|x-\bx\|<\ga\iv\de_2\AND \max\{d(x,\Omega_1),d(x,\Omega_2)\} <\ga\iv\al'\rho,$$
and a point $x'\in\Omega_2$ such that $\|x-x'\|\le\gamma d(x,\Omega_2).$
Then
\begin{align*}
\|x'-\bx\|&\le\|x-x'\|+\|x-\bx\|\le\gamma d(x,\Omega_2)+\|x-\bx\|
\\
&\le(\gamma+1)\|x-\bx\|<(1+\ga\iv)\de_2<2\delta_2,\\
d(x',\Omega_1)&\le\|x-x'\|+d(x,\Omega_1) \le(\ga+1)\max\{d(x,\Omega_1),d(x,\Omega_2)\}\\
&
<(1+\ga\iv)\al'\de_1<2\al'\de_1=\bar t.
\end{align*}
Hence, by \eqref{P4.2-2} with $x'$ in place of $x$,
\begin{align*}
d(x,\Omega_1\cap\Omega_2)
&\le\|x-x'\|+d(x',\Omega_1\cap\Omega_2)
\le\|x-x'\|+\varphi(d(x',\Omega_1))\\
&\le\|x-x'\| +\al d(x',\Omega_1)
\le(1+\al)\|x-x'\|+\al d(x,\Omega_1)\\
&\le(1+\al)\gamma d(x,\Omega_2)+\al d(x,\Omega_1)\\
&\le((1+\al)\gamma+\al)\max\{d(x,\Omega_1),d(x,\Omega_2)\}.
\end{align*}
Letting $\ga\downarrow1$, we arrive at
\begin{gather*}
d(x,\Omega_1\cap\Omega_2)\le
(1+2\al)\max\{d(x,\Omega_1),d(x,\Omega_2)\}<\rho.
\end{gather*}
Hence, $\Omega_1\cap\Omega_2\cap B_\rho(x){\ne\es}$ and, by Definition~\ref{D0}(ii),
$\{\Omega_1,\Omega_2\}$ is
$\al'-$sub\-transversal at $\bx$ with $\de_1$ and $\de_2$.

\item
The proof follows that of assertion (i) with the sets $\Omega_1-\omega_1$ and $\Omega_2-\omega_2$ in place of $\Omega_1$ and $\Omega_2$, respectively.
\qed\end{enumerate}
\end{proof}

\begin{remark}	
\begin{enumerate}
\item
In the linear case, Proposition~\ref{P6}(ii) recaptures \cite[Theorem~1(iii)]{KruLukTha18}, while
parts (i) and (iii) seem new.
\item
Restricted versions of the metric conditions in Theorem~\ref{P3.5} can be produced in a similar way.
\end{enumerate}
\end{remark}	

Checking the metric estimates of the $\varphi-$subtransversality and $\varphi-$transversality can be simplified as illustrated by the following proposition referring to condition \eqref{T1-2} in {Theorem~\ref{T4.1}(ii)}.
Equivalent versions of conditions \eqref{P2.2-1} and \eqref{P2.2-2} in Theorem~\ref{P3.5} look similar.
\begin{proposition}\label{P5}
The following conditions are equivalent:
\begin{enumerate}
\item
inequality \eqref{T1-2} holds true;
\item
for all $\omega_i\in\Omega_i$ $(i=1,\ldots,n)$, it holds
\begin{equation}\label{P3-1}
d\Big(x,\bigcap_{i=1}^n\Omega_i\Big)\le \varphi\Big(\max_{1\le i\le n}\|x-\omega_i\|\Big);
\end{equation}
\item
inequality \eqref{P3-1} holds true for all $\omega_i\in\Omega_i$ with $\|\omega_i-\bx\|<\|x-\bx\|+\varphi\iv(\|x-\bx\|)$ $(i=1,\ldots,n)$;

\item
inequality \eqref{P3-1} holds true for all $\omega_i\in\Omega_i$ with $\varphi(\|\omega_i-x\|)<\|x-\bx\|$ $(i=1,\ldots,n)$.
\end{enumerate}
\end{proposition}

\begin{proof}
The equivalence (i) \iff (ii) and implications (ii) \folgt (iii) \folgt (iv) are straightforward.
{We next} show that	
(iv) \folgt (ii).
Let condition (iv) hold true, $\omega_i\in\Omega_i$ $(i=1,\ldots,n)$, and $\varphi(\|\omega_i-x\|)\ge\|x-\bx\|$ for some $i$.
Then
\begin{equation*}
d\Big(x,\bigcap_{i=1}^n\Omega_i\Big)\le \|x-\bx\|\le\varphi\Big(\max_{1\le i\le n} \|x-\omega_i\|\Big),
\end{equation*}
i.e. {inequality} \eqref{P3-1} is satisfied, and consequently {condition} (ii) holds true.
\qed\end{proof}	

\section{Slope Sufficient Conditions}\label{S4}

In this section, we formulate slope sufficient conditions for the properties in Definition~\ref{D1}.
The conditions are straightforward consequences of the \EVP\ (Lemma~\ref{EVP}) applied to appropriate \lsc\ functions.
Throughout this section, $X$ is a Banach space, and the sets $\Omega_1,\ldots,\Omega_n$ are closed.
These are exactly the assumptions which ensure that the \EVP\ is applicable.
In view of Proposition~\ref{pro1}(iii), it suffices to assume that $\bx\in\bd\cap_{i=1}^n\Omega_i$.

The sufficient conditions for {the three properties} follow the same pattern.
We first {establish} nonlocal slope sufficient conditions arising from the \EVP.
These nonlocal conditions are largely of theoretical interest (unless the sets are convex): they encapsulate the application of the \EVP\ and serve as a source of more practical local (infinitesimal) conditions.
The corresponding local slope sufficient conditions, their H\"older as well as simplified $\de$-free versions are formulated as corollaries.
This way we expose the hierarchy of this type of conditions.

Along with the standard maximum norm on $X^{n+1}$, we are going to use also the following norm depending on a parameter $\ga>0$:
\begin{gather}\label{pnorm}
\|(x_1,\ldots,x_n,x)\|_{\ga} :=\max\Big\{\|x\|,\ga\max_{1\le i\le n} \|x_i\|\Big\},\quad
x_1,\ldots,x_n,x\in X.
\end{gather}

\subsection{Semitransversality}

\begin{theorem}\label{P7}
The collection $\{\Omega_1,\ldots,\Omega_n\}$ is $\varphi-$semitransversal at $\bx$ with some $\delta>0$ if, for some $\ga>0$ and any $x_i\in{X}$ $(i=1,\ldots,n)$ satisfying
\begin{gather}\label{P7-0}
0<\max_{1\le i\le n}\|x_i\|<\varphi\iv(\de),
\end{gather}
there exists a $\la\in]\varphi\left(\max_{1\le i\le n}\|x_i\|\right),\de[$ such that
\sloppy
\begin{align}\label{P7-1}
\sup_{\substack{u_i\in\Omega_i\;(i=1,\ldots,n),\;u\in X\\
(u_1,\ldots,u_n,u) \ne(\omega_1,\ldots,\omega_n,x)}} \dfrac{\varphi\Big(\max\limits_{1\le i\le n} \|\omega_i-x_i-x\|\Big)-\varphi\Big(\max\limits_{1\le i\le n} \|u_i-x_i-u\|\Big)}{\|(u_1,\ldots,u_n,u) -(\omega_1,\ldots,\omega_n,x)\|_{\ga}}\ge1
\end{align}
\if{
\NDC{8.5.19
Should we use $|\nabla\widehat{f}|_\ga^\diamond(\omega_1,\ldots,\omega_n,x)$here?}
}\fi
for all $x\in{X}$ and $\omega_i\in\Omega_i$ $(i=1,\ldots,n)$ satisfying
\begin{gather}\label{P7-2}
\|x-\bx\|<\la,\quad
\max_{1\le i\le n}\|\omega_i-\bx\|<\frac{\la}{\ga}, \\\label{P7-4}
0<\max_{1\le i\le n}{\|\omega_i-x_i-x\|}\le\max_{1\le i\le n}{\|x_i\|}.
\end{gather}
\end{theorem}

The proof below employs two closely related nonnegative functions on $X^{n+1}$ determined by the given function $\varphi\in\mathcal{C}$ and vectors
$x_1,\ldots,x_n\in X$:
\begin{gather}\label{f1}
f(u_1,\ldots,u_n,u):=\varphi\Big(\max\limits_{1\le i\le n}\|u_i-x_i-u\|\Big),\quad u_1,\ldots,u_n,u\in X,
\\\label{hatf}
\widehat{f}:=f+i_{\Omega_1\times\ldots\times\Omega_n}. \end{gather}

\begin{proof}
Suppose $\{\Omega_1,\ldots,\Omega_n\}$ is not $\varphi-$semitransversal at $\bx$ with some ${\delta>0}$, and
let $\ga>0$ be given.
By Definition~\ref{D1}(i), there exist a $\rho\in]0,\delta[$ and ${x_i\in{X}}$ {$(i=1,\ldots,n)$} with ${{\varphi(\max_{1\le i\le n}\|x_i\|)<\rho}}$ such that $\cap_{i=1}^{n}(\Omega_i-x_i)\cap B_\rho(\bar{x})=\emptyset$.
Thus, ${\max_{1\le i\le n}\|x_i\|>0}$.
Let ${\la\in]\varphi\left(\max_{1\le i\le n}\|x_i\|\right),\de[}$ and $\la':=\min\{\la,\rho\}$.
Then ${\la'>\varphi\left(\max_{1\le i\le n}\|x_i\|\right)}$,
$\cap_{i=1}^{n}(\Omega_i-x_i)\cap B_{\la'}(\bar{x})=\emptyset$, and consequently,
\begin{gather}\label{P7P01}
\max_{1\le i\le n}\|u_i-x_i-u\|>0
\qdtx{for all}
u_i\in\Omega_i\;(i=1,\ldots,n),\;u\in B_{\la'}(\bar{x}).
\end{gather}
Let $f$ and $\widehat{f}$ be defined by \eqref{f1} and \eqref{hatf}, respectively, while $X^{n+1}$ be equipped with the metric induced by the norm \eqref{pnorm}.
We have $\widehat{f}(\bx,\ldots,\bx,\bx) =\varphi\left(\max_{1\le i\le n} \|x_i\|\right)<\la'$.
Choose a number $\eps$ such that $\widehat{f}(\bx,\ldots,\bx,\bx) <\eps<\la'$.
Applying {the \EVP}, we can find points $\omega_i\in\Omega_i$ $(i=1,\ldots,n)$ and $x\in X$ such that
\sloppy
\begin{align}\label{P7P02}
\|(\omega_1,\ldots,\omega_n,x)-(\bx,\ldots,\bx,\bx)\|_{\ga}<\la'\le\la,\quad f(\omega_1,\ldots,\omega_n,x)\le f(\bx,\ldots,\bx,\bx),
\\\label{P7P03}
f(\omega_1,\ldots,\omega_n,x)
-f(u_1,\ldots,u_n,u)\le \frac{\varepsilon}{\la'}\|(u_1,\ldots,u_n,u) -(\omega_1,\ldots,\omega_n,x)\|_{\ga}
\end{align}
for all $(u_1,\ldots,u_n,u)\in \Omega_1\times\ldots\times \Omega_n\times X$.
In view of \eqref{P7P01} and the definitions of $\la'$ and~$f$, conditions \eqref{P7P02} yield \eqref{P7-2} and \eqref{P7-4}.
Since $\eps/\la'<1$, condition \eqref{P7P03} contradicts~\eqref{P7-1}.	
\qed\end{proof}

\begin{remark}%\label{R5.2}
The expression in the left-hand side of \eqref{P7-1} is the nonlocal $\ga$-slope \cite[p.~60]{Kru15}
at $(\omega_1,\ldots,\omega_n,x)$
of the function \eqref{hatf}.
%Using the slope notation can simplify the statements of Proposition~\ref{P7} and some others in this section (possibly at the expense of the clarity of the presentation).
\end{remark}

\if{
In the H\"older setting, Proposition~\ref{P7} yields the following statement.
}\fi

\if{
\begin{corollary}\label{P7+}
Let $\Omega_1,\ldots,\Omega_n$ be closed subsets of a Banach space $X$, $\bx\in \cap_{i=1}^n\Omega_i$, $\al>0$ and $q>0$.
$\{\Omega_1,\ldots,\Omega_n\}$ is $\al-$semitransversal of order $q$ at $\bx$ with some $\delta>0$ if, for some $\ga>0$ and any $x_i\in{X}$ $(i=1,\ldots,n)$ with $0<\max_{1\le i\le n}\|x_i\|<(\al\de)^{\frac{1}{q}}$, there exists a $\la\in]\al\iv\left(\max_{1\le i\le n}\|x_i\|\right)^q,\de[$ such that
\sloppy
\begin{align}\label{C4-1}
\sup_{\substack{u_i\in\Omega_i\;(i=1,\ldots,n),\;u\in X\\(u_1,\ldots,u_n,u) \ne (\omega_1,\ldots,\omega_n,x)}}\dfrac{\left(\max\limits_{1\le i\le n} \|\omega_i-x_i-x\|\right)^q-\left(\max\limits_{1\le i\le n} \|u_i-x_i-u\|\right)^q} {\|(u_1,\ldots,u_n,u) -(\omega_1,\ldots,\omega_n,x)\|_{\ga}}\ge\al
\end{align}
for all $x\in{X}$ and $\omega_i\in\Omega_i$ $(i=1,\ldots,n)$ satisfying \eqref{P7-2} and \eqref{P7-4}.	
\end{corollary}
\begin{proof}
The assertion is a direct consequence of Theorem~\ref{P7} with $\varphi(t):=\al\iv t^q$ for all $t\ge0$; then of course, $\varphi\iv(t)=(\al t)^{\frac{1}{q}}$.
\qed\end{proof}	
}\fi

The next statement is a localized version of Theorem~\ref{P7}.

\begin{corollary}\label{P7.2}
\begin{enumerate}
\item
The collection $\{\Omega_1,\ldots,\Omega_n\}$ is $\varphi-$semitransversal at $\bx$ with some $\delta>0$ if, for some $\ga>0$ and any $x_i\in{X}$ $(i=1,\ldots,n)$ satisfying
\eqref{P7-0}, there exists a $\la\in]\varphi\left(\max_{1\le i\le n}\|x_i\|\right),\de[$ such that
\sloppy
\begin{align}\label{P7.2-1}
\limsup_{\substack{{u_i}\stackrel{\Omega_i}{\to}\omega_i\; (i=1,\ldots,n),\;u\to x\\ (u_1,\ldots,u_n,u)\ne(\omega_1,\ldots,\omega_n,x)}} \hspace{-1mm}\dfrac{\varphi\Big(\max\limits_{1\le i\le n} \|\omega_i-x_i-x\|\Big)-\varphi\Big(\max\limits_{1\le i\le n}\|u_i-x_i-u\|\Big)}{\|(u_1,\ldots,u_n,u) -(\omega_1,\ldots,\omega_n,x)\|_{\ga}}\ge1
\end{align}
for all $x\in{X}$ and $\omega_i\in\Omega_i$ $(i=1,\ldots,n)$ {satisfying} \eqref{P7-2} and \eqref{P7-4}.

\item
If $\varphi\in\mathcal{C}^1$,
then {inequality \eqref{P7.2-1} in part \rm (i)} can be replaced by
\begin{multline}\label{C6-2}
\varphi'\Big(\max\limits_{1\le i\le n} \|\omega_i-x_i-x\|\Big)\times\\
\limsup_{\substack{{u_i}\stackrel{\Omega_i}{\to}\omega_i\; (i=1,\ldots,n),\;u\to x\\(u_1,\ldots,u_n,u)\ne(\omega_1,\ldots,\omega_n,x)}} \dfrac{\max\limits_{1\le i\le n}\|\omega_i-x_i-x\|-\max\limits_{1\le i\le n} \|u_i-x_i-u\|}{\|(u_1,\ldots,u_n,u)-(\omega_1,\ldots,\omega_n,x)\|_{\ga}}\ge1.
\end{multline}
\end{enumerate}
\end{corollary}

\begin{proof}
The expression in the left-hand side of \eqref{P7.2-1} is the $\ga$-slope \cite[p.~61]{Kru15} of the function \eqref{hatf} at $(\omega_1,\ldots,\omega_n,x)$.
The first assertion follows from Theorem~\ref{P7} in view of Proposition~\ref{P1.1}(ii), while the second one is a consequence of Lemma~\ref{L2} in view of Remark~\ref{R01}(i).
\qed\end{proof}

In the H\"older setting, Theorem~\ref{P7} and Corollary~\ref{P7.2} yield the following statement.

\begin{corollary}\label{C5.2}
Let
$\al>0$ and $q>0$.
The collection $\{\Omega_1,\ldots,\Omega_n\}$ is $\al-$semitransver\-sal of order $q$ at $\bx$ with some $\delta>0$ if, for some $\ga>0$ and any $x_i\in{X}$ $(i=1,\ldots,n)$ with $0<\max_{1\le i\le n}\|x_i\|<(\al\de)^\frac{1}{q}$, there exists a $\la\in]\al\iv(\max_{1\le i\le n}\|x_i\|)^q,\de[$ such that
\begin{align}\label{C3-0}
\sup_{\substack{u_i\in\Omega_i\;(i=1,\ldots,n),\;u\in X\\ (u_1,\ldots,u_n,u) \ne (\omega_1,\ldots,\omega_n,x)}}\dfrac{\Big(\max\limits_{1\le i\le n} \|\omega_i-x_i-x\|\Big)^q-\Big(\max\limits_{1\le i\le n} \|u_i-x_i-u\|\Big)^q}{\|(u_1,\ldots,u_n,u) -(\omega_1,\ldots,\omega_n,x)\|_{\ga}}\ge\al
\end{align}	
for all $x\in{X}$ and $\omega_i\in\Omega_i$ $(i=1,\ldots,n)$ satisfying \eqref{P7-2} and \eqref{P7-4}, or all the more, such that
\begin{multline}\label{C3-2}
q\Big(\max\limits_{1\le i\le n} \|\omega_i-x_i-x\|\Big)^{q-1}\\
\times \limsup_{\substack{{u_i}\stackrel{\Omega_i}{\to}\omega_i\; (i=1,\ldots,n),\;u\to x\\(u_1,\ldots,u_n,u)\ne(\omega_1,\ldots,\omega_n,x)}} \dfrac{\max\limits_{1\le i\le n}\|\omega_i-x_i-x\|-\max\limits_{1\le i\le n} \|u_i-x_i-u\|}{\|(u_1,\ldots,u_n,u)-(\omega_1,\ldots,\omega_n,x)\|_{\ga}}\ge\al.
\end{multline}
\end{corollary}

\begin{proof}
The statement is a direct consequence of Theorem~\ref{P7} and Corollary~\ref{P7.2} with $\varphi(t):=\al\iv t^q$ for all $t\ge0$.
Observe that $\varphi\iv(t)=(\al t)^{\frac{1}{q}}$.
\qed\end{proof}	

\begin{remark}\label{R5.3}
\begin{enumerate}
\item
On top of the explicitly given restriction $\|\omega_i-\bx\|<\la/\ga$ in {Theorem~\ref{P7}} (and similar conditions in its corollaries) on the choice of the points {$\omega_i\in\Omega_i$}, which involves $\ga$, the other conditions implicitly impose another one:
\begin{gather*}%\label{f}
\|\omega_i-\bx\|\le\|x-\bx\|+\|\omega_i-x_i-x\|+\|x_i\| \le\|x-\bx\|+2\max_{1\le i\le n}\|x_i\|,
\end{gather*}
and consequently, $\|\omega_i-\bx\|<\la+2\varphi\iv(\de)$.
This alternative restriction can be of importance when $\ga$ is small.

\item
The statements of Theorem~\ref{P7} and its corollaries can be simplified (and weakened!) by dropping condition \eqref{P7-4}.

\item
Inequalities \eqref{P7-1}, \eqref{P7.2-1}--\eqref{C3-2}, which are crucial for checking nonlinear semitransversality, involve two groups of parameters: on one hand, sufficiently small vectors {$x_i\in{X}$}, not all zero, and on the other hand, points $x\in{X}$ and {$\omega_i\in\Omega_i$} near $\bx$.
Note an important difference between these two groups.
The magnitudes of $x_i$ are directly controlled by the value of $\de$ in the definition of $\varphi-$semitransversality: $\varphi\left(\max_{1\le i\le n}\|x_i\|\right)<\de$.
At the same time, taking into account that $\la$ can be made arbitrarily close to $\varphi\left(\max_{1\le i\le n}\|x_i\|\right)$, the magnitudes of $x-\bx$ and $\omega_i-\bx$ (as well as $\omega_i-x_i-x$) are determined by $\de$ indirectly; they are controlled by $\max_{1\le i\le n}{\|x_i\|}$: cf. conditions \eqref{P7-2} and \eqref{P7-4}.

\item
In view of the definition of the parametric norm \eqref{pnorm}, if any of the inequalities \eqref{P7-1}, \eqref{P7.2-1}--\eqref{C3-2} holds true for some $\gamma>0$, then it also holds for any $\ga'\in]0,\ga[$.
\if{
\NDC{4.11.19
In the other two subsections, we are missing a similar comment about $\ga'$.}
}\fi

\item
Even in the linear setting, the characterizations in 	Corollary~\ref{C5.2} are new.
\end{enumerate}
\end{remark}

The next corollary provides a simplified (and weaker!) version of Theorem~\ref{P7}.
The simplification comes at the expense of eliminating the difference between the two groups of parameters highlighted in Remark~\ref{R5.3}(iii).

\begin{corollary}\label{C2.1}
The collection $\{\Omega_1,\ldots,\Omega_n\}$ is $\varphi-$semitransversal at $\bx$ with some ${\delta>0}$ if, for some $\ga>0$ and any $x_i\in{X}$ $(i=1,\ldots,n)$ satisfying
\eqref{P7-0}, inequality \eqref{P7-1} holds
for all $x\in{B}_\de(\bx)$ and $\omega_i\in\Omega_i\cap{B}_{\de/\ga}(\bx)$ $(i=1,\ldots,n)$ {satisfying} \eqref{P7-4}.
\sloppy
\end{corollary}

Sacrificing the estimates for $\de$ in Theorem~\ref{P7}, and Corollaries~\ref{P7.2} and ~\ref{C2.1}, we arrive at the following `$\de$-free' statement.

\begin{corollary}\label{C3+}
The collection $\{\Omega_1,\ldots,\Omega_n\}$ is $\varphi-$semitransversal at $\bx$ if,
for some $\ga>0$ and all $x_i\in{X}$ ${(i=1,\ldots,n)}$ near $0$ with $\max_{1\le i\le n}{\|x_i\|}>0$, {$x\in{X}$ and} $\omega_i\in\Omega_i$ $(i=1,\ldots,n)$ near $\bx$ satisfying
\eqref{P7-4}, inequality \eqref{P7-1} holds true.
Moreover,
inequality \eqref{P7-1} can be replaced by its localized version \eqref{P7.2-1}, or by \eqref{C6-2} if $\varphi\in\mathcal{C}^1$.
\end{corollary}

\subsection{Subtransversality}\label{S5.2}

\begin{theorem}\label{P8}
The collection $\{\Omega_1,\ldots,\Omega_n\}$ is $\varphi-$subtransversal at $\bx$ with some $\de_1>0$ and $\de_2>0$ if, for some $\ga>0$ and any $x'\in X$ satisfying
\begin{gather}\label{P8-0}
\|x'-\bx\|<\de_2,\quad
0<\max_{1\le i\le n}d(x',\Omega_i)<\varphi\iv(\de_1),
\end{gather}
there exists a $\la\in]\varphi\left(\max_{1\le i\le n} d(x',\Omega_i)\right),\de_1[$ such that
\sloppy
\begin{align}\label{P8-1}
\sup_{\substack{u_i\in\Omega_i\;(i=1,\ldots,n),\;u\in X\\
(u_1,\ldots,u_n,u) \ne(\omega_1,\ldots,\omega_n,x)}} \dfrac{\varphi\Big(\max\limits_{1\le i\le n} \|\omega_i-x\|\Big)-\varphi\Big(\max\limits_{1\le i\le n} \|u_i-u\|\Big)} {\|(u_1,\ldots,u_n,u)-(\omega_1,\ldots,\omega_n,x)\|_{\ga}} \ge1
\end{align}
for all $x\in{X}$ and $\omega_i,\omega'_i\in\Omega_i$ $(i=1,\ldots,n)$ satisfying
\begin{gather}\label{P8-3}
\|x-x'\|<\la,\quad
\max_{1\le i\le n}\|\omega_i-\omega'_i\|<\frac{\la}{\ga},
\\\label{P8-4}
0<\max_{1\le i\le n}{\|\omega_i-x\|}\le\max_{1\le i\le n} \|\omega'_i-x'\|<\varphi\iv(\la).
\end{gather}
\end{theorem}

The proof below follows the pattern of that of Theorem~\ref{P7}.
It employs
a continuous real-valued function $f:X^{n+1}\to\R_+$ determined by the given function $\varphi\in\mathcal{C}$:
\begin{gather}\label{f}
f(u_1,\ldots,u_n,u):=\varphi\Big(\max\limits_{1\le i\le n} \|u_i-u\|\Big),\quad u_1,\ldots,u_n,u\in X,
\end{gather}
and its restriction to $\Omega_1\times\ldots\times\Omega_n\times X$ given by \eqref{hatf}.
Note that the function \eqref{f} is a particular case of \eqref{f1} corresponding to setting $x_i:=0$ $(i=1,\ldots,n)$.
We provide here the proof of Theorem~\ref{P8} for completeness and to expose the differences in handling the two transversality properties, but we skip the proofs of most of its corollaries.

\begin{proof}
Suppose $\{\Omega_1,\ldots,\Omega_n\}$ is not $\varphi-$subtransversal at $\bx$ with some $\de_1>0$ and ${\de_2>0}$, and
let $\ga>0$ be given.
By Definition~\ref{D1}(ii), there exist a number $\rho\in]0,\de_1[$ and a point $x'\in B_{\de_2}(\bx)$ such that $\varphi\left(\max_{1\le i\le n} d(x',\Omega_i)\right)<\rho$ and $\cap_{i=1}^n\Omega_i\cap B_\rho(x')=\es$.
Hence, $x'\notin\cap_{i=1}^n\Omega_i$ and
\begin{equation*}%\label{P8P1}
0<\varphi\Big(\max_{1\le i\le n}d(x',\Omega_i)\Big)
<\rho\le d\Big(x',\bigcap_{i=1}^n\Omega_i\Big).
\end{equation*}
Let $\la\in]\varphi\left(\max_{1\le i\le n}d(x',\Omega_i)\right),\de_1[$.
Choose numbers $\eps$ and $\la'$ such that \begin{equation*}%\label{P8P1}
\varphi\Big(\max_{1\le i\le n}d(x',\Omega_i)\Big) <\eps<\la'<\min\left\{\la,\rho\right\},
\end{equation*}
and points $\omega'_i\in\Omega_i$ $(i=1,\ldots,n)$ such that
$\varphi\left(\max_{1\le i\le n}\|\omega'_i-x'\|\right)<\eps$.
Let $f$ and $\widehat{f}$ be defined by \eqref{f} and \eqref{hatf}, respectively, while $X^{n+1}$ be equipped with the metric induced by the norm \eqref{pnorm}.
We have $\widehat{f}(\omega'_1,\ldots,\omega'_n,x')<\eps$.
Applying {the \EVP}, we can find points $\omega_i\in\Omega_i$ $(i=1,\ldots,n)$ and $x\in X$ such that
\sloppy
\begin{gather}\label{P8P2}
\|(\omega_1,\ldots,\omega_n,x)-(\omega'_1,\ldots, \omega'_n,x')\|_{\ga}<\lambda',\;f(\omega_1,\ldots,\omega_n,x)\le f(\omega'_1,\ldots,\omega'_n,x'),
\\\label{P8P3}
f(\omega_1,\ldots,\omega_n,x)
-f(u_1,\ldots,u_n,u)\le\frac{\varepsilon}{\lambda'}\|(u_1,\ldots,u_n,u)-(\omega_1,\ldots,\omega_n,x)\|_{\ga}
\end{gather}
for all $(u_1,\ldots,u_n,u)\in \Omega_1\times\ldots\times \Omega_n\times X$.
Thanks to \eqref{P8P2}, we have
$\|x-x'\|<\lambda'$, and consequently,
\begin{gather*}
d\Big(x,\bigcap_{i=1}^n\Omega_i\Big)\ge d\Big(x',\bigcap_{i=1}^n \Omega_i\Big)-\|x-x'\| >d\Big(x',\bigcap_{i=1}^n\Omega_i\Big)-\lambda'>0.
\end{gather*}
Hence, $x\notin\cap_{i=1}^n\Omega_i$, and $\max_{1\le i\le n}\|\omega_i-x\|>0$.
In view of the definitions of $\la'$ and $f$, conditions \eqref{P8P2} together with the last inequality yield \eqref{P8-3} and \eqref{P8-4}.
Since $\eps/\la'<1$, condition \eqref{P8P3} contradicts \eqref{P8-1}.
\qed\end{proof}

\if{
In the H\"older setting, Theorem~\ref{P8} yields the following statement.
\begin{corollary}\label{P9+}
Let $\Omega_1,\ldots,\Omega_n$ be closed subsets of a Banach space $X$, $\bx\in \cap_{i=1}^n\Omega_i,\al>0$ and $q\in]0,1]$.
$\{\Omega_1,\ldots,\Omega_n\}$ is $\al-$subtransversal of order $q$ at $\bx$ with some $\de_1>0$ and $\de_2>0$ if, for some $\ga>0$ and any $x'\in B_{\de_2}(\bx)$ with $0<\max_{1\le i\le n}d(x',\Omega_i)<(\al\de_1)^{\frac{1}{q}}$,
there exists a $\la\in]\al\iv\left(\max_{1\le i\le n} d(x',\Omega_i)\right)^q,\de_1[$ such that
\sloppy
\begin{align}\label{C13-1}
\sup_{\substack{u_i\in\Omega_i\;(i=1,\ldots,n),\;u\in X\\(u_1,\ldots,u_n,u) \ne (\omega_1,\ldots,\omega_n,x)}} \dfrac{\left(\max\limits_{1\le i\le n} \|\omega_i-x\|\right)^q-\left(\max\limits_{1\le i\le n} \|u_i-u\|\right)^q}{\|(u_1,\ldots,u_n,u)-(\omega_1,\ldots,\omega_n,x)\|_{\ga}} \ge\al
\end{align}
for all $x\in{X}$ and $\omega_i,\omega'_i\in\Omega_i$ $(i=1,\ldots,n)$ satisfying \eqref{P8-3} and
\begin{align}\label{C13-2}
0<\max_{1\le i\le n}{\|\omega_i-x\|}\le\max_{1\le i\le n} \|\omega'_i-x'\|<(\al\la)^\frac{1}{q}.
\end{align}
\end{corollary}
\begin{proof}
Both assertions are consequences of the corresponding ones in Theorem~\ref{P8} with $\varphi(t):=\al\iv t^q$ for all $t\ge0$; cf. the proof of Corollary~\ref{P7+}.
\qed\end{proof}	
}\fi

The next statement is a localized version of Theorem~\ref{P8}.

\begin{corollary}\label{4.1}
\begin{enumerate}
\item
The collection $\{\Omega_1,\ldots,\Omega_n\}$ is $\varphi-$subtransversal at $\bx$ with some $\de_1>0$ and $\de_2>0$ if, for some $\ga>0$ and any $x'\in X$ satisfying \eqref{P8-0},
there exists a $\la\in]\varphi\left(\max_{1\le i\le n} d(x',\Omega_i)\right),\de_1[$ such that
\sloppy
\begin{align}\label{P6-1}
\limsup_{\substack{u_i\stackrel{\Omega_i}{\to}\omega_i\; (i=1,\ldots,n),\;u\to x\\
(u_1,\ldots,u_n,u)\ne(\omega_1,\ldots,\omega_n,x)}} \dfrac{\varphi\Big(\max\limits_{1\le i\le n} \|\omega_i-x\|\Big)-\varphi\Big(\max\limits_{1\le i\le n} \|u_i-u\|\Big)}{\|(u_1,\ldots,u_n,u) -(\omega_1,\ldots,\omega_n,x)\|_{\ga}}\ge1
\end{align}
for all $x\in{X}$ and $\omega_i,\omega'_i\in\Omega_i$ $(i=1,\ldots,n)$ satisfying \eqref{P8-3} and \eqref{P8-4}.

\item
If $\varphi\in\mathcal{C}^1$, then inequality \eqref{P6-1} {in part \rm (i)} can be replaced by
\begin{multline}\label{P6-2}
\varphi'\Big(\max\limits_{1\le i\le n}\|\omega_i-x\|\Big) \\
\times \limsup_{\substack{{u_i}\stackrel{\Omega_i}{\to}\omega_i\; (i=1,\ldots,n),\;u\to x\\
(u_1,\ldots,u_n,u)\ne(\omega_1,\ldots,\omega_n,x)}} \dfrac{\max\limits_{1\le i\le n}\|\omega_i-x\|-\max\limits_{1\le i\le n}\|u_i-u\|}{\|(u_1,\ldots,u_n,u)-(\omega_1,\ldots,\omega_n,x)\|_{\ga}}\ge1.
\end{multline}
\end{enumerate}
\end{corollary}

In the H\"older setting, Theorem~\ref{P8} and Corollary~\ref{4.1} yield the following statement.
In view of Remark~\ref{R2}, we assume that $q\le1$.

\begin{corollary}\label{C9}
Let ${\al>0}$ and $q\in]0,1]$.
The collection $\{\Omega_1,\ldots,\Omega_n\}$ is $\al-$subtrans\-ver\-sal of order $q$ at $\bx$ with {some} $\de_1>0$ and $\de_2>0$ if, for some $\ga>0$ and any $x'\in B_{\de_2}(\bx)$ with ${0<\max\limits_{1\le i\le n} d(x',\Omega_i)<(\al\de_1)^\frac{1}{q}}$,
there exists a $\la\in\big]\al\iv\big(\max\limits_{1\le i\le n} d(x',\Omega_i)\big)^q,\de_1\big[$ such that
\begin{align}\label{C5.6-1}
\sup_{\substack{u_i\in\Omega_i\;(i=1,\ldots,n),\;u\in X\\ (u_1,\ldots,u_n,u) \ne (\omega_1,\ldots,\omega_n,x)}} \dfrac{\Big(\max\limits_{1\le i\le n} \|\omega_i-x\|\Big)^q-\Big(\max\limits_{1\le i\le n} \|u_i-u\|\Big)^q}{\|(u_1,\ldots,u_n,u) -(\omega_1,\ldots,\omega_n,x)\|_{\ga}}\ge\al,
\end{align}	
for all $x\in{X}$ and ${\omega_i,\omega'_i\in\Omega_i}$ $(i=1,\ldots,n)$ satisfying \eqref{P8-3} and
%\sloppy
\begin{align*}
0<\max_{1\le i\le n}{\|\omega_i-x\|}\le\max_{1\le i\le n} \|\omega'_i-x'\|<(\al\la)^\frac{1}{q},
\end{align*}
or all the more, such that
\begin{multline}\label{C5.6-2}
q\Big(\max\limits_{1\le i\le n} \|\omega_i-x\|\Big)^{q-1}\hspace{-2mm}
\\
\times\limsup_{\substack{{u_i}\stackrel{\Omega_i}{\to}\omega_i\; (i=1,\ldots,n),\;u\to x\\
(u_1,\ldots,u_n,u)\ne(\omega_1,\ldots,\omega_n,x)}} \dfrac{\max\limits_{1\le i\le n}\|\omega_i-x\|-\max\limits_{1\le i\le n}\|u_i-u\|}{\|(u_1,\ldots,u_n,u)-(\omega_1,\ldots,\omega_n,x)\|_{\ga}}\ge \al.
\end{multline}
\end{corollary}
\if{
\begin{proof}
The assertion is a direct consequence of Theorem~\ref{P8} and Corollary~\ref{4.1} with $\varphi(t):=\al\iv t^q$ for all $t\ge0$; then of course, $\varphi\iv(t)=(\al t)^{\frac{1}{q}}$.
\qed\end{proof}	
}\fi

\begin{remark}\label{R10}
\begin{enumerate}
\item
The expressions in the left-hand sides of \eqref{P8-1} and \eqref{P6-1} are, respectively, the nonlocal $\ga$-slope {and} the $\ga$-slope
at $(\omega_1,\ldots,\omega_n,x)$
of the function \eqref{hatf}.

\item
Under the conditions of Theorem~\ref{P8}, there are two ways for estimating $\|\omega_i-\bx\|$:
\begin{align*}
\|\omega_i-\bx\|&\le\|x'-\bx\|+\|\omega_i-\omega'_i\|+\|\omega'_i-x'\| <\de_2+\la/\ga+\varphi\iv(\la),\\
\|\omega_i-\bx\|&\le\|x-\bx\|+\|\omega_i-x\|\\
&\le\|x'-\bx\|+\|x-x'\|+\max_{1\le i\le n}\|\omega'_i-x'\|<\de_2+\la+\varphi\iv(\la).
\end{align*}
The second estimate does not involve $\ga$ and is better than the first one when $\ga<1$.
A similar observation can be made about Corollary~\ref{C5}.

\item
It can be observed from the proof of Theorem~\ref{P8} that the sufficient conditions {for} $\varphi-$sub\-transversality can be strengthened by adding another restriction on the choice of $x'$: $\varphi\left(\max_{1\le i\le n}d(x',\Omega_i)\right)
<d\left(x',\cap_{i=1}^n\Omega_i\right)$.

\item
The statement of Theorem~\ref{P8} and its corollaries can be simplified by dropping condition \eqref{P8-4}.

\item
Inequalities \eqref{P8-1}, \eqref{P6-1}--\eqref{C5.6-2}, which are crucial for checking nonlinear subtransversality, involve points $x\in{X}$ and {$\omega_i\in\Omega_i$} near $\bx$.
Their distance from $\bx$ is determined in Theorem~\ref{P8} via other points: $x'\notin\cap_{i=1}^n \Omega_i$ and {$\omega'_i\in\Omega_i$}; cf. conditions \eqref{P8-3} and \eqref{P8-4}.
Only the distance from $x'$ to $\bx$ and to the sets {$\Omega_i$} is directly controlled by the values of $\de_1$ and $\de_2$ in the definition of $\varphi-$subtransversality: $x'\in B_{\de_2}(\bx)$ and $\varphi\left(\max_{1\le i\le n} d(x',\Omega_i)\right)<\de_1$.
All the other distances are controlled by $\la$, which can be made arbitrarily close to $\varphi\left(\max_{1\le i\le n}d(x',\Omega_i)\right)$.

\item
In view of the definition of the parametric norm \eqref{pnorm}, if any of the inequalities \eqref{P8-1}, \eqref{P6-1}--\eqref{C5.6-2} holds true for some $\gamma>0$, then it also holds for any $\ga'\in]0,\ga[$.

\item
Corollary~\ref{C9} strengthens \cite[Proposition~6]{KruTha14}.
In the linear case, it {improves} \cite[Proposition~10]{KruLukTha17}.
\end{enumerate}
\end{remark}

The next corollary provides a simplified (and weaker!) version of Theorem~\ref{P8}; cf.
Remark~\ref{R10}(v).

\begin{corollary}\label{C5}
The collection $\{\Omega_1,\ldots,\Omega_n\}$ is $\varphi-$subtransversal at $\bx$ with some ${\de_1>0}$ and $\de_2>0$ {if, for some $\ga>0$,}
inequality \eqref{P8-1} holds
for all $x\in{B}_{\de_1+\de_2}(\bx)$ and $\omega_i\in\Omega_i\cap {B}_{\de_2+\de_1/\ga+\varphi\iv(\delta_1)}(\bx)$ $(i=1,\ldots,n)$ satisfying
$0<\max_{1\le i\le n}{\|\omega_i-x\|} <\varphi\iv(\delta_1)$.
\sloppy
\end{corollary}

\begin{proof}
Let $\de_1>0$ and $\de_2>0$, $x'\in B_{\de_2}(\bx)\setminus\cap_{i=1}^n \Omega_i$, $\la\in]\varphi\left(\max_{1\le i\le n} d(x',\Omega_i)\right),\de_1[$, and points $x\in{X}$ and $\omega_i,\omega'_i\in\Omega_i$ $(i=1,\ldots,n)$ satisfy conditions \eqref{P8-3} and \eqref{P8-4}.
Then
\sloppy
\begin{align*}
\|x-\bx\|&\le\|x-x'\|+\|x'-\bx\|<\la+\de_2<\de_1+\de_2,\\
\|\omega_i-\bx\|&\le\|x'-\bx\|+\|\omega_i-\omega'_i\|+\|\omega'_i-x'\|\\
&<\de_2+\la/\ga+\varphi\iv(\la) <\de_2+\de_1/\ga+\varphi\iv(\delta_1),\\
\|\omega_i-x\|&<\varphi\iv(\la)<\varphi\iv(\delta_1),
\end{align*}
i.e. points $x\in{X}$ and $\omega_i\in\Omega_i$ $(i=1,\ldots,n)$ satisfy all the conditions in the corollary.
Hence, inequality \eqref{P8-1} holds.
It follows from Theorem~\ref{P8} that $\{\Omega_1,\ldots,\Omega_n\}$ is $\varphi-$subtransversal at $\bx$ with $\de_1$ and $\de_2$.
\qed\end{proof}

Sacrificing the estimates for $\de_1$ and $\de_2$ in Theorem~\ref{P8}, and Corollaries~\ref{4.1} and
\ref{C5}, we can formulate the following `$\de$-free' statement.

\begin{corollary}\label{C2}
The collection $\{\Omega_1,\ldots,\Omega_n\}$ is $\varphi-$subtransversal at $\bx$ if inequality \eqref{P8-1} holds true
for some ${\ga>0}$ and all $x\in{X}$ near $\bx$ and $\omega_i\in\Omega_i$ $(i=1,\ldots,n)$ near $\bx$ satisfying
$\max_{1\le i\le n}{\|\omega_i-x\|}>0$.
Moreover,
inequality \eqref{P8-1} can be replaced by its localized version \eqref{P6-1}, or by \eqref{P6-2} if $\varphi\in\mathcal{C}^1$.
\end{corollary}

\subsection{Transversality}\label{S5.3}

Since $\varphi-$transversality is in a sense an overarching property covering both $\varphi-$semi\-trans\-versality and $\varphi-$subtransversality (see Proposition~\ref{pro1}(iii)), the next theorem contains some elements of both Theorems~\ref{P7} and \ref{P8}, {and} its proof goes along the same lines.
Similar to the proof of Theorem~\ref{P7},
it employs functions \eqref{f1} and~\eqref{hatf}.

\begin{theorem}\label{P9}
The collection $\{\Omega_1,\ldots,\Omega_n\}$ is $\varphi-$transversal at $\bx$ with some ${\de_1>0}$ and $\de_2>0$ if, for some $\ga>0$ and any $\omega'_i\in\Omega_i\cap B_{\delta_2}(\bar{x})$ $(i=1,\ldots,n)$ and ${\xi\in]0,\varphi\iv(\de_1)[}$, there exists a $\la\in]\varphi(\xi),\de_1[$ such that
inequality \eqref{P7-1} holds
for all $x,x_i\in{X}$ and $\omega_i\in\Omega_i$ $(i=1,\ldots,n)$ satisfying
\sloppy
\begin{gather}
\label{P11-1}
\|x-\bx\|<\la,\quad
\max_{1\le i\le n}\|\omega_i-\omega'_i\|<\frac{\la}{\ga},
\\\label{P11-3}
0<\max_{1\le i\le n}{\|\omega_i-x_i-x\|}\le\max_{1\le i\le n} \|\omega'_{i}-x_i-\bx\|=\xi.
\end{gather}
\if{
\AK{25/04/19.
Could $\xi$ be eliminated from the statement without making it weaker?}
\NDC{30.4.19
Yes, I think $\xi$ can be equivalently dropped. In this case, I think part (i) can be rewritten in the form that:
`$\{\Omega_1,\ldots,\Omega_n\}$ is $\varphi-$transversal at $\bx$ with some $\de_1>0$ and $\de_2>0$ if, for some $\ga>0$ and any $\omega'_i\in\Omega_i\cap B_{\delta_2}(\bar{x})$, $x_i\in X$ $(i=1,\ldots,n)$ with
$0<\max_{1\le i\le n}\|\omega'_{i}-x_i-\bx\|<\varphi\iv(\de_1)$, there exists a $\la\in]\varphi(\max_{1\le i\le n}\|\omega'_{i}-x_i-\bx\|),\de_1[$ such that
inequality \eqref{P7-1} holds
for all $x\in{X}$ and $\omega_i\in\Omega_i$ $(i=1,\ldots,n)$ satisfying $\ldots$'.}
\AK{4/05/19.
I am not sure about the equivalence.
Your green sufficient condition looks stronger (although still meaningful; it could possibly make sense to formulate it as a corollary).
The above statement allows choosing $\xi$ arbitrarily small and, thanks to \eqref{P11-3}, reducing the set of $x,x_i\in{X}$ and $\omega_i\in\Omega_i$ $(i=1,\ldots,n)$ to be tested in \eqref{P7-1}.}
\NDC{8.5.19
I think it would be good to add it as corollary together with your comment about the role of $\xi$.
}
}\fi
\end{theorem}

\begin{proof}
Suppose $\{\Omega_1,\ldots,\Omega_n\}$ is not $\varphi-$transversal at $\bx$ with some $\de_1>0$ and ${\de_2>0}$, and let $\ga>0$ be given.
By Definition~\ref{D1}(iii), there exist a number $\rho\in]0,\delta_1[$ and points $\omega'_i\in\Omega_i \cap B_{\delta_2}(\bar{x})$ and $x'_i\in{X}$ {$(i=1,\ldots,n)$} with ${{\varphi(\max_{1\le i\le n}\|x'_i\|)<\rho}}$ such that
${\cap_{i=1}^n(\Omega_i-\omega'_i-x'_i) \cap(\rho\B)=\emptyset}$.
Thus, $\xi:=\max_{1\le i\le n}\|x'_i\|>0$ and
${\xi<\varphi\iv(\rho)<\varphi\iv(\de_1)}$.
Set $x_i:=\omega'_i+x'_i-\bx$ ${(i=1,\ldots,n)}$.
Then
\begin{align*}
\max_{1\le i\le n}\|\omega'_{i}-x_i-\bx\|=\max_{1\le i\le n}\|x'_i\|=\xi.
\end{align*}
Let $\la\in]\varphi(\xi),\de_1[$ and $\la':=\min\{\la,\rho\}$.
Then $\cap_{i=1}^n(\Omega_i-x_i)\cap B_{\la'}(\bx)=\emptyset$, and consequently, condition \eqref{P7P01} holds true.
Let $f$ and $\widehat{f}$ be defined by \eqref{f1} and \eqref{hatf}, respectively, while $X^{n+1}$ be equipped with the metric induced by the norm \eqref{pnorm}.
We have $\widehat{f}(\omega'_1,\ldots,\omega'_n,\bx) =\varphi\left(\max_{1\le i\le n}\|x'_i\|\right)=\varphi\left(\xi\right)<\la'$.
Choose a number $\eps$ such that
$\widehat{f}(\omega'_1,\ldots,\omega'_n,\bx)<\eps<\la'$.
Applying {the \EVP}, we can find points $\omega_i\in\Omega_i$ $(i=1,\ldots,n)$ and $x\in X$ such that
\sloppy
\begin{gather}\label{P9P02}
\|(\omega_1,\ldots,\omega_n,x)-(\omega'_1,\ldots,\omega'_n,\bx)\|_{\ga}<\la',\;\;
f(\omega_1,\ldots,\omega_n,x)\le f(\omega'_1,\ldots,\omega'_n,\bx),
\end{gather}
and condition \eqref{P7P03} holds
for all $u\in X$ and $u_i\in\Omega_i$ $(i=1,\ldots,n)$.
In view of \eqref{P7P01} and the definitions of $\la'$ and $f$, conditions \eqref{P9P02} yield \eqref{P11-1} and \eqref{P11-3}.
Since $\eps/\la'<1$, condition \eqref{P7P03} contradicts \eqref{P7-1}.
\qed\end{proof}

The next statement is a localized version of Theorem~\ref{P9}.

\begin{corollary}\label{P12}
\begin{enumerate}
\item
The collection $\{\Omega_1,\ldots,\Omega_n\}$ is $\varphi-$transversal at $\bx$ with some ${\de_1>0}$ and $\de_2>0$ if, for some $\ga>0$ and any $\omega'_i\in\Omega_i\cap B_{\delta_2}(\bar{x})$ $(i=1,\ldots,n)$ and ${\xi\in]0,\varphi\iv(\de_1)[}$, there exists a ${\la\in]\varphi(\xi),\de_1[}$ such that
inequality \eqref{P7.2-1} holds for all $x,x_i\in{X}$ and $\omega_i\in\Omega_i$ ${(i=1,\ldots,n)}$ satisfying
\eqref{P11-1} and \eqref{P11-3}.
\sloppy

\item
If $\varphi\in\mathcal{C}^1$, then {inequality \eqref{P7.2-1} in part \rm (i)} can be replaced by \eqref{C6-2}.
\end{enumerate}
\end{corollary}

In the H\"older setting, Theorem~\ref{P9} and Corollary~\ref{P12} yield the following statement.
In view of Remark~\ref{R2}, we assume that $q\le1$.

\begin{corollary}\label{C5.10}
Let $\al>0$ and {$q\in]0,1]$}.
The collection $\{\Omega_1,\ldots,\Omega_n\}$ is $\al-$trans\-ver\-sal of order $q$ at $\bx$ with some $\de_1>0$ and $\de_2>0$ if, for some $\ga>0$ and any ${\omega'_i\in\Omega_i\cap B_{\delta_2}(\bar{x})}$ $(i=1,\ldots,n)$ and $\xi\in]0,(\al\de_1)^\frac{1}{q}[$, there exists a $\la\in]\al\iv\xi^q,\de_1[$ such that inequality \eqref{C3-0} holds true for all $x,x_i\in{X}$ and $\omega_i\in\Omega_i$ $(i=1,\ldots,n)$ satisfying \eqref{P11-1} and \eqref{P11-3}, or all the more, such that inequality \eqref{C3-2} holds true.
\end{corollary}

\begin{remark}\label{R13}
\begin{enumerate}
\item
On top of the explicitly given restriction $\|\omega_i-\omega'_i\|<\la/\ga$ in Theorem~\ref{P9} (and similar conditions in its corollaries), which involves $\ga$, the other conditions implicitly impose another one:
\begin{align*}
\|\omega_i-\omega'_i\|&\le\|x-\bx\|+\|\omega_i-x_i-x\| +\|\omega'_i-x_i-\bx\|\\
&\le\|x-\bx\|+2\xi <\la+2\varphi\iv(\de_1).
\end{align*}
This alternative restriction can be of importance when $\ga$ is small.
%A similar observation can be made about Corollary~\ref{C10}.

\item
It can be observed from the proof of Theorem~\ref{P9} that the sufficient conditions for $\varphi-$transversality can be strengthened by adding another restriction on the choice of $\xi$ and {$x_i$}:
$\varphi(\xi)<d(\bx,\cap_{i=1}^n(\Omega_i-x_i))$.

\item
The sufficient conditions {for} $\varphi-$semitransversality and $\varphi-$subtransversality in Theorems~\ref{P7} and \ref{P8} are particular cases of those in Theorem~\ref{P9},
corresponding to setting {$\omega'_i:=\bx$} and $x_1=\ldots=x_n$, respectively.

\item
The statement of Theorem~\ref{P9} and its corollaries can be simplified by dropping condition \eqref{P11-3}.

\item	
Inequalities \eqref{P7-1}, \eqref{P7.2-1}--\eqref{C3-2}, which are crucial for checking nonlinear transversality, involve a collection of parameters: $x,x_i\in{X}$ and {$\omega_i\in\Omega_i$}, which are related to another collection: a small number $\xi>0$ and points {$\omega'_i\in\Omega_i$} near $\bx$.
The value of $\xi$ and magnitudes of {$\omega'_i-\bx$} are directly controlled by the values of $\de_1$ and $\de_2$ in the definition of $\varphi-$transversality: $\varphi(\xi)<\de_1$ and {$\omega'_i\in B_{\delta_2}(\bar{x})$}.
At the same time, taking into account that $\la$ can be made arbitrarily close to $\varphi(\xi)$, the magnitudes of $x-\bx$, $\omega_i-\omega'_i$ and {$x_i$} are determined by $\de_1$ and $\de_2$ indirectly; they are controlled by $\xi$: cf. conditions \eqref{P11-1} and \eqref{P11-3}.
Thus, the derived parameters $x,x_i\in{X}$ and {$\omega_i\in\Omega_i$} involved in \eqref{P7-1} possess the natural properties: when $\de_1$ and $\de_2$ are small, the points $x$ and {$\omega_i$} are near $\bx$ and the vectors {$x_i$} are small.

\item
In view of the definition of the parametric norm \eqref{pnorm}, if any of the inequalities \eqref{P7-1}, \eqref{P7.2-1}--\eqref{C3-2} holds true for some $\gamma>0$, then it also holds for any $\ga'\in]0,\ga[$.

\item
Even in the linear setting, the characterizations in 	Corollary~\ref{C5.10} are new.
\end{enumerate}
\end{remark}

The next corollary provides a simplified (and weaker!) version of Theorem~\ref{P9}; cf. Remark~\ref{R13}(v).

\begin{corollary}\label{C10}
The collection $\{\Omega_1,\ldots,\Omega_n\}$ is $\varphi-$transversal at $\bx$ with some $\de_1>0$ and $\de_2>0$ if, for some $\ga>0$, inequality \eqref{P7-1} holds for all $x\in{B}_{\de_1}(\bx)$, $x_i\in X$ and $\omega_i\in\Omega_i\cap {B}_{\de_2+\de_1/\ga}(\bx)$ $(i=1,\ldots,n)$ satisfying $\varphi\left({\max}_{1\le i\le n} d(x_i+\bx,\Omega_i)\right)<\de_1$ and ${0<\max_{1\le i\le n}{\|\omega_i-x_i-x\|}<\varphi\iv(\delta_1)}$.
\end{corollary}

\begin{proof}
Let $\de_1>0$, $\de_2>0$, $\omega'_i\in\Omega_i\cap B_{\delta_2}(\bar{x})$, $\xi\in]0,\varphi\iv(\de_1)[$, $\la\in]\varphi(\xi), \de_1[$, and points $x,x_i\in{X}$ and $\omega_i\in\Omega_i$ $(i=1,\ldots,n)$ satisfy conditions \eqref{P11-1} and \eqref{P11-3}.
Then
\sloppy
\begin{gather*}%\label{f}
\|x-\bx\|<\la<\de_1,\quad
\|\omega_i-\bx\|\le\|\omega'_i-\bx\|+\|\omega_i-\omega'_i\| <\de_2+\la/\ga<\de_2+\de_1/\ga,\\
d(x_i+\bx,\Omega_i)\le\|x_i+\bx-\omega'_i\|\le\xi <\varphi\iv(\delta_1),\\
0<\max_{1\le i\le n}{\|\omega_i-x_i-x\|}\le\xi<\varphi\iv(\delta_1),
\end{gather*}
i.e. points $x,x_i\in{X}$ and $\omega_i\in\Omega_i$ $(i=1,\ldots,n)$ satisfy all the conditions in the corollary.
Hence, inequality \eqref{P7-1} holds.
It follows from Theorem~\ref{P9} that $\{\Omega_1,\ldots,\Omega_n\}$ is $\varphi-$transversal at $\bx$ with $\de_1$ and $\de_2$.
\qed\end{proof}

Sacrificing the estimates for $\de_1$ and $\de_2$ in Theorem~\ref{P9}, and Corollaries~\ref{P12} and
\ref{C10}, we can formulate the following `$\de$-free' statement.

\begin{corollary}\label{C11.2}
The collection $\{\Omega_1,\ldots,\Omega_n\}$ is $\varphi-$transversal at $\bx$ if,
for some $\ga>0$ and all $x\in{X}$ near $\bx$, $x_i\in X$ $(i=1,\ldots,n)$ near $0$ and $\omega_i\in\Omega_i$ $(i=1,\ldots,n)$ near $\bx$ satisfying $\max_{1\le i\le n}{\|\omega_i-x_i-x\|}>0$, inequality \eqref{P7-1} holds true.
Moreover,
inequality \eqref{P7-1} can be replaced by its localized version \eqref{P7.2-1}, or by \eqref{C6-2} if $\varphi\in\mathcal{C}^1$.
\end{corollary}

\begin{remark}
The sufficient conditions for $\varphi-$semitransversality and $\varphi-$transversality {in} Theorems~\ref{P7} and \ref{P9} and their corollaries use the same (slope) inequalities \eqref{P7-1}, \eqref{P7.2-1} and \eqref{C6-2}.
Nevertheless, the sufficient conditions in Theorem~\ref{P9} and Corollary~\ref{C11.2} are stronger than the corresponding ones in Theorem~\ref{P7} and Corollary~\ref{C3+}, respectively, as they require the inequalities to be satisfied on a larger set of points.
This is natural as $\varphi-$transversality is a stronger property than $\varphi-$semitransversality.
At the same time, the `$\de$-free' versions in {Corollaries~\ref{C3+} and \ref{C11.2}} are almost identical: the only difference is the additional condition
\sloppy
$$\max_{1\le i\le n}{\|\omega_i-x_i-x\|}\le\max_{1\le i\le n}{\|x_i\|}$$ in {Corollary~\ref{C3+}}.
The sufficient condition in {Corollary~\ref{C11.2}} is still acceptable for characterizing $\varphi-$transversality, but the one in {Corollary~\ref{C3+}} seems a little too strong for $\varphi-$semitransversality.
That is why we prefer not to oversimplify these sufficient conditions.
\end{remark}

\section{Transversality and Regularity}\label{S5}

In this section, we provide quantitative relations between the nonlinear transversality of collections of sets and the corresponding nonlinear regularity properties of set-valued mappings.
Besides, nonlinear extensions of the new \emph{transversality properties of a \SVM\ to a set in the range space} due to Ioffe \cite{Iof17} are {discussed}.

\subsection{Regularity of Set-Valued Mappings}

Our model here is a set-valued mapping $F:X\rightrightarrows Y$ between metric spaces.
We consider its local regularity properties near a given point $(\bx,\by)\in\gph F$.
The nonlinearity in the {definitions} of the properties is determined by a function $\varphi\in\mathcal{C}$.

Regularity of \SVM s have been intensively studied for decades due to their numerous important applications; see monographs \cite{DonRoc14,Iof17,KlaKum02,Mor06.1}.
Nonlinear regularity properties have also been considered by many authors; cf. \cite{FraQui12,GayGeoJea11,Kum09,LiMor12, BorZhu88,Fra87,Iof13,Kru16,Kru16.2,OuyZhaZhu19, ZheZhu16}.
The relations between transversality and regularity properties are well known in the linear case \cite{KruTha15,KruLukTha18,KruLukTha17,Kru06,Kru09,Kru05, Iof00,Iof16,Kru18} as well as in
the Hölder setting \cite{KruTha14}.
Below we briefly discuss more general nonlinear models.

\begin{definition}\label{D3}
The mapping $F$ is
\begin{enumerate}
\item
$\varphi-$semiregular at $(\bar{x},\bar{y})$ if there exists a $\delta>0$ such that
\begin{align*}
d(\bx,F^{-1}(y))\le\varphi(d(y,\by))
\end{align*}
for all $y\in Y$ with $\varphi(d(y,\by))<\delta$;

\item
$\varphi-$subregular at $(\bar{x},\bar{y})$ if there {exist} $\delta_1>0$ and $\delta_2>0$ such that
\begin{align*}
d(x,F^{-1}(\by))\le\varphi(d(\by,F(x)))
\end{align*}
for all $x\in B_{\de_2}(\bx)$ with $\varphi(d(\by,F(x)))<\de_1$;

\item
$\varphi-$regular at $(\bar{x},\bar{y})$ if there {exist} $\delta_1>0$ and $\delta_2>0$ such that
\begin{align}\label{D3-3}
d(x,F^{-1}(y))\le\varphi(d(y,F(x)))
\end{align}
for all $x\in X$ and $y\in Y$ with $d(x,\bx)+d(y,\by)<\delta_2$ and $\varphi(d(y,F(x)))<\de_1$.
\end{enumerate}
\end{definition}

The function $\varphi\in\mathcal{C}$ in the above definition plays the role of a kind of rate or modulus of the respective property.
In the H\"older setting, i.e. when
$\varphi(t):=\alpha^{-1} t^q$ with $\alpha>0$ and $q>0$, we refer to the respective properties in Definition~\ref{D3} as $\al-$semiregularity, $\al-$subregularity and $\al-$regularity of order $q$.
These regularity properties have been studied in \cite{KruTha14,FraQui12,GayGeoJea11,Kum09,Kru16.2,LiMor12}.
It is usually assumed that $q\le1$.
The exact upper bound of all $\al>0$ such that a property holds with some $\delta>0$, or $\delta_1>0$ and $\delta_2>0$, is called the \emph{modulus} of this property.
We use notations s$_{\rm e}$rg$_q[F](\bx,\by)$, srg$_q[F](\bx,\by)$ and rg$_q[F](\bx,\by)$ for the moduli of the respective properties.
If a property does not hold, then by convention the respective modulus equals~0.
With $q=1$ (linear case), the properties are called metric \textit{semiregularity}, \textit{subregularity} and \textit{regularity}, respectively; cf. \cite{Mor06.1,DonRoc14, RocWet98,Iof17, Kru09,CibFabKru19}.
%In \cite{ApeDurStr13,AraMor11,Ude18}, property (i) is referred to as \textit{hemiregularity}.

The following assertion is a direct consequence of Definition~\ref{D3}.

\begin{proposition}%\label{P16}
If $F$ is $\varphi-$regular at $(\bar{x},\bar{y})$ with some $\delta_1>0$ and $\delta_2>0$, then it is $\varphi-$se\-miregular at $(\bx,\by)$ with $\de:=\min\{\de_1,\varphi(\de_2)\}$ and $\varphi-$subregular at $(\bx,\by)$ with {$\delta_1$ and $\delta_2$}.
\end{proposition}

Note the combined inequality $d(x,\bx)+d(y,\by)<\delta_2$ employed in part (iii) of Definition~\ref{D3} instead of the more traditional separate conditions $x\in B_{\delta_2}(\bx)$ and $y\in B_{\delta_2}(\by)$.
This replacement does not affect the property of $\varphi-$regularity itself, but can have an effect on the value of $\de_2$.
Employing this inequality makes the property a direct analogue of the metric characterization of $\varphi-$trans\-versality in {Theorem~\ref{P3.5}} and is convenient for establishing relations between the regularity and transversality properties.
The next proposition provides also an important special case when the point $x$ in \eqref{D3-3} can be fixed: $x=\bx$.

\begin{proposition}\label{P10}
Let $\delta_1>0$ and $\delta_2>0$.
Consider the following conditions:
\begin{enumerate}
\item [\rm(a)]
{inequality \eqref{D3-3} holds}
for all $x\in B_{\delta_2}(\bx)$ and $y\in B_{\delta_2}(\by)$ with $\varphi(d(y,F(x)))<\de_1$;

\item [\rm(b)]
{inequality \eqref{D3-3} holds}
for all $x\in X$ and $y\in Y$ with
$d(x,\bx)+d(y,\by)<\delta_2$ and $\varphi(d(y,F(x)))<\de_1$;
\item [\rm(c)]
$d(\bx,F^{-1}(y))\le\varphi(d(y,F(\bx)))$
for all $y\in B_{\delta_2}(\by)$ with $\varphi(d(y,F(\bx)))<\de_1$.
\end{enumerate}
Then
\begin{enumerate}
\item
{\rm (a) \folgt (b) \folgt (c)}.
Moreover, condition {\rm (b)} implies {\rm (a)} with $\de'_2:=\de_2/2$ in place of~$\de_2$.
\item
If $X$ is a normed space, $Y=X^n$ for some $n\in\N$, $\by=(\bx_1,\ldots,\bx_n)$ and $F:X\rightrightarrows X^n$ is given by
\begin{align}\label{P10-1}
F(x):=(\Omega_1-x)\times\ldots\times (\Omega_n-x),\quad x\in X,
\end{align}
where $\Omega_1,\ldots,\Omega_n\subset X$,
then {\rm (b) \iff (c)}.
\end{enumerate}
\end{proposition}

\begin{proof}
\begin{enumerate}
\item
All the implications are straightforward.
\item
In view of (i), we only need to prove (c) \folgt (b).
Suppose condition (c) is satisfied.
Let $x\in X$, $y=(x_1,\ldots,x_n)\in X^n$, $\|x-\bx\|+\|y-\by\|<\delta_2$ and $\varphi(d(y,F(x)))<\de_1$.
Set $x_i':=x_i+x-\bx$ $(i=1,\ldots,n)$ and $y':=(x_1',\ldots,x_n')$.
Then
\begin{gather*}
\|y'-\by\|\le\|y'-y\|+\|y-\by\|=\|x-\bx\|+\|y-\by\| <\de_2,\\
d(x,F^{-1}(y))=d(x,\cap_{i=1}^n (\Omega_i-x_i))=d(\bx,\cap_{i=1}^n (\Omega_i-x_i'))=d(\bx,F^{-1}(y')),\\
d(y,F(x))=\max_{1\le i\le n}d(x_i,\Omega_i-x)=\max_{1\le i\le n}d(x_i',\Omega_i-\bx)=d(y',F(\bx)).
\end{gather*}
and, thanks to (c), $d(x,F^{-1}(y))\le\varphi(d(y,F(x)))$.
\qed\end{enumerate}
\end{proof}

The set-valued mapping \eqref{P10-1} plays the key role in establishing relations between the regularity and transversality properties.
It was most likely first used by Ioffe in \cite{Iof00}.
Observe that
$F\iv(x_1,\ldots,x_n)=(\Omega_1-x_1)\cap\ldots\cap(\Omega_n-x_n)$ for all $x_1,\ldots,x_n\in X$ and,
if $\bx\in\cap_{i=1}^n\Omega_i$, then $(0,\ldots,0)\in F(\bx)$.

\begin{theorem}\label{theorem13}
Let $\Omega_1,\ldots,\Omega_n$ be subsets of a normed space $X$, $\bx\in\cap_{i=1}^n\Omega_i$, $\varphi\in\mathcal{C}$, and $F$ be defined by \eqref{P10-1}.
\begin{enumerate}
\item
The collection $\{\Omega_1,\ldots,\Omega_n\}$ is $\varphi-$semitransversal at $\bx$ with some $\de>0$ if and only if $F$ is $\varphi-$semi\-regular at $(\bx,(0,\ldots,0))$ with {$\de$}.
\item
The collection $\{\Omega_1,\ldots,\Omega_n\}$ is $\varphi-$subtransversal at $\bx$ with some $\de_1>0$ and $\de_2>0$ if and only if $F$ is $\varphi-$sub\-regular at $(\bx,(0,\ldots,0))$ with {$\de_1$ and $\de_2$}.
\item
If $\{\Omega_1,\ldots,\Omega_n\}$ is $\varphi-$transversal at $\bx$ with some $\de_1>0$ and $\de_2>0$,
then $F$ is $\varphi-$regular at $(\bx,(0,\ldots,0))$ with
any $\de_1'\in ]0,\de_1]$ and $\de_2'>0$ satisfying ${\varphi\iv(\de_1')+\de_2'}\le\de_2$.
\sloppy

Conversely, if $F$ is $\varphi-$regular at $(\bx,(0,\ldots,0))$ with some $\de_1>0$ and $\de_2>0$, then $\{\Omega_1,\ldots,\Omega_n\}$ is $\varphi-$transversal at $\bx$ with any $\de_1'\in]0,\de_1]$ and $\de_2'>0$ satisfying ${\varphi\iv(\de_1')+\de_2'}\le\de_2$.
\end{enumerate}
\end{theorem}

\begin{proof}
(i) and (ii) follow from {Theorem~\ref{T4.1}(i) and (ii)}, respectively, while (iii) is a consequence of Theorem~\ref{P3.5}.
\qed\end{proof}

The next corollary provides $\de-$free versions of the assertions in Theorem \ref{theorem13}.

\begin{corollary}\label{C11}
Let $\Omega_1,\ldots,\Omega_n$ be subsets of a normed space $X$, $\bx\in \cap_{i=1}^n\Omega_i$, $\varphi \in \mathcal{C}$, and $F$ be defined by \eqref{P10-1}.
The collection $\{\Omega_1,\ldots,\Omega_n\}$ is \begin{enumerate}
\item
$\varphi-$semitransversal at $\bx$ if and only if $F$ is $\varphi-$semiregular at $(\bx,(0,\ldots,0))$;
\item
$\varphi-$subtransversal at $\bx$ if and only if $F$ is $\varphi-$subregular at $(\bx,(0,\ldots,0))$;
\item
$\varphi-$transversal at $\bx$ if and only if $F$ is $\varphi-$regular at $(\bx,(0,\ldots,0))$.
\sloppy
\end{enumerate}
\end{corollary}

\begin{remark}
\begin{enumerate}
\item
In the H\"older setting Corollary~\ref{C11} reduces to \cite[Proposition~9]{KruTha14}.
\item
Apart from the mapping $F$ defined by \eqref{P10-1},	
in the case of two sets other set-valued mappings can be used to ensure similar equivalences between the transversality and regularity properties; see \cite{Iof17}.
\end{enumerate}
\end{remark}

In view of Theorem~\ref{theorem13}, the nonlinear transversality properties of collections of sets can be viewed as particular cases of the corresponding nonlinear regularity properties of \SVM s.
We are going to show that the two popular models are in a sense equivalent.

Given an arbitrary set-valued mapping $F:X\rightrightarrows Y$ between metric spaces and a point $(\bx,\by)\in\gph F$, we can consider the two sets:
\begin{align}\label{15*}
\Omega_1:=\gph F,\quad\Omega_2:=X\times\{\by\}
\end{align}
in the product space $X\times Y$.
Note that $(\bx,\by)\in\Omega_1\cap\Omega_2=F^{-1}(\by) \times \{\by\}$.
To establish the relationship between the two sets of properties, we have to assume in the next two theorems that $X$ and $Y$ are normed vector spaces.

\begin{theorem}\label{theo4}
Let $X$ and $Y$ be normed spaces, $F:X\rightrightarrows Y$, $(\bx,\by)\in\gph F$, and $\varphi\in\mathcal{C}$.
Let $\Omega_1$ and $\Omega_2$ be defined by \eqref{15*}, and {$\psi(t):=\varphi(2t)+t$ for all $t\ge 0$}.
\begin{enumerate}
\item
If $F$ is $\varphi-$semiregular at $(\bx,\by)$ with some $\delta>0$, then $\{\Omega_1,\Omega_2\}$ is $\psi-$semi\-transversal at $(\bx,\by)$ with $\delta':=\de+\varphi^{-1}(\de)/2$.
\item
If $F$ is $\varphi-$subregular at $(\bx,\by)$ with some $\delta_1>0$ and $\delta_2>0$, then $\{\Omega_1,\Omega_2\}$ is $\psi-$sub\-transversal at $(\bx,\by)$ with any $\delta'_1>0$ and $\delta'_2>0$ such that ${\varphi(2\psi\iv(\de'_1))\le\de_1}$ and
${\psi\iv(\de_1')+\de_2'}\le\delta_2$.

\item
If $F$ is $\varphi-$regular at $(\bx,\by)$ with some $\delta_1>0$ and $\delta_2>0$, then $\{\Omega_1,\Omega_2\}$ is $\psi-$trans\-versal at $(\bx,\by)$ with any $\delta'_1>0$ and $\delta'_2>0$ such that
$\varphi(2\psi\iv(\de'_1))\le\de_1$ and
${\psi\iv(\de_1')+\de_2'}\le\delta_2/2$.
\end{enumerate}
\end{theorem}

\begin{proof}
{Observe} that {$\psi\in\mathcal{C}$,} $\varphi(2\psi\iv(t))+\psi\iv(t)=t$ and $\psi(\varphi^{-1}(t)/2)=t+\varphi^{-1}(t)/2$ for all $t\ge0$.
\begin{enumerate}
\item
Let $F$ be $\varphi-$semiregular at $(\bx,\by)$ with some $\delta>0$.
Set $\delta':=\de+\varphi^{-1}(\de)/2=\psi(\varphi^{-1}(\de)/2)$.
Let $\rho\in]0,\delta'[$ and $(u_1,v_1),(u_2,v_2)\in \psi^{-1}(\rho)\B$.
Set $y':=\by+v_1-v_2$.
Observe that
\begin{align*}
(\Omega_1-(u_1,v_1))\cap(\Omega_2-(u_2,v_2))
&=(\gph F-(u_1,v_1))\cap(X\times\{\by-v_2\})\\
&=\left(F^{-1}(y')-u_1\right)\times\{\by-v_2\}.
\end{align*}
We have
$\|y'-\by\|=\|v_1-v_2\|\le\|v_1\|+\|v_2\| <2\psi^{-1}(\rho)$,
and consequently, ${\varphi(\|y'-\by\|)<\varphi(2\psi^{-1}(\rho)) <\varphi(2\psi^{-1}(\de'))=\de}$.
By Definition~\ref{D3}(i),
\sloppy
\begin{align*}
d(\bx,F^{-1}(y')-u_1)
&\le d(\bx,F^{-1}(y'))+\|u_1\|\\
&\le\varphi(\|y'-\by\|)+\|u_1\| <\varphi(2\psi^{-1}(\rho))+\psi^{-1}(\rho)=\rho,
\end{align*}
and consequently,
\begin{align*}
d\left((\bx,\by),(\Omega_1-(u_1,v_1))\cap (\Omega_2-(u_2,v_2))\right) &\le\max\{d(\bx,F^{-1}(y')-u_1),\|v_2\|\}\\
&<\max\{\rho,\psi^{-1}(\rho)\}=\rho;
\end{align*}
hence,
\begin{align} \label{61}
\left(\Omega_1-(u_1,v_1)\right)\cap \left(\Omega_2-(u_2,v_2)\right)\cap B_{\rho}(\bx,\by)\neq \emptyset.
\end{align}
{By Definition~\ref{D1}(i),} $\{\Omega_1,\Omega_2\}$ is $\varphi-$semitransversal at $(\bx,\by)$ with $\delta'$.

\item
Let $F$ be $\varphi-$subregular at $(\bx,\by)$ with some $\delta_1>0$ and $\delta_2>0$.
Choose numbers $\delta'_1>0$ and $\delta'_2>0$ such that
$\varphi(2\psi\iv(\de'_1))\le\de_1$ and
${\psi\iv(\de_1')+\de_2'}\le\delta_2$.
Let $\rho\in]0,\delta'_1[$ and $(x,y)\in B_{\de'_2}(\bx,\by)$ with ${\psi(\max\{d((x,y),\Omega_1),d((x,y),\Omega_2)\})<\rho}$, i.e. $\|y-\by\|<\psi\iv(\rho)$ and there exists a point $(x_1,y_1)\in\gph F$ such that ${\|(x,y)-(x_1,y_1)\|<\psi\iv(\rho)}$.
Then
\begin{gather*}
\|x_1-\bx\|\le{\|x_1-x\|+\|x-\bx\|}<{\psi^{-1}(\de'_1)+\delta'_2}\le\delta_2,\\
d(\by,F(x_1))\le\|y_1-\by\|\le\|y-\by\|+\|y_1-y\|<2\psi\iv(\rho),
\end{gather*}
and consequently,
$\varphi(d(\by,F(x_1)))<\varphi(2\psi\iv(\rho)) <\varphi(2\psi\iv(\de'_1))\le\de_1$.
Choose a positive $\eps<2\psi^{-1}(\rho)-d(\by,F(x_1))$.
By Definition~\ref{D3}(ii), there exists an ${x'\in F^{-1}(\by)}$ such that $\|x'-x_1\|<\varphi(d(\by,F(x_1))+\eps) <\varphi(2\psi^{-1}(\rho))$.
{Hence}, $(x',\by)\in\Omega_1\cap\Omega_2$ and
\begin{gather*}
\|x-x'\|\le\|x_1-x'\|+\|x-x_1\| <\varphi(2\psi^{-1}(\rho))+\psi^{-1}(\rho)=\rho,\\
\|y-\by\|<\psi\iv(\rho)<\rho.
\end{gather*}
Thus, $\Omega_1\cap\Omega_2\cap B_\rho(x,y)\ne\es$.
By Definition~\ref{D1}(ii), $\{\Omega_1,\Omega_2\}$ is $\psi-$subtransversal at $(\bx,\by)$ with $\delta'_1$ and $\delta'_2$.
\item
Let $F$ be $\varphi-$regular at $(\bx,\by)$ with some $\delta_1>0$ and $\delta_2>0$.
Choose numbers $\delta'_1>0$ and $\delta'_2>0$ such that
$\varphi(2\psi\iv(\de'_1))\le\de_1$ and
${\psi^{-1}(\de'_1)+\delta'_2}\le\delta_2/2$.
Let ${\rho\in]0,\delta'_1[}$, $(x_1,y_1)\in\gph F\cap B_{\delta'_2}(\bx,\by)$, $x_2\in B_{\delta'_2}(\bx)$ and $(u_1,v_1),(u_2,v_2)\in\psi^{-1}(\rho)\B$.
Set ${y':=y_1+v_1-v_2}$.
Then
\begin{gather*}
\|x_1-\bx\|+\|y'-\by\| \le{\|v_1\|+\|v_2\|+\|x_1-\bx\|+\|y_1-\by\|}< {2\psi^{-1}(\de'_1)+2\delta'_2}\le\delta_2,\\
\varphi(d(y',F(x_1)))\le\varphi(\|y'-y_1\|) \le\varphi(\|v_1\|+\|v_2\|) <\varphi(2\psi^{-1}(\de'_1))\le\de_1.
\end{gather*}
Choose a positive $\eps<2\left(\psi^{-1}(\rho)-\max\{\|v_1\|,\|v_2\|\}\right)$.
By Definition~\ref{D3}(iii), there exists an $x'\in F^{-1}(y')$ such that
$$\|x_1-x'\|<\varphi(\|y'-y_1\|+\eps)\le \varphi(2\max\{\|v_1\|,\|v_2\|\}+\eps) <\varphi(2\psi^{-1}(\rho)).$$
Denote $\hat x:=x'-x_1-u_1$ and $\hat y:=y'-y_1-v_1$.
Thus, $(x',y')\in\Omega_1$ and $(\hat x,\hat y)\in\Omega_1-(x_1,y_1)-(u_1,v_1)$.
At the same time, $\hat y=-v_2$ and $(\hat x,\hat y)\in\Omega_2-(x_2,\by)-(u_2,v_2)$.
Moreover,
\begin{gather*}
\|\hat x\|\le\|x'-x_1\|+\|u_1\| <\varphi(2\psi^{-1}(\rho))+\psi^{-1}(\rho)=\rho,\\
\|\hat y\|=\|v_2\|<\psi^{-1}(\rho)<\rho;
\end{gather*}
hence $(x',y')\in\rho\B$.
By Definition~\ref{D1}(iii), $\{\Omega_1,\Omega_2\}$ is $\psi-$transversal at $(\bx,\by)$ with $\delta'_1$ and $\delta'_2$.
\qed\end{enumerate}
\end{proof}

\begin{theorem}\label{T7.3}
Let $X$ and $Y$ be normed spaces, $F:X\rightrightarrows Y$, $(\bx,\by)\in\gph F$, and $\varphi\in\mathcal{C}$.
Let $\Omega_1$ and $\Omega_2$ be defined by \eqref{15*}, and {$\psi(t):=\varphi(t/2)$ for all $t\ge 0$}.
\begin{enumerate}
\item
If $\{\Omega_1,\Omega_2\}$ is $\varphi-$semitransversal at $(\bx,\by)$ with some $\delta>0$, then $F$ is ${\psi}-$semi\-regular at $(\bx,\by)$ with {$\delta$}.	

\item
If $\{\Omega_1,\Omega_2\}$ is $\varphi-$subtransversal at $(\bx,\by)$ with some $\delta_1>0$ and $\delta_2>0$, then $F$ is ${\psi}-$subregular at $(\bx,\by)$ with $\delta'_1:=\min\{\de_1,\psi(2\delta_2)\}$ and $\de_2$.
\item
If $\{\Omega_1,\Omega_2\}$ is $\varphi-$transversal at $(\bx,\by)$ with some $\delta_1>0$ and $\delta_2>0$, then $F$ is $\psi-$re\-gular at $(\bx,\by)$ with any $\delta'_1\in]0,\delta_1]$ and $\delta'_2>0$ such that
$\psi\iv(\de'_1)+\delta'_2\le\delta_2$.
\end{enumerate}
\end{theorem}

\begin{proof}
{Observe that $\psi\in\mathcal{C}$.}
\begin{enumerate}
\item
Let $\{\Omega_1,\Omega_2\}$ be $\varphi-$semitransversal at $(\bx,\by)$ with some $\delta>0$.
{By Definition~\ref{D1}(i), condition} \eqref{61} is satisfied for all $\rho\in]0,\delta[$ and $(u_1,v_1), (u_2,v_2)\in \varphi^{-1}(\rho)\B$.
Let $y\in Y$ with {$\rho_0:=\psi(\|y-\by\|)<\delta$.
%Thus, $0\le\rho_0<\de$.
Choose} a $\rho\in]\rho_0,\delta[$ and observe that
$$\varphi\left(\norm{\left(0,\frac{y-\by}{2}\right)}\right) =\varphi\left(\frac{\|y-\by\|}{2}\right)
{=\psi}
\left(\|y-\by\|\right)<\rho.$$
In view of \eqref{61}, we can find $(x_1,y_1) \in \gph F$ and $x_2\in X$ such that
\begin{align*}
(x_1,y_1)-\left(0,\frac{y-\by}{2}\right) =(x_2,\by)-\left(0,\frac{\by-y}{2}\right)\in B_{\rho}(\bx,\by).
\end{align*}
Hence, $y_1=\by+2\frac{y-\by}{2}=y$, $x_1 \in F^{-1}(y)$, $\|x_1-\bx\|<\rho$, and consequently, $d(\bx,F^{-1}(y))<\rho$.
Letting $\rho\downarrow\rho_0$, we obtain $d(\bx,F^{-1}(y)) \le
{\psi}(\|y-\by\|)$.
{By Definition~\ref{D3}(i)}, $F$ is $\psi-$se\-miregular at $(\bx,\by)$ with $\delta$.

\item
Let $\{\Omega_1,\Omega_2\}$ be $\varphi-$subtransversal at $(\bx,\by)$ with some $\delta_1>0$ and ${\delta_2>0}$.
By Definition~\ref{D1}(ii), $\gph F\cap(X\times\{\by\})\cap B_\rho(x,y)\ne\es$ for all $\rho\in]0,\de_1[$ and $(x,y)\in B_{\de_2}(\bx,\by)$ with $\varphi(d((x,y),\gph F))<\rho$ and $\varphi(\|y-\by\|)<\rho$.
Set ${\delta'_1:=\min\{\de_1,\psi(2\delta_2)\}}$.
Let $x\in B_{\de_2}(\bx)$ and $\psi(d(\by,F(x)))<\de'_1$.
Choose a $y\in F(x)$ such that
$\rho_0:=\psi(\|\by-y\|)<\delta'_1$, and a
$\rho\in]\rho_0,\delta'_1[$.
Set $\hat y:=\frac{y+\by}{2}$.
Observe that
\begin{gather*}
\|\hat y-y\|=\|\hat y-\by\|=\frac{\left\|\by-y\right\|}{2} =\frac{{\psi}\iv(\rho_0)}{2} <\frac{{\psi}\iv(\rho)}{2}{=}\varphi\iv(\rho),\\
\|\hat y-\by\|<\frac{{\psi}\iv(\rho)}{2} <\frac{{\psi}\iv(\de'_1)}{2}\le\de_2.
\end{gather*}
Thus, $\rho\in]0,\de_1[$, $(x,\hat y)\in B_{\de_2}(\bx,\by)$, $\varphi(d((x,\hat y),\gph F))\le\varphi(\|\hat y-y\|)<\rho$ and $\varphi(\|\hat y-\by\|)<\rho$.
Hence, $\gph F\cap(X\times\{\by\})\cap B_\rho(x,\hat y)\ne\es$, and consequently, $d(x,F^{-1}(\by)) < \rho$.
Letting $\rho\downarrow\rho_0$, we obtain $d(x,F^{-1}(\by)) \le\psi(\|\by-y\|)$.
Taking the infimum in the \RHS\ of this inequality over $y\in F(x)$, we conclude that
$F$ is $\psi-$subregular at $(\bx,\by)$ with $\delta'_1$ and $\delta_2$ {in view of Definition~\ref{D3}(ii)}.

\item
Let $\{\Omega_1,\Omega_2\}$ be $\varphi-$transversal at $(\bx,\by)$ with some $\delta_1>0$ and $\delta_2>0$, i.e. for all $\rho\in]0,\delta_1[$, $(x',y')\in\gph F\cap B_{\de_2}(\bx,\by)$, $u_1\in X$ and $v_1,v_2\in Y$ with $\varphi(\max\{\|u_1\|,\|v_1\|,\|v_2\|\})<\rho$, it holds
\begin{align*}
\left(\gph F-(x',y')-(u_1,v_1)\right)\cap \left(X\times\{-v_2\}\right)\cap(\rho\B)\neq \emptyset,
\end{align*}
or equivalently,
$d\left(x'+u_1,F\iv(y'+v_1-v_2)\right)<\rho$.
In other words, ${d\left(x,F\iv(y)\right)<\rho}$
for all $\rho\in]0,\delta_1[$, $(x',y')\in\gph F\cap B_{\de_2}(\bx,\by)$, $x\in X$ and $y\in Y$ with $\| x-x'\|<\varphi\iv(\rho)$ and $\|y-y'\|<2\varphi\iv(\rho)$.
Choose numbers $\delta'_1\in]0,\delta_1]$ and $\delta'_2>0$ such that $\psi\iv(\de'_1)+\delta'_2\le\delta_2$.
Let $x\in X$ and $y\in Y$ with $\|x-\bx\|+\|y-\by\|<\delta'_2$ and $\psi(d(y,F(x)))<\de'_1$.
Choose a ${y'\in F(x)}$ such that $\rho_0:=\psi(\|y-y'\|)<\de'_1$ and a $\rho\in]\rho_0,\delta'_1[$.
Then $\rho\in]0,\delta_1[$, $(x,y')\in\gph F$,
$\|x-\bx\|<\delta'_2<\delta_2$, $\|y'-\by\|\le\|y'-y\|+\|y-\by\| <\psi\iv(\de'_1)+\delta'_2\le\delta_2$ and $\|y-y'\|<\psi\iv(\rho)=2\varphi\iv(\rho)$.
Hence, $d\left(x,F\iv(y)\right)<\rho$.
Letting $\rho\downarrow\rho_0$, we obtain $d(x,F^{-1}(y)) \le\psi(\|y-y'\|)$.
Taking the infimum in the \RHS\ of this inequality over $y'\in F(x)$, we conclude that $F$ is $\psi-$regular at $(\bx,\by)$ with $\delta'_1$ and~$\delta'_2$ {in view of Definition~\ref{D3}(iii)}.
\qed\end{enumerate}
\end{proof}

The next corollary of Theorems~\ref{theo4} and \ref{T7.3} provides qualitative relations between the regularity and transversality properties.

\begin{corollary}%\label{C12}
Let $X$ and $Y$ be normed spaces, $F:X\rightrightarrows Y$, $(\bx,\by)\in\gph F$, and $\varphi\in\mathcal{C}$.
Let $\Omega_1$ and $\Omega_2$ be defined by \eqref{15*}, {$\psi_1(t):=\varphi(2t)+t$ and $\psi_2(t):=\varphi(t/2)$ for all $t\ge 0$.}
\begin{enumerate}
\item
If $F$ is $\varphi-$(semi-/sub-)regular at $(\bx,\by)$, then $\{\Omega_1,\Omega_2\}$ is $\psi_1-$(semi-/sub-)trans\-versal at $(\bx,\by)$.

\item
If $\{\Omega_1,\Omega_2\}$ is $\varphi-$(semi-/sub-)transversal at $(\bx,\by)$, then $F$ is ${\psi_2}-$(semi-/sub-)re\-gular at $(\bx,\by)$.
\end{enumerate}
\end{corollary}

The next statement addresses the H\"older setting.
It is a consequence of Theorems~\ref{theo4} and \ref{T7.3} with $\varphi(t):=\al\iv t^q$ for some $\al>0$, $q>0$ and all $t\ge 0$.

\begin{corollary}\label{C7.3}
Let $X$ and $Y$ be normed spaces, $F:X\rightrightarrows Y$, $(\bx,\by)\in\gph F$, ${\al>0}$ and $q>0$.
Let $\Omega_1$ and $\Omega_2$ be defined by \eqref{15*}, $\al_1:=2^{-q}\al$, $\al_2:=2^{q}\al$, and ${\psi(t):=\al_1\iv t^q+t}$ {for all $t\ge 0$}.
\begin{enumerate}
\item
If $F$ is $\al-$semiregular of order $q$ at $(\bx,\by)$ with some $\delta>0$, then $\{\Omega_1,\Omega_2\}$ is $\psi-$se\-mitransversal at $(\bx,\by)$ with $\delta':=\de+(\al\de)^{\frac{1}{q}}/2$.

If $\{\Omega_1,\Omega_2\}$ is $\al-$semitransversal of order $q$ at $(\bx,\by)$ with some $\delta>0$, then $F$ is ${\al_2}-$semiregular of order $q$ at $(\bx,\by)$ with ${\delta}$.

\item
Let $q\le1$.
If $F$ is $\al-$subregular of order $q$ at $(\bx,\by)$ with some $\delta_1>0$ and $\delta_2>0$, then $\{\Omega_1,\Omega_2\}$ is $\psi-$sub\-transversal at $(\bx,\by)$ with any $\delta'_1>0$ and $\delta'_2>0$ such that $(2\psi\iv(\de'_1))^q\le\al\de_1$ and
${\psi^{-1}(\de'_1)+\delta'_2}\le\delta_2$.

If $\{\Omega_1,\Omega_2\}$ is $\al-$subtransversal of order $q$ at $(\bx,\by)$ with some $\delta_1>0$ and $\delta_2>0$,
then $F$ is ${\al_2}-$subregular of order $q$ at $(\bx,\by)$ with $\delta'_1:=\min\{\de_1,\al\iv\delta_2^q\}$ and $\de_2$.	
	
\item
Let $q\le1$.
If $F$ is $\al-$regular of order $q$ at $(\bx,\by)$ with some $\delta_1>0$ and $\delta_2>0$, then $\{\Omega_1,\Omega_2\}$ is $\psi-$transversal at $(\bx,\by)$ with any $\delta'_1>0$ and $\delta'_2>0$ such that
$(2\psi\iv(\de'_1))^q\le\al\de_1$ and
${\psi^{-1}(\de'_1)+\delta'_2}\le\delta_2/2$.

If $\{\Omega_1,\Omega_2\}$ is $\al-$transversal of order $q$ at $(\bx,\by)$ with some $\delta_1>0$ and $\delta_2>0$,
then $F$ is $\al_2-$re\-gular of order $q$
at $(\bx,\by)$ with
any $\delta'_1\in]0,\delta_1]$ and $\delta'_2>0$ such that $2(\al\de'_1)^{\frac{1}{q}}+\delta'_2\le\delta_2$.
\end{enumerate}
\end{corollary}

In view of Corollary~\ref{C7.3}, H\"older transversality properties of $\{\Omega_1,\Omega_2\}$ imply the corresponding H\"older regularity properties of $F$, while H\"older regularity properties of $F$ {imply} certain `H\"older-type' transversality properties of $\{\Omega_1,\Omega_2\}$ determined by the function $\psi$.
Utilizing Proposition~\ref{P2.3}, they can be approximated by proper H\"older (or even linear) transversality properties.

\begin{corollary}\label{C7.4}
Let $X$ and $Y$ be normed spaces, $F:X\rightrightarrows Y$, $(\bx,\by)\in\gph F$, $\al>0$ and $q>0$.
Let $\Omega_1$ and $\Omega_2$ be defined by \eqref{15*} and $\al_1:=2^{-q}\al$.
If
$F$ is $\al-$(semi-/sub-) transversal at $(\bx,\by)$, then $\{\Omega_1,\Omega_2\}$ is $\al'-$(semi-/sub-)transversal of order $q'$ at $\bx$, where:
\begin{enumerate}
\item
if $q<1$, then $q'=q$ and $\al'$ is any number in $]0,\al_1[$;
\item
if $q=1$, then $q'=1$ and $\al':=(1+\al_1\iv)\iv$;
\item
if $q>1$, then $q'=1$ and $\al'$ is any number in $]0,1[$.
\end{enumerate}
\end{corollary}

Thanks to Corollaries \ref{C7.3} and \ref{C7.4}, in the case $q\in]0,1]$ we have full equivalence between the two sets of properties.
The following corollary recaptures \cite[Proposition~10]{KruTha14}.

\begin{corollary}
Let $X$ and $Y$ be normed spaces, $F:X\rightrightarrows Y$, $(\bx,\by)\in\gph F$, and $q\in]0,1]$.
Let $\Omega_1$ and $\Omega_2$ be defined by \eqref{15*}.
\begin{enumerate}
\item
$\{\Omega_1,\Omega_2\}$ is $\al-$semitransversal of order $q$ at $(\bx,\by)$ if and only if $F$ is semiregular of order $q$ at $(\bx,\by)$.
Moreover,
\begin{align*}
\dfrac{{\rm s_{e}rg}_q[F](\bx,\by)}{{\rm s_{e}rg}_q[F](\bx,\by)+2^q}
\le{\rm s_{e}tr}_q[\Omega_1,\Omega_2](\bx)\le\dfrac{{\rm s_{e}rg}_q[F](\bx,\by)}{2^q}.
\end{align*}

\item
$\{\Omega_1,\Omega_2\}$ is $\al-$subtransversal of order $q$ at $(\bx,\by)$ if and only if $F$ is subregular of order $q$ at $(\bx,\by)$.
Moreover,
\begin{align*}
\dfrac{{\rm srg}_q[F](\bx,\by)}{{\rm srg}_q[F](\bx,\by)+2^q}
\le{\rm str}_q[\Omega_1,\Omega_2](\bx)\le\dfrac{{\rm srg}_q[F](\bx,\by)}{2^q}.
\end{align*}

\item
$\{\Omega_1,\Omega_2\}$ is $\al-$transversal of order $q$ at $(\bx,\by)$ if and only if $F$ is regular of order $q$ at $(\bx,\by)$.
Moreover,
\begin{align*}
\dfrac{{\rm rg}_q[F](\bx,\by)}{{\rm rg}_q[F](\bx,\by)+2^q}
\le {\rm tr}_q[\Omega_1,\Omega_2](\bx)\le\dfrac{{\rm rg}_q[F](\bx,\by)}{2^q}.
\end{align*}
\end{enumerate}
\end{corollary}

\subsection{Transversality of a Mapping to a Set in the Range Space}\label{S6.2}

Finally, we briefly discuss {metric characterizations of} nonlinear extensions of the new \emph{transversality properties of a \SVM\ to a set in the range space} due to {Ioffe \cite{Iof16,Iof17}}.
{For geometric and subdifferential/normal cone characterizations of the properties, we refer the reader to \cite{CuoKru4,CuoKru5,CuoKru6}.}
In the rest of {this section}, $F:X\rightrightarrows Y$ is a \SVM\ between normed spaces, $(\bx,\by)\in\gph F$, $S$ is a subset of $Y$, $\by\in S$, and $\varphi\in\mathcal{C}$.

\begin{definition}\label{D4}
The mapping $F$ is
\begin{enumerate}
\item
$\varphi-$semitransversal to $S$ at $(\bx,\by)$ if $\{\gph F,X\times S\}$ is $\varphi-$semi\-transversal at $(\bx,\by)$, i.e. there exists a $\delta>0$ such that
\begin{align*}
\left(\gph F-(u_1,v_1) \right)
\cap \left(X \times(S-v_2)\right)\cap B_\rho(\bx,\by)\neq\emptyset
\end{align*}
for all $\rho \in ]0,\delta[$, $u_1\in X$, $v_1,v_2\in Y$ with $\varphi\left(\max\{\|u_1\|,\|v_1\|,\|v_2\|\}\right)<\rho$;

\item
$\varphi-$subtransversal to $S$ at $(\bx,\by)$ if $\{\gph F,X\times S\}$ is $\varphi-$subtransversal at $(\bx,\by)$, i.e. there exist $\delta_1>0$ and $\de_2>0$ such that
\begin{align*}
\gph F\cap(X\times S)\cap B_\rho(x,y)\ne\es
\end{align*}
for all $\rho\in]0,\de_1[$ and $(x,y)\in B_{\de_2}(\bx,\by)$ with $\varphi(\max\{d((x,y),\gph F),d(y,S)\})<\rho$;

\item
$\varphi-$transversal to $S$ at $(\bx,\by)$ if $\{\gph F,X\times S\}$ is $\varphi-$transversal at $(\bx,\by)$, i.e. there exist $\delta_1>0$ and $\de_2>0$ such that
\begin{align*}
\left(\gph F-(x_1,y_1)-(u_1,v_1)\right)
\cap\left(X\times(S-y_2-v_2)\right)\cap(\rho\B)\neq \emptyset
\end{align*}
for all $\rho \in ]0,\delta_1[$, $(x_1,y_1)\in\gph F\cap B_{\delta_2}(\bar{x},\by)$, $y_2\in S\cap B_{\delta_2}(\by)$, $u_1\in X$, $v_1,v_2\in Y$ with $\varphi\left(\max\{\|u_1\|,\|v_1\|,\|v_2\|\}\right)<\rho$.
\end{enumerate}
\end{definition}

The two-set model $\{\gph F,X\times S\}$ employed in Definition~\ref{D4} is an extension of the model \eqref{15*}, which corresponds to the case when $S$ is a singleton: $S:=\{\by\}$.

The metric characterizations of the properties in the next two statements are consequences of Theorems~\ref{T4.1} and \ref{P3.5}, respectively.
{Each characterization can be used as an equivalent definition for the respective property.}

\begin{corollary}\label{C5.6}
The mapping $F$ is
\begin{enumerate}
\item
$\varphi-$semitransversal to $S$ at $(\bx,\by)$ with some $\delta>0$ if and only if
\begin{align*}
d\big((\bx,\by),(\gph F-(x_1,y_1))\cap(X\times (S-y_2))\big) \le\varphi\left(\max\{\|x_1\|,\|y_1\|,\|y_2\|\}\right)
\end{align*}
for all $x_1\in X$, $y_1,y_2\in Y$ with $\varphi(\max\{\|x_1\|,\|y_1\|,\|y_2\|\})<\delta$;
\item
is $\varphi-$subtransversal to $S$ at $(\bx,\by)$ with some $\de_1>0$ and $\de_2>0$ if and only if
the following equivalent conditions hold:
\begin{enumerate}
\item
for all $(x,y)\in B_{\delta_2}(\bar{x},\by)$ with $\varphi\big(\max\left\{d((x,y),\gph F),d(y,S)\right\}\big)<\de_1$, it holds
\begin{align}\notag%\label{T8-2}
d\left((x,y),\gph F\cap(X\times S)\right)\le \varphi\left(\max\left\{d((x,y),\gph F),d(y,S)\right\}\right);
\end{align}
\item
for all $(x_1,y_1)\in\gph F\cap B_{\delta_2}(\bar{x},\by)$, $y_2\in S\cap B_{\delta_2}(\by)$ and $u_1\in X$, $v_1,v_2\in Y$ with $\varphi\left(\max\{\|u_1\|,\|v_1\|,\|v_2\|\}\right)<\de_1$ and $x_1+u_1\in B_{\de_2}(\bx)$,
$y_1+v_1=y_2+v_2\in B_{\de_2}(\by)$,
it holds
\begin{multline}\label{T8-3}
d\left((0,0),(\gph F-(x_1,y_1)-(u_1,v_1))\cap(X \times (S-y_2-v_2))\right)\\
\le\varphi\left(\max\{\|u_1\|,\|v_1\|,\|v_2\|\}\right);
\end{multline}
\end{enumerate}
\item
$\varphi-$transversal to $S$ at $(\bx,\by)$ with some $\de_1>0$ and $\de_2>0$ if and only if inequality \eqref{T8-3} holds for all $(x_1,y_1)\in\gph F\cap B_{\delta_2}(\bar{x},\by)$, ${y_2\in S\cap B_{\delta_2}(\by)}$ and $u_1\in X$, $v_1,v_2\in Y$ with $\varphi\left(\max\{\|u_1\|,\|v_1\|,\|v_2\|\}\right)<\de_1$.
\end{enumerate}
\end{corollary}

\begin{corollary}\label{C7.13}
Let $\de_1>0$ and $\de_2>0$.
The following conditions are equivalent:
\begin{enumerate}
\item
for all $(x_1,y_1)\in\gph F\cap B_{\delta_2}(\bar{x},\by)$, $y_2\in S\cap B_{\delta_2}(\by)$ and $u_1\in X$, $v_1,v_2\in Y$ with $x_1+u_1\in B_{\de_2}(\bx)$, $y_1+v_1,y_2+v_2\in B_{\de_2}(\by)$ and $\varphi\left(\max\{\|u_1\|,\|v_1\|,\|v_2\|\}\right)<\de_1$, inequality \eqref{T8-3} holds true;
\item
for all $x_1,y_1,y_2\in\de_2\B$ with
$\varphi(\max\{d((\bx,\by),\gph F-(x_1,y_1)),d(\by,S-y_2)\})<\de_1$, it holds
\begin{multline}
d\left((\bx,\by),(\gph F-(x_1,y_1))\cap(X\times (S-y_2))\right)\\
\le\varphi\left(\max\{d((\bx,\by),\gph F-(x_1,y_1)),
d(\by,S-y_2)\}\right);
\end{multline}
\item
for all $x,x_1\in X$, $y,y_1,y_2\in Y$ such that $x+x_1\in B_{\de_2}(\bx)$, $y+y_1,y+y_2\in B_{\de_2}(\by)$ and $\varphi\left(\max\{d((x,y),\gph F-(x_1,y_1)),
d(y,S-y_2)\}\right)<\de_1$, it holds
\begin{multline*}
d\big((x,y),(\gph F-(x_1,y_1))\cap(X\times(S-y_2))\big)\\
\le\varphi\big(\max\{d((x,y),\gph F-(x_1,y_1)),
d(y,S-y_2)\}\big).
\end{multline*}
\end{enumerate}

Moreover, if $F$ is $\varphi-$transversal to $S$ at $(\bx,\by)$ with some $\delta_1>0$ and $\de_2>0$, then conditions {\rm (i)--(iii)} hold with any $\de_1'\in]0,\de_1]$ and $\de_2'>0$ satisfying ${\varphi\iv(\de_1')+\de_2'}\le \de_2$ in place of $\de_1$ and $\de_2$.

Conversely, if conditions {\rm (i)--(iii)} hold with some $\de_1>0$ and $\de_2>0$, then $F$ is $\varphi-$tra\-ns\-versal to $S$ at $(\bx,\by)$ with any $\de_1'\in ]0,\de_1]$ and $\de_2'>0$ satisfying ${\varphi\iv(\de_1')+\de_2'\le \de_2}$.
\end{corollary}

\begin{remark}
In the linear case, i.e. when $\varphi(t):=\al t$ for some $\al>0$ and all $t\ge0$,
in view of {Corollaries~\ref{C5.6}(ii)(a)} and \ref{C7.13}(iii),
the properties in parts (ii) and (iii) of Definition~\ref{D4} reduce, respectively, to the ones in \cite[Definitions~7.11 and 7.8]{Iof17}.
The property in part (i) is new.
\end{remark}

The set-valued mapping \eqref{P10-1}, crucial for establishing equivalences between trans\-versality properties of collections of sets and the corresponding regularity properties of set-valued mappings, in the setting considered here translates into the mapping $G:X\times Y\rightrightarrows(X\times Y)\times(X\times Y)$ of the following form:
\begin{align}\label{82}
G(x,y):=\big(\gph F-(x,y)\big)\times\big(X\times(S-y)\big),
\quad (x,y)\in X\times Y.
\end{align}
Observe that
$G\iv(x_1,y_1,x_2,y_2)=\big(\gph F-(x_1,y_1)\big) \cap\big(X\times(S-y_2)\big)$ for all $x_1,x_2\in X$, $y_1,y_2\in Y$ and,
if $(\bx,\by)\in\gph F$, $\by\in S$, then $\big((0,0),(0,0)\big)\in G(\bx,\by)$.

The relationships between the nonlinear transversality and regularity properties in the next statement are direct consequences of Theorem~\ref{theorem13}.

\begin{theorem}\label{T9}
Let $G$ be defined by \eqref{82}.
\begin{enumerate}
\item
$F$ is $\varphi-$semitransversal to $S$ at $(\bx,\by)$ with some $\de>0$ if and only if $G$ is $\varphi-$se\-miregular at $\big((\bx,\by),(0,0),(0,0)\big)$ with {$\de$}.
\item
$F$ is $\varphi-$subtransversal to $S$ at $(\bx,\by)$ with some $\de_1>0$ and $\de_2>0$ if and only if $G$ is $\varphi-$sub\-regular at $\big((\bx,\by),(0,0),(0,0)\big)$ with {$\de_1$ and $\de_2$}.
\item
If $F$ is $\varphi-$transversal to $S$ at $(\bx,\by)$ with some $\de_1>0$ and $\de_2>0$, then $G$ is $\varphi-$re\-gular at $\big((\bx,\by),(0,0),(0,0)\big)$ with
any $\de_1'\in ]0,\de_1]$ and $\de_2'>0$ satisfying $\de_2'+\varphi\iv(\de'_1)\le\de_2$.
\smallskip
\sloppy

Conversely, if $G$ is $\varphi-$regular at $\big((\bx,\by), (0,0),(0,0)\big)$ with some $\de_1>0$ and $\de_2>0$, then $F$ is $\varphi-$transversal to $S$ at $(\bx,\by)$ with any $\de_1'\in]0,\de_1]$ and $\de_2'>0$ satisfying $\de_2'+\varphi\iv(\de'_1)\le\de_2$.
\end{enumerate}
\end{theorem}
\begin{remark}
It is easy to see that the \SVM\ \eqref{82} can be replaced in our considerations by the truncated mapping $\mathcal{G}:X\times Y\rightrightarrows X\times Y\times Y$ defined by
\begin{align*}
\mathcal{G}(x,y):=\big(\gph F-(x,y)\big)\times(S-y),
\quad (x,y)\in X\times Y.
\end{align*}
The last mapping admits a simple representation $\mathcal{G}(x,y)=\gph\mathcal{F}-(x,y,y)$, where
the set-valued mapping $\mathcal{F}:X\rightrightarrows Y\times Y$ is defined by
\begin{align*}
\mathcal{F}(x):=F(x)\times S,\quad x\in X.
\end{align*}
It was shown in \cite[Theorems 7.12 and 7.9]{Iof17} that in the linear case the subtransversality and transversality of $F$ to $S$ at $(\bx,\by)$ are equivalent to the metric subregularity and regularity, respectively, of the mapping $(x,y)\mapsto\mathcal{F}(x)-(y,y)$ at $((\bx,\by),0)$.
\end{remark}

\section*{Acknowledgement}

The authors wish to thank the referee and the handling editor for their careful reading of the manuscript and valuable comments and suggestions.


\def\cprime{$'$} \def\cftil#1{\ifmmode\setbox7\hbox{$\accent"5E#1$}\else
  \setbox7\hbox{\accent"5E#1}\penalty 10000\relax\fi\raise 1\ht7
  \hbox{\lower1.15ex\hbox to 1\wd7{\hss\accent"7E\hss}}\penalty 10000
  \hskip-1\wd7\penalty 10000\box7} \def\cprime{$'$} \def\cprime{$'$}
  \def\cprime{$'$} \def\cprime{$'$} \def\cprime{$'$}
  \def\Dbar{\leavevmode\lower.6ex\hbox to 0pt{\hskip-.23ex \accent"16\hss}D}
  \def\cfac#1{\ifmmode\setbox7\hbox{$\accent"5E#1$}\else
  \setbox7\hbox{\accent"5E#1}\penalty 10000\relax\fi\raise 1\ht7
  \hbox{\lower1.15ex\hbox to 1\wd7{\hss\accent"13\hss}}\penalty 10000
  \hskip-1\wd7\penalty 10000\box7} \def\cprime{$'$}
\begin{thebibliography}{10}
\providecommand{\url}[1]{{#1}}
\providecommand{\urlprefix}{URL }
\expandafter\ifx\csname urlstyle\endcsname\relax
  \providecommand{\doi}[1]{DOI~\discretionary{}{}{}#1}\else
  \providecommand{\doi}{DOI~\discretionary{}{}{}\begingroup
  \urlstyle{rm}\Url}\fi

\bibitem{AzeCor14}
Az\'{e}, D., Corvellec, J.N.: Nonlinear local error bounds via a change of
  metric.
\newblock J. Fixed Point Theory Appl. \textbf{16}(1-2), 351--372 (2014).
\newblock \doi{10.1007/s11784-015-0220-9}

\bibitem{AzeCor17}
Az\'{e}, D., Corvellec, J.N.: Nonlinear error bounds via a change of function.
\newblock J. Optim. Theory Appl. \textbf{172}(1), 9--32 (2017).
\newblock \doi{10.1007/s10957-016-1001-3}

\bibitem{BakDeuLi05}
Bakan, A., Deutsch, F., Li, W.: Strong {CHIP}, normality, and linear regularity
  of convex sets.
\newblock Trans. Amer. Math. Soc. \textbf{357}(10), 3831--3863 (2005).
\newblock \doi{10.1090/S0002-9947-05-03945-0}

\bibitem{BauBor96}
Bauschke, H.H., Borwein, J.M.: On projection algorithms for solving convex
  feasibility problems.
\newblock SIAM Rev. \textbf{38}(3), 367--426 (1996).
\newblock \doi{10.1137/S0036144593251710}

\bibitem{BauBorLi99}
Bauschke, H.H., Borwein, J.M., Li, W.: Strong conical hull intersection
  property, bounded linear regularity, {J}ameson's property {$(G)$}, and error
  bounds in convex optimization.
\newblock Math. Program., Ser. A \textbf{86}(1), 135--160 (1999).
\newblock \doi{10.1007/s101070050083}

\bibitem{BolNguPeySut17}
Bolte, J., Nguyen, T.P., Peypouquet, J., Suter, B.W.: From error bounds to the
  complexity of first-order descent methods for convex functions.
\newblock Math. Program., Ser. A \textbf{165}(2), 471--507 (2017).
\newblock \doi{10.1007/s10107-016-1091-6}

\bibitem{BorLiTam17}
Borwein, J.M., Li, G., Tam, M.K.: Convergence rate analysis for averaged fixed
  point iterations in common fixed point problems.
\newblock SIAM J. Optim. \textbf{27}(1), 1--33 (2017).
\newblock \doi{10.1137/15M1045223}

\bibitem{BorLiYao14}
Borwein, J.M., Li, G., Yao, L.: Analysis of the convergence rate for the cyclic
  projection algorithm applied to basic semialgebraic convex sets.
\newblock SIAM J. Optim. \textbf{24}(1), 498--527 (2014).
\newblock \doi{10.1137/130919052}

\bibitem{BorZhu88}
Borwein, J.M., Zhuang, D.M.: Verifiable necessary and sufficient conditions for
  openness and regularity of set-valued and single-valued maps.
\newblock J. Math. Anal. Appl. \textbf{134}(2), 441--459 (1988).
\newblock \doi{10.1016/0022-247X(88)90034-0}

\bibitem{BuiCuoKru}
Bui, H.T., Cuong, N.D., Kruger, A.Y.: Transversality of collections of sets:
  Geometric and metric characterizations.
\newblock Vietnam J. Math.  (2020).
\newblock \doi{10.1007/s10013-020-00388-1}

\bibitem{BuiKru19}
Bui, H.T., Kruger, A.Y.: Extremality, stationarity and generalized separation
  of collections of sets.
\newblock J. Optim. Theory Appl. \textbf{182}(1), 211--264 (2019).
\newblock \doi{10.1007/s10957-018-01458-8}

\bibitem{CibFabKru19}
Cibulka, R., Fabian, M., Kruger, A.Y.: On semiregularity of mappings.
\newblock J. Math. Anal. Appl. \textbf{473}(2), 811--836 (2019).
\newblock \doi{10.1016/j.jmaa.2018.12.071}

\bibitem{CorMot08}
Corvellec, J.N., Motreanu, V.V.: Nonlinear error bounds for lower
  semicontinuous functions on metric spaces.
\newblock Math. Program., Ser. A \textbf{114}(2), 291--319 (2008)

\bibitem{CuoKru5}
Cuong, N.D., Kruger, A.Y.: Dual sufficient characterizations of transversality
  properties.
\newblock Positivity  (2020).
\newblock \doi{10.1007/s11117-019-00734-9}

\bibitem{CuoKru4}
Cuong, N.D., Kruger, A.Y.: Nonlinear transversality of collections of sets:
  Dual space necessary characterizations.
\newblock J. Convex Anal. \textbf{27}(1), 287--308 (2020)

\bibitem{CuoKru6}
Cuong, N.D., Kruger, A.Y.: Primal space necessary characterizations of
  transversality properties.
\newblock Preprint, Optimization Online \textbf{2020-01-7579} (2020)

\bibitem{DegMarTos80}
De~Giorgi, E., Marino, A., Tosques, M.: Evolution problerns in in metric spaces
  and steepest descent curves.
\newblock Atti Accad. Naz. Lincei Rend. Cl. Sci. Fis. Mat. Natur. (8)
  \textbf{68}(3), 180--187 (1980).
\newblock In Italian. English translation: Ennio De Giorgi, Selected Papers,
  Springer, Berlin 2006, 527--533

\bibitem{DonRoc14}
Dontchev, A.L., Rockafellar, R.T.: Implicit Functions and Solution Mappings. A
  View from Variational Analysis, 2 edn.
\newblock Springer Series in Operations Research and Financial Engineering.
  Springer, New York (2014).
\newblock \doi{10.1007/978-1-4939-1037-3}

\bibitem{DruIofLew15}
Drusvyatskiy, D., Ioffe, A.D., Lewis, A.S.: Transversality and alternating
  projections for nonconvex sets.
\newblock Found. Comput. Math. \textbf{15}(6), 1637--1651 (2015).
\newblock \doi{10.1007/s10208-015-9279-3}

\bibitem{DruLiWol17}
Drusvyatskiy, D., Li, G., Wolkowicz, H.: A note on alternating projections for
  ill-posed semidefinite feasibility problems.
\newblock Math. Program., Ser. A \textbf{162}(1-2), 537--548 (2017).
\newblock \doi{10.1007/s10107-016-1048-9}

\bibitem{Fra87}
Frankowska, H.: An open mapping principle for set-valued maps.
\newblock J. Math. Anal. Appl. \textbf{127}(1), 172--180 (1987).
\newblock \doi{10.1016/0022-247X(87)90149-1}

\bibitem{FraQui12}
Frankowska, H., Quincampoix, M.: H\"older metric regularity of set-valued maps.
\newblock Math. Program., Ser. A \textbf{132}(1-2), 333--354 (2012).
\newblock \doi{10.1007/s10107-010-0401-7}

\bibitem{GayGeoJea11}
Gaydu, M., Geoffroy, M.H., Jean-Alexis, C.: Metric subregularity of order {$q$}
  and the solving of inclusions.
\newblock Cent. Eur. J. Math. \textbf{9}(1), 147--161 (2011).
\newblock \doi{10.2478/s11533-010-0087-3}

\bibitem{HesLuk13}
Hesse, R., Luke, D.R.: Nonconvex notions of regularity and convergence of
  fundamental algorithms for feasibility problems.
\newblock SIAM J. Optim. \textbf{23}(4), 2397--2419 (2013).
\newblock \doi{10.1137/120902653}

\bibitem{Iof00}
Ioffe, A.D.: Metric regularity and subdifferential calculus.
\newblock Russian Math. Surveys \textbf{55}, 501--558 (2000).
\newblock \doi{10.1070/rm2000v055n03ABEH000292}

\bibitem{Iof13}
Ioffe, A.D.: Nonlinear regularity models.
\newblock Math. Program. \textbf{139}(1-2), 223--242 (2013).
\newblock \doi{10.1007/s10107-013-0670-z}

\bibitem{Iof16}
Ioffe, A.D.: Metric regularity -- a survey. {P}art {I}. {T}heory.
\newblock J. Aust. Math. Soc. \textbf{101}(2), 188--243 (2016).
\newblock \doi{10.1017/S1446788715000701}

\bibitem{Iof17}
Ioffe, A.D.: Variational Analysis of Regular Mappings. Theory and Applications.
\newblock Springer Monographs in Mathematics. Springer (2017).
\newblock \doi{10.1007/978-3-319-64277-2}

\bibitem{KlaKum02}
Klatte, D., Kummer, B.: Nonsmooth Equations in Optimization. Regularity,
  Calculus, Methods and Applications, \emph{Nonconvex Optimization and its
  Applications}, vol.~60.
\newblock Kluwer Academic Publishers, Dordrecht (2002)

\bibitem{Kru05}
Kruger, A.Y.: Stationarity and regularity of set systems.
\newblock Pac. J. Optim. \textbf{1}(1), 101--126 (2005)

\bibitem{Kru06}
Kruger, A.Y.: About regularity of collections of sets.
\newblock Set-Valued Anal. \textbf{14}(2), 187--206 (2006).
\newblock \doi{10.1007/s11228-006-0014-8}

\bibitem{Kru09}
Kruger, A.Y.: About stationarity and regularity in variational analysis.
\newblock Taiwanese J. Math. \textbf{13}(6A), 1737--1785 (2009).
\newblock \doi{10.11650/twjm/1500405612}

\bibitem{Kru15}
Kruger, A.Y.: Error bounds and metric subregularity.
\newblock Optimization \textbf{64}(1), 49--79 (2015).
\newblock \doi{10.1080/02331934.2014.938074}

\bibitem{Kru16.2}
Kruger, A.Y.: Error bounds and {H}\"older metric subregularity.
\newblock Set-Valued Var. Anal. \textbf{23}(4), 705--736 (2016).
\newblock \doi{10.1007/s11228-015-0330-y}

\bibitem{Kru16}
Kruger, A.Y.: Nonlinear metric subregularity.
\newblock J. Optim. Theory Appl. \textbf{171}(3), 820--855 (2016).
\newblock \doi{10.1007/s10957-015-0807-8}

\bibitem{Kru18}
Kruger, A.Y.: About intrinsic transversality of pairs of sets.
\newblock Set-Valued Var. Anal. \textbf{26}(1), 111--142 (2018).
\newblock \doi{10.1007/s11228-017-0446-3}

\bibitem{KruLop12.1}
Kruger, A.Y., L\'{o}pez, M.A.: Stationarity and regularity of infinite
  collections of sets.
\newblock J. Optim. Theory Appl. \textbf{154}(2), 339--369 (2012).
\newblock \doi{10.1007/s10957-012-0043-4}

\bibitem{KruLop12.2}
Kruger, A.Y., L\'opez, M.A.: Stationarity and regularity of infinite
  collections of sets. {A}pplications to infinitely constrained optimization.
\newblock J. Optim. Theory Appl. \textbf{155}(2), 390--416 (2012).
\newblock \doi{10.1007/s10957-012-0086-6}

\bibitem{KruLukTha17}
Kruger, A.Y., Luke, D.R., Thao, N.H.: About subtransversality of collections of
  sets.
\newblock Set-Valued Var. Anal. \textbf{25}(4), 701--729 (2017).
\newblock \doi{10.1007/s11228-017-0436-5}

\bibitem{KruLukTha18}
Kruger, A.Y., Luke, D.R., Thao, N.H.: Set regularities and feasibility
  problems.
\newblock Math. Program., Ser. B \textbf{168}(1-2), 279--311 (2018).
\newblock \doi{10.1007/s10107-016-1039-x}

\bibitem{KruTha13}
Kruger, A.Y., Thao, N.H.: About uniform regularity of collections of sets.
\newblock Serdica Math. J. \textbf{39}(3-4), 287--312 (2013)

\bibitem{KruTha14}
Kruger, A.Y., Thao, N.H.: About {$[q]$}-regularity properties of collections of
  sets.
\newblock J. Math. Anal. Appl. \textbf{416}(2), 471--496 (2014).
\newblock \doi{10.1016/j.jmaa.2014.02.028}

\bibitem{KruTha15}
Kruger, A.Y., Thao, N.H.: Quantitative characterizations of regularity
  properties of collections of sets.
\newblock J. Optim. Theory Appl. \textbf{164}(1), 41--67 (2015).
\newblock \doi{10.1007/s10957-014-0556-0}

\bibitem{KruTha16}
Kruger, A.Y., Thao, N.H.: Regularity of collections of sets and convergence of
  inexact alternating projections.
\newblock J. Convex Anal. \textbf{23}(3), 823--847 (2016)

\bibitem{Kum09}
Kummer, B.: Inclusions in general spaces: {H}oelder stability, solution schemes
  and {E}keland's principle.
\newblock J. Math. Anal. Appl. \textbf{358}(2), 327--344 (2009).
\newblock \doi{10.1016/j.jmaa.2009.04.060}

\bibitem{LewLukMal09}
Lewis, A.S., Luke, D.R., Malick, J.: Local linear convergence for alternating
  and averaged nonconvex projections.
\newblock Found. Comput. Math. \textbf{9}(4), 485--513 (2009).
\newblock \doi{10.1007/s10208-008-9036-y}

\bibitem{Li13}
Li, G.: Global error bounds for piecewise convex polynomials.
\newblock Math. Program. \textbf{137}(1-2, Ser. A), 37--64 (2013).
\newblock \doi{10.1007/s10107-011-0481-z}

\bibitem{LiMor12}
Li, G., Mordukhovich, B.S.: H\"older metric subregularity with applications to
  proximal point method.
\newblock SIAM J. Optim. \textbf{22}(4), 1655--1684 (2012).
\newblock \doi{10.1137/120864660}

\bibitem{Mor06.1}
Mordukhovich, B.S.: Variational Analysis and Generalized Differentiation. {I}:
  {B}asic {T}heory, \emph{Grundlehren der Mathematischen Wissenschaften
  [Fundamental Principles of Mathematical Sciences]}, vol. 330.
\newblock Springer, Berlin (2006)

\bibitem{NgZan07}
Ng, K.F., Zang, R.: Linear regularity and {$\phi$}-regularity of nonconvex
  sets.
\newblock J. Math. Anal. Appl. \textbf{328}(1), 257--280 (2007).
\newblock \doi{10.1016/j.jmaa.2006.05.028}

\bibitem{NgaThe01}
Ngai, H.V., Th{\'e}ra, M.: Metric inequality, subdifferential calculus and
  applications.
\newblock Set-Valued Anal. \textbf{9}(1-2), 187--216 (2001).
\newblock \doi{10.1023/A:1011291608129}

\bibitem{NgaThe08}
Ngai, H.V., Th{\'e}ra, M.: Error bounds in metric spaces and application to the
  perturbation stability of metric regularity.
\newblock SIAM J. Optim. \textbf{19}(1), 1--20 (2008).
\newblock \doi{10.1137/060675721}

\bibitem{NolRon16}
Noll, D., Rondepierre, A.: On local convergence of the method of alternating
  projections.
\newblock Found. Comput. Math. \textbf{16}(2), 425--455 (2016).
\newblock \doi{10.1007/s10208-015-9253-0}

\bibitem{OuyZhaZhu19}
Ouyang, W., Zhang, B., Zhu, J.: H{\"o}lder metric subregularity for constraint
  systems in {A}splund spaces.
\newblock Positivity \textbf{23}(1), 161--175 (2019).
\newblock \doi{10.1007/s11117-018-0600-7}

\bibitem{Pen13}
Penot, J.P.: Calculus Without Derivatives, \emph{Graduate Texts in
  Mathematics}, vol. 266.
\newblock Springer, New York (2013).
\newblock \doi{10.1007/978-1-4614-4538-8}

\bibitem{RocWet98}
Rockafellar, R.T., Wets, R.J.B.: Variational Analysis.
\newblock Springer, Berlin (1998)

\bibitem{ThaBuiCuoVer20}
Thao, N.H., Bui, T.H., Cuong, N.D., Verhaegen, M.: Some new characterizations
  of intrinsic transversality in {H}ilbert spaces.
\newblock Set-Valued Var. Anal. \textbf{28}(1), 5--39 (2020).
\newblock \doi{10.1007/s11228-020-00531-7}

\bibitem{YaoZhe16}
Yao, J.C., Zheng, X.Y.: Error bound and well-posedness with respect to an
  admissible function.
\newblock Appl. Anal. \textbf{95}(5), 1070--1087 (2016)

\bibitem{ZheNg08}
Zheng, X.Y., Ng, K.F.: Linear regularity for a collection of subsmooth sets in
  {B}anach spaces.
\newblock SIAM J. Optim. \textbf{19}(1), 62--76 (2008).
\newblock \doi{10.1137/060659132}

\bibitem{ZheWeiYao10}
Zheng, X.Y., Wei, Z., Yao, J.C.: Uniform subsmoothness and linear regularity
  for a collection of infinitely many closed sets.
\newblock Nonlinear Anal. \textbf{73}(2), 413--430 (2010).
\newblock \doi{10.1016/j.na.2010.03.032}

\bibitem{ZheZhu16}
Zheng, X.Y., Zhu, J.: Generalized metric subregularity and regularity with
  respect to an admissible function.
\newblock SIAM J. Optim. \textbf{26}(1), 535--563 (2016).
\newblock \doi{10.1137/15M1016345}

\end{thebibliography}
\end{document}